\documentclass{amsart}
\usepackage{amscd,amsmath,xypic,amssymb,combelow,tikz-cd,etoolbox,calligra,mathrsfs,enumitem}

\DeclareMathOperator{\sHom}{\mathscr{H}\text{\kern -3pt {\calligra\large om}}\,}
\DeclareMathOperator{\sRHom}{\mathscr{RH}\text{\kern -3pt {\calligra\large om}}\,}
\DeclareMathOperator{\sQuot}{\mathscr{Q}\text{\kern -3pt {\calligra\large uot}}\,}
\DeclareMathOperator{\sZ}{\mathscr{Z}\,}

\makeatletter
\patchcmd{\@settitle}{\uppercasenonmath\@title}{}{}{}
\makeatother

\xyoption{all}
\CompileMatrices

\newcommand{\nc}{\newcommand}

\makeatletter
\@addtoreset{equation}{section}
\makeatother

\newtheorem{theorem}[subsection]{Theorem}
\newtheorem{proposition}[subsection]{Proposition}
\newtheorem{lemma}[subsection]{Lemma}
\newtheorem{corollary}[subsection]{Corollary}

\newtheorem{definition}[subsection]{Definition}

\newtheorem{claim}[subsection]{Claim}

\newtheorem{example}[subsection]{Example}
\newtheorem{remark}[subsection]{Remark}

\nc{\fa}{{\mathfrak{a}}}
\nc{\fb}{{\mathfrak{b}}}
\nc{\fg}{{\mathfrak{g}}}
\nc{\fh}{{\mathfrak{h}}}
\nc{\fj}{{\mathfrak{j}}}
\nc{\fn}{{\mathfrak{n}}}
\nc{\fm}{{\mathfrak{m}}}
\nc{\fu}{{\mathfrak{u}}}
\nc{\fp}{{\mathfrak{p}}}
\nc{\fr}{{\mathfrak{r}}}
\nc{\ft}{{\mathfrak{t}}}
\nc{\fsl}{{\mathfrak{sl}}}
\nc{\fgl}{{\mathfrak{gl}}}
\nc{\hsl}{{\widehat{\mathfrak{sl}}}}
\nc{\hgl}{{\widehat{\mathfrak{gl}}}}
\nc{\hg}{{\widehat{\mathfrak{g}}}}
\nc{\chg}{{\widehat{\mathfrak{g}}}{}^\vee}
\nc{\hn}{{\widehat{\mathfrak{n}}}}
\nc{\chn}{{\widehat{\mathfrak{n}}}{}^\vee}

\nc{\Mod}{{\textrm{Mod}}}

\nc{\wGL}{{\widehat{GL}^+}}

\nc{\BA}{{\mathbb{A}}}
\nc{\BC}{{\mathbb{C}}}
\nc{\BM}{{\mathbb{M}}}
\nc{\BN}{{\mathbb{N}}}
\nc{\BF}{{\mathbb{F}}}
\nc{\BH}{{\mathbb{H}}}
\nc{\BP}{{\mathbb{P}}}
\nc{\BQ}{{\mathbb{Q}}}
\nc{\BR}{{\mathbb{R}}}
\nc{\BZ}{{\mathbb{Z}}}

\nc{\ff}{{\mathbb{F}}}
\nc{\kk}{{\mathbb{K}}}
\nc{\kko}{{\mathbb{K}}}

\nc{\coh}{{\text{Coh}}}

\nc{\CA}{{\mathcal{A}}}
\nc{\CC}{{\mathcal{C}}}
\nc{\CB}{{\mathcal{B}}}
\nc{\DD}{{\mathcal{D}}}
\nc{\CE}{{\mathcal{E}}}
\nc{\CF}{{\mathcal{F}}}
\nc{\tCF}{{\widetilde{\CF}}}
\nc{\tCT}{{\widetilde{\CT}}}
\nc{\oCF}{{\overline{\CF}}}
\nc{\CG}{{\mathcal{G}}}
\nc{\CL}{{\mathcal{L}}}
\nc{\CK}{{\mathcal{K}}}
\nc{\CI}{{\mathcal{I}}}
\nc{\CM}{{\mathcal{M}}}
\nc{\CH}{{\mathcal{H}}}
\nc{\CN}{{\mathcal{N}}}
\nc{\CO}{{\mathcal{O}}}
\nc{\CP}{{\mathcal{P}}}
\nc{\CR}{{\mathcal{R}}}
\nc{\CQ}{{\mathcal{Q}}}
\nc{\CS}{{\mathcal{S}}}
\nc{\CT}{{\mathcal{T}}}
\nc{\tCU}{{\widetilde{\CU}}}
\nc{\CU}{{\mathcal{U}}}
\nc{\CV}{{\mathcal{V}}}
\nc{\CW}{{\mathcal{W}}}

\nc{\tpsi}{{\widetilde{\Psi}}}
\nc{\wpi}{{\widetilde{\pi}}}
\nc{\Ker}{{\text{Ker }}}

\nc{\CX}{{\mathcal{X}}}
\nc{\tCX}{{\widetilde{\mathcal{X}}}}
\nc{\CY}{{\mathcal{Y}}}
\nc{\tCY}{{\widetilde{\mathcal{Y}}}}

\nc{\tN}{{\widetilde{\CN}}}
\nc{\pN}{{\BP\widetilde{\CN}}}

\nc{\tT}{{T}}

\nc{\fC}{{\mathfrak{C}}}
\nc{\fZ}{{\mathfrak{Z}}}
\nc{\fU}{{\mathfrak{U}}}
\nc{\fV}{{\mathfrak{V}}}
\nc{\fS}{{\mathfrak{S}}}

\nc{\od}{{\overline{d}}}
\nc{\rg}{{\textrm{R}\Gamma}}
\nc{\erg}{{\emph{R}\Gamma}}
\nc{\id}{{\textrm{Id}}}
\nc{\rhom}{{\textrm{RHom}}}

\def\Ext{\textrm{Ext}}
\def\Hom{\textrm{Hom}}
\def\RHom{\textrm{RHom}}
\def\e{\varepsilon}

\def\tab{\text{} \\}

\def\pt{\textrm{pt}}
\def\and{\textrm{ }\&\textrm{ }}

\def\sym{\textrm{Sym}}
\def\esym{\emph{\sym}}

\def\tCF{\widetilde{\CF}}

\def\Tan{\text{Tan}}

\def\km{K_\CM}

\def\kms{K_{\CM \times S}}
\def\kmss{K_{\CM \times S \times S}}
\def\kmms{K_{\CM' \times S}}

\def\ks{K_S}
\def\kss{K_{S \times S}}

\def\rk{\text{rank }}
\def\big{\text{big}}
\def\ebig{\emph{big}}
\def\sm{\text{sm}}
\def\esm{\emph{sm}}
\def\fr{\text{Frac }\BF}

\def\quot{\text{Quot}}
\def\flag{\text{Flag}}

\def\b{{\Big|_{\infty - 0}}}
\def\i{{\int_{\infty - 0}}}

\newcommand{\pa}[1]{\left\{\mkern-4mu\left\{#1\right\}\mkern-4mu\right\}}

\def\bfZ{\bar{\fZ}}

\begin{document}

\title[Shuffle algebras associated to surfaces]{\large{\textbf{SHUFFLE ALGEBRAS ASSOCIATED TO SURFACES}}}
\author[Andrei Negu\cb t]{Andrei Negu\cb t}
\address{MIT, Department of Mathematics, Cambridge, MA, USA}
\address{Simion Stoilow Institute of Mathematics, Bucharest, Romania}
\email{andrei.negut@gmail.com}

\maketitle

\renewcommand{\thefootnote}{\fnsymbol{footnote}} 
\renewcommand{\thefootnote}{\arabic{footnote}} 

\renewcommand{\thefootnote}{\fnsymbol{footnote}} 
\renewcommand{\thefootnote}{\arabic{footnote}} 

\begin{abstract} We consider the algebra of Hecke correspondences (elementary transformations at a single point) acting on the algebraic $K$--theory groups of the moduli spaces of stable sheaves on a smooth projective surface $S$. We derive quadratic relations between the Hecke correspondences, and compare the algebra they generate with the Ding-Iohara-Miki algebra (at a suitable specialization of parameters), as well as with a generalized shuffle algebra.

\end{abstract}

\section{Introduction} 
\label{sec:introduction}

\medskip

\subsection{}

We work over an algebraically closed field of characteristic 0, which we assume is $\BC$. Fix a smooth projective surface $S$ with an ample divisor $H \subset S$, and fix a choice of rank and first Chern class $(r,c_1) \in \BN \times H^2(S,\BZ)$. We will study the moduli space of $H$--stable sheaves with these invariants and arbitrary second Chern class:
\begin{equation}
\label{eqn:m}
\CM = \bigsqcup_{c_2 = \left \lceil \frac {r-1}{2r} c_1^2 \right \rceil}^\infty \CM_{(r,c_1,c_2)}
\end{equation}
The reason for the lower bound on $c_2$ is Bogomolov's inequality, which states that the moduli space of stable sheaves is empty if $c_2 < \frac {r-1}{2r} c_1^2$. We make the following:
\begin{equation}
\label{eqn:assumption a}
\textcolor{red}{\textbf{Assumption A: }} \gcd(r, c_1 \cdot H) = 1
\end{equation}
which implies (see for example Corollary 4.6.7 of \cite{HL}) that:
$$
\exists \textcolor{red}{\textbf{ a universal sheaf }} \CU \textcolor{red}{\textbf{ on }} \CM \times S
$$
and moreover, the moduli space $\CM$ is projective (which is a consequence of the fact that under Assumption A, any semistable sheaf is stable). One could do without Assumption A, but then universal sheaves exist only locally on the moduli space $\CM$, and one would have to adapt the contents of the present paper to the setting of twisted coherent sheaves (\cite{C}). We foresee no major difficulty in doing so, but also no crucial benefit, and so we leave the details to the interested reader. \\

\noindent One of the most important objects of study for us are the $K$--theory groups:
\begin{equation}
\label{eqn:k}
K_\CM = \bigoplus_{c_2 = \left \lceil \frac {r-1}{2r} c_1^2 \right \rceil}^\infty K_{\CM_{(r,c_1,c_2)}}
\end{equation}
There are two contexts in which we will make sense of these $K$--theory groups. The first one,  somewhat particular, is when we make the restriction the appears in \cite{Ba}: 
\begin{equation}
\label{eqn:assumption s}
\textcolor{red}{\textbf{Assumption S: the canonical bundle of }}S 
\end{equation}
$$
\textcolor{red}{\textbf{is either trivial, or satisfies }} c_1(\CK_S) \cdot H < 0
$$
In this case, it is well-known that the space $\CM$ is smooth (Theorem 4.5.4 of \cite{HL}), and so the groups \eqref{eqn:k} are rings endowed with proper push-forwards and pull-backs under lci morphisms between smooth schemes (see \cite{CG}). However, our constructions make sense outside Assumption S: all we need is a $K$-theory defined for all Noetherian schemes $X$ and all \textbf{virtual} zero loci of sections $s$ of locally free sheaves on $X$ (see Subsection \ref{sub:derived bundles}). We interpret such a virtual zero section as the dg subscheme of $X$ whose dg algebra of functions is the Koszul complex of $s^\vee$, in the language of \cite{CK}. Since such dg schemes can be interpreted as derived schemes, the required $K$--theory will be the spectrum of the $\infty$--category of cohomologically bounded coherent sheaves (many thanks to Mauro Porta for pointing this out). This theory has proper push-forwards, as well as pull-backs under either lci maps, or restriction maps from a space to the virtual zero section of a locally free sheaf. \\

\noindent If there is an algebraic torus acting on $S$, for example $\BC^* \times \BC^* \curvearrowright \BP^2$, then we may also consider equivariant $K$--theory groups instead of \eqref{eqn:k}. While strictly speaking we will not follow this avenue, it is natural to expect that one can generalize many of the constructions in the present paper to non-projective surfaces $S$, as long as the torus fixed point set is proper. Examples include the total space of a line bundle over a smooth projective curve with the action of $\BC^*$ by dilating fibers, or the case of $\BC^* \times \BC^*$ scaling $\BA^2$. The latter case was treated in \cite{W} and \cite{Mod}, and the present paper grew out of the attempt to globalize the results of these two papers (note that if $S$ is not proper, one usually has to adapt the definition of the moduli space of stable sheaves, e.g. by working instead with framed sheaves). \\

\noindent An important object in the representation theory of affine quantum groups is the Ding-Iohara-Miki algebra, which is generated over the ring $\BZ[a^{\pm 1}, b^{\pm 1}]$ by the coefficients of bi-infinite series of symbols $e(z)$, $f(z)$, $h^\pm(z)$ satisfying relations \eqref{eqn:dim 1}, \eqref{eqn:dim 2}, \eqref{eqn:dim 3}. The term ``bi-infinite series" refers to formal power series indexed by all $n \in \BZ$, such as the delta function $\delta(z) = \sum_{n \in \BZ} z^n$. The space of such bi-infinite series with coefficients in an abelian group $G$ will be denoted by $G\pa{z}$. We will define operators:
\begin{equation}
\label{eqn:def e intro}
e(z) : K_\CM \longrightarrow K_{\CM' \times S} \pa{z}
\end{equation}
\begin{equation}
\label{eqn:def f intro}
f(z) : K_{\CM'} \longrightarrow K_{\CM \times S} \pa{z}
\end{equation}
given by the formal series of $K$--theory classes $\delta \left(\frac {\CL}z \right)$ on the Hecke correspondence:
$$
\xymatrix{
& \fZ = \{(\CF,\CF',x), \CF/\CF' \cong \BC_x\} \ar[ld]^{\CF} \ar[d]^{\CF'} \ar[rd]_{x} & \\
\CM & \CM' & S} 
$$
where $\CM = \CM' = $ the moduli space \eqref{eqn:m}, and the line bundle $\CL$ on $\fZ$ is has fibers equal to the 1-dimensional quotients $\CF_x/\CF'_x$. The history of the operators \eqref{eqn:def e intro} and \eqref{eqn:def f intro} is long and has generated some very beautiful mathematics, but our approach is closest to the original construction of Nakajima and Grojnowski in cohomology (see \cite{G}, \cite{Nak}), as well as the higher rank generalization by Baranovsky (see \cite{Ba}). We also consider the operators: 
\begin{equation}
\label{eqn:def h intro}
h^\pm(z) : K_{\CM} \longrightarrow K_{\CM \times S} [[z^{\mp 1}]]
\end{equation}
of tensor product with the full exterior power of the universal sheaf times $[\CK_S]^{-1} - 1$ (see \eqref{eqn:def h} for the precise formula). The meaning of the sign $\pm$ is that there are two ways to expand the full exterior power as a function of $z$, either near 0 or near $\infty$, and this gives rise to two power series of operators. Then our main result is: \\

\begin{theorem}
\label{thm:main}

The operators $e(z)$, $f(z)$, $h^\pm(z)$ satisfy the commutation relations \eqref{eqn:comm rel 1 intro}, \eqref{eqn:comm rel 2 intro}, \eqref{eqn:comm rel 3 intro}. When restricted  to the diagonal $S \hookrightarrow S \times S$, the relations match those in the Ding-Iohara-Miki algebra, specifically \eqref{eqn:dim 1}, \eqref{eqn:dim 2}, \eqref{eqn:dim 3} (in which case the parameters $a$ and $b$ are identified with the Chern roots of $\Omega_S^1$). \\

\end{theorem}

\noindent Therefore, the algebra generated by the operators \eqref{eqn:def e intro}, \eqref{eqn:def f intro}, \eqref{eqn:def h intro} can be interpreted as an ``off the diagonal" version of the Ding-Iohara-Miki algebra. To explain what we mean by this, let us make the simplifying assumption that: 
\begin{equation}
\label{eqn:big kunneth}
K_{\CM \times S \times ... \times S} \cong \km \boxtimes \ks \boxtimes ... \boxtimes \ks
\end{equation}
which holds as soon as the class of the diagonal is decomposable in $\kss$ (see \cite{CG}). In this case, the operators \eqref{eqn:def e intro}--\eqref{eqn:def h intro} can be interpreted as endomorphisms of $\km$ with coefficients in $\ks$, and the composition of two such operators can be thought of as an endomorphism of $\km$ with coefficients in $\kss \cong \ks \boxtimes \ks$. Relations \eqref{eqn:comm rel 1 intro}--\eqref{eqn:comm rel 3 intro} are equalities of endomorphisms of $\km$ with coefficients in $\kss \cong \ks \boxtimes \ks$. When one restricts the coefficients to the diagonal $\Delta^* : \kss \cong \ks \boxtimes \ks \rightarrow \ks$, then the above-mentioned relations match those in the Ding-Iohara-Miki algebra defined over the ring $\ks$ instead of over $\BZ[a^{\pm 1}, b^{\pm 1}]$. \\

\noindent We will perform explicit computations of the operators $e(z)$, $f(z)$, $h^\pm(z)$ under an extra assumption on the surface $S$. Specifically, recall that Assumption A of \eqref{eqn:assumption a} entails the existence of a universal sheaf:
\begin{equation}
\label{eqn:diagram intro}
\xymatrix{
& \CU \ar@{.>}[d] & \\
& \CM \times S \ar[ld]_{p_1} \ar[rd]^{p_2} & \\
\CM & & S}
\end{equation}
whose restriction to any point $\CF \in \CM$ is precisely $\CF$ interpreted as a sheaf on $S$ (the universal sheaf is only determined up to tensoring with line bundles in $p_1^*(\text{Pic}_\CM)$, but this ambiguity is resolved by the fact that we will work with the projectivization of $\CU$). Moreover, let us assume that the diagonal $\Delta : S \hookrightarrow S \times S$ is decomposable, i.e. there exist classes $\{l_i, l^i\} \in \ks$ for some indexing set $i \in I$ such that:
\begin{equation}
\label{eqn:diagonal decomposes}
[\CO_\Delta] = \sum_i l_i \boxtimes l^i \qquad \left(\text{the RHS is shorthand for } \sum_{i \in I} p_1^*l_i \otimes p_2^* l^i \right)
\end{equation}
In this case, one has the Kunneth decomposition \eqref{eqn:big kunneth} (see Theorem 5.6.1 of \cite{CG}), so we may use it to decompose the universal sheaf:
\begin{equation}
\label{eqn:universal decomposes}
[\CU] \ = \ \sum_{i} [\CT_i] \boxtimes l^i \in \km  \boxtimes \ks
\end{equation}
for certain $K$--theory classes $[\CT_i]$ on $\km$. We may add an extra level of concreteness to our computations by assuming that $\CT_i$ are locally free. To summarize, from Subsection \ref{sub:computations} onwards, we impose the following assumption on top of \eqref{eqn:assumption a}:
$$
\textcolor{red}{\textbf{Assumption B: there are decompositions }} \eqref{eqn:diagonal decomposes} \textcolor{red}{\textbf{ and }} \eqref{eqn:universal decomposes},
$$
\begin{equation}
\label{eqn:assumption b}
\textcolor{red}{\textbf{where }}[\CT_i] \textcolor{red}{\textbf{ are classes of locally free sheaves on }}\CM
\end{equation}
A particular example which falls under Assumption \eqref{eqn:assumption b} is $S = \BP^2$, in which case the sheaves $\CT_i$ are constructed via Beilinson's monad (Example \ref{ex:monad plane}). \\

\noindent In Section \ref{sec:shuffle}, we focus on the algebra generated by the operators $e(z)$ and define a universal shuffle algebra that describes it. More specifically, in Definition \ref{def:shuffle} we introduce the following graded algebra:
$$
\CS_\sm \subset \CS_\big = \bigoplus_{k=0}^\infty \BF_k(z_1,...,z_k)^\sym
$$
where $\BF_k$ is the coefficient ring defined in \eqref{eqn:fk}. This ring is endowed with an action of $S(k)$, and the superscript $\sym$ in the right hand side refers to rational functions that are invariant under the simultaneous action of $S(k)$ on $\BF_k$ and the variables $z_1,...,z_k$. The multiplication in $\CS_\big$ is defined in \eqref{eqn:shuffle product} with respect to the rational function \eqref{eqn:zeta}, and $\CS_\sm$ is defined as the subalgebra generated by the rational functions in a single variable $\{z_1^n\}_{n \in \BZ}$. The generators $z_1^n$ satisfy relations \eqref{eqn:comm rel 1 intro}, and because of Theorem \ref{thm:main}, it is natural to ask whether we have an action:
\begin{equation}
\label{eqn:shuffle acts}
\CS_\sm \Big|_{\BF_k \mapsto K_{S \times ... \times S}} \curvearrowright \km \qquad \text{given by} \qquad \delta \left( \frac {z_1}w \right) =  e(w) \curvearrowright \km
\end{equation}
for a smooth projective surface $S$ subject to Assumption A. We prove \eqref{eqn:shuffle acts} under the more restrictive Assumption B (see Corollary \ref{cor:act}). In general, proving that \eqref{eqn:shuffle acts} is well-defined would require one to better understand the ideal of relations between the generators $z_1^n \in \CS_\sm$ (or to connect the shuffle algebra with the $K$--theoretic Hall algebra of Schiffmann and Vasserot, see \cite{SV}). \\

\noindent Note that certain statements in the present paper (for example Proposition \ref{prop:comm rel 1}) hold even when $S$ is replaced by a smooth projective variety of dimension $ > 2$, although care must be taken in defining the appropriate dg scheme structures considered in what follows. However, in the resulting algebra of Hecke correspondences, the multiplicative structure would require replacing \eqref{eqn:zeta} by a more complicated rational function, with which we do not yet know how to deal. \\

\noindent I would like to thank Tom Bridgeland, Kevin Costello, Emanuele Macri, Davesh Maulik, Alexander Minets, Madhav Nori, Andrei Okounkov, Mauro Porta, Claudiu Raicu, Nick Rozenblyum, Francesco Sala, Aaron Silberstein, Richard Thomas and Alexander Tsymbaliuk for their help and many interesting discussions. I gratefully acknowledge the support of NSF grant DMS--1600375. \\

\subsection{}
\label{sub:not}

All schemes used in this paper will be Noetherian, and all sheaves will be coherent. Let us introduce certain notations that will be used throughout, such as:
$$
K_X = \text{Grothendieck group of } \text{Coh}(X) 
$$
for the Grothendieck group of the category of coherent sheaves on a scheme $X$. We will also encounter dg subschemes of $X$, which will all be of the form:
$$
Y \hookrightarrow  X \quad \text{where} \quad Y = \text{Spec}_X \wedge^\bullet \left(V^\vee \stackrel{s^\vee}\longrightarrow \CO_X\right)
$$
for a section $s: \CO_X \rightarrow V$ of a locally free sheaf $V$ on $X$. We call $Y$ as above the virtual zero locus of $s$. If $s$ is regular, then $Y$ is an actual scheme (this will be the case throughout the present paper, if one imposes Assumption S). \\

\noindent Let us now assume that $U$ is a coherent sheaf on $X$ of projective dimension 1, i.e.:
\begin{equation}
\label{eqn:length 1 gen}
0 \rightarrow W \rightarrow V \rightarrow U \rightarrow 0
\end{equation}
for certain locally free sheaves $V,W$ on $X$. In this case, the exterior powers of $U$:
$$
\wedge^\bullet(x \cdot U) = \sum_{i=0}^{\infty} (-x)^i \cdot [\wedge^i U], \qquad \wedge^\bullet( - x \cdot U) = \sum_{i=0}^\infty x^i \cdot [S^i U]
$$
are defined by the formulas:
\begin{equation}
\label{eqn:multiplicative}
\wedge^\bullet(x \cdot U) = \frac {\wedge^\bullet(x \cdot V)}{\wedge^\bullet(x \cdot W)}, \qquad \qquad \wedge^\bullet(- x \cdot U) = \frac {\wedge^\bullet(x \cdot W)}{\wedge^\bullet(x \cdot V)}
\end{equation}
expanded as power series. We may also think of \eqref{eqn:multiplicative} as rational functions in $x$. \\

\noindent We will often abuse notation by denoting $K$--theory classes as $V$ instead of $[V]$, and also writing $1/V$ instead of $[V^\vee]$. Therefore, the reader will often see the notation:
\begin{equation}
\label{eqn:frac notation}
\frac {V'}{V} \quad \text{instead of } \quad [V'] \otimes [V^\vee]
\end{equation}
for locally free sheaves $V,V'$. Note that we have the identity:
\begin{equation}
\label{eqn:dual wedge}
\wedge^\bullet (V^\vee) = \frac {(-1)^{\rk V}}{\det V} \wedge^\bullet(V)
\end{equation}
which also holds if $V$ is a sheaf of projective dimension 1, such as $U$ in \eqref{eqn:length 1 gen}. \\

\subsection{}
\label{sub:convergence}

A lot of our calculus will involve bi-infinite formal series, the standard example being the function $\delta(z) = \sum_{n \in \BZ} z^n \in \BZ\pa{z}$. It has the fundamental property that:
\begin{equation}
\label{eqn:fundamental property}
\delta \left( \frac zw \right) P(z) = \delta \left( \frac zw \right) P(w)
\end{equation}
for any (one-sided) Laurent series $P$. An intuitive way of writing the $\delta$ function is:
$$
\delta(z) = \frac 1{1 - \frac 1z} - \frac 1{1 - \frac 1z}
$$
where the first fraction is expanded in negative powers of $z$ and the second fraction is expanded in positive powers of $z$. We will use similar notation for any rational function $R(z)$:
\begin{equation}
\label{eqn:expand}
R(z) \b = \Big\{ R \text{ expanded around } z = \infty \Big\} - \Big\{ R \text{ expanded around } z = 0 \Big\}
\end{equation}
A way to extract mileage from this notation is to interpret it as a residue computation. Specifically, the coefficient of $z^{-n}$ in the right hand side of \eqref{eqn:expand} equals:
\begin{equation}
\label{eqn:def single}
\i z^n R(z) \quad := \quad \int_{|z| = r_{\text{big}}} z^n R(z) \frac {dz}{2\pi i z} - \int_{|z| = r_{\text{small}}} z^n R(z) \frac {dz}{2\pi i z}
\end{equation}
with $r_{\text{big}}$ and $r_{\text{small}}$ being bigger and smaller, respectively, than the finite (that is, different from $0$ and $\infty$) poles of the rational function $R(z)$. In general, given a bi-infinite formal series $e(z) = \sum_{n \in \BZ} e_n z^{-n}$, we may recover its coefficients as:
\begin{equation}
\label{eqn:coeff series}
e_n = \i z^n e(z)
\end{equation}
We will also encounter the notation $|_{0 - \infty} = - |_{\infty - 0}$ and $\int_{0 - \infty} = - \i$. \\

\section{Geometry of the moduli space of sheaves}
\label{sec:geometry}

\medskip

\subsection{}
\label{sub:universal}

We operate under Assumption A of \eqref{eqn:assumption a} throughout this Section. This means that the smooth projective surface $S$, the ample divisor $H$ and $(r,c_1) \in$ $\BN \times H^2(S,\BZ)$ are such that there exists a universal sheaf $\CU$ on $\CM \times S$, where $\CM$ denotes the moduli space of stable sheaves with the invariants $(r,c_1)$ on $S$.\footnote{Note that this assumption can be dropped if one is willing to think of $\CU$ as a twisted sheaf, i.e. that the universal sheaf exists locally on $\CM$ and the gluing maps between such local universal sheaves are defined up to tensoring with certain local line bundles, see \cite{C}} \\

\noindent Because the universal sheaf $\CU$ is flat over $\CM$, it inherits certain properties from the stable sheaves it parametrizes, such as having projective dimension one (indeed, any semistable sheaf of rank $r>0$ is torsion free, and any torsion free sheaf on a smooth projective surface has projective dimension one, see Example 1.1.16 of \cite{HL}): \\

\begin{proposition}
\label{prop:length 1}

There exists a short exact sequence:
\begin{equation}
\label{eqn:length 1}
0 \rightarrow \CW \rightarrow \CV \rightarrow \CU \rightarrow 0
\end{equation}
with $\CW$ and $\CV$ locally free sheaves on $\CM \times S$. \\

\end{proposition}

\begin{proof} Consider the projection maps $p_1 : \CM \times S \rightarrow \CM$ and $p_2 : \CM \times S \rightarrow S$. For any large enough $m \in \BN$, we have:
$$
p_{1*} \left(\CU \otimes p_2^*\CO(mH) \right) \text{ is locally free on }\CM 
$$
$$
R^ip_{1*} \left(\CU \otimes p_2^*\CO(mH) \right) = 0 \qquad \quad \ \forall \ i>0
$$
and the natural adjunction map:
\begin{equation}
\label{eqn:def v}
\CV := p_1^* \Big[ p_{1*} \left(\CU \otimes p_2^*\CO(mH) \right) \Big] \otimes p_2^*\CO(-mH) \stackrel{\text{ev}}\longrightarrow \CU
\end{equation}
is surjective (see \cite{HL} for the proof of these statements; they actually hold for any flat family of semistable sheaves). Then we define the short exact sequence \eqref{eqn:length 1} by setting $\CW = \text{Ker}(\text{ev})$. Since the sheaves $\CW,\CV,\CU$ are all flat over $\CM$, then if we restrict to any closed point $\{\CF\} \times S \hookrightarrow \CM \times S$, we obtain a short exact sequence:
$$
0 \rightarrow \CW|_{\{\CF\} \times S} \rightarrow \CV|_{\{\CF\} \times S} = H^0\left(\CF \otimes \CO(mH) \right) \otimes \CO(-mH) \stackrel{\text{ev}}\rightarrow \CF \rightarrow 0
$$
Since $\CV$ is locally free and $\CF$ has projective dimension 1, then $\CW|_{\{\CF\} \times S}$ is locally free for all $\CF$. By Lemma 2.1.7 of \cite{HL}, this implies that $\CW$ is locally free on $\CM \times S$.

\end{proof}

\subsection{}
\label{sub:derived bundles}

For a locally free sheaf $V$ on a Noetherian scheme $X$, we write $\BP_X (V)$ for the Proj construction applied to the sheaf of $\CO_X$-algebras $S^*_{\CO_X} V$. We may extend this notion to a coherent sheaf of projective dimension 1, namely $U = V / W$ where $V$ and $W$ are locally free sheaves on $X$. In this case, we \textbf{define}:
\begin{equation}
\label{eqn:derived proj 1}
\xymatrix{\BP_X(U) \ar[rd]_\pi \ar@{^{(}->}[r]^-\iota
& \BP_{X} \left( V \right) \ar[d]^\rho \\
& X}
\end{equation}
where $\iota$ is the virtual zero locus of the map $s : \rho^*(W) \hookrightarrow \rho^*(V) \stackrel{\text{taut}}\longrightarrow \CO(1)$ on $\BP_X(V)$. In other words $\BP_X(U)$ is \textbf{defined} as the following dg subscheme of $\BP_X(V)$:
$$
\BP_X(U) = \text{Spec}_{\BP_X(V)} \left( ... \stackrel{s^\vee}\longrightarrow \wedge^2 \rho^*(W) \otimes \CO(-2) \stackrel{s^\vee}\longrightarrow \rho^*(W) \otimes \CO(-1) \stackrel{s^\vee}\longrightarrow \CO_X \right)
$$
We will encounter the $K$--theoretic push-forward and pull-back maps associated to the map $\pi$ in \eqref{eqn:derived proj 1}. These will always be computed by factoring $\pi$ into $\iota$ and $\rho$, and we note that these maps admit push-forwards (since $\iota$ is a closed embedding and $\rho$ is a projective bundle) and pull-back maps (since $\iota$ is lci in a dg sense and $\rho$ is smooth). \\

\noindent In the same situation as above, we shall encounter the dg scheme $\BP_X(U^\vee[1])$, where $[1]$ stands for homological shift. By definition, this is the dg subscheme:
\begin{equation}
\label{eqn:derived proj 2}
\xymatrix{\BP_X(U^\vee[1]) \ar[rd]_{\pi'} \ar@{^{(}->}[r]^-{\iota'}
& \BP_{X} \left( W^\vee \right) \ar[d]^{\rho'} \\
& X}
\end{equation}
defined as the virtual zero locus of $s : \CO(-1) \stackrel{\text{taut}}\longrightarrow {\rho'}^*(W) \hookrightarrow {\rho'}^*(V)$ on $\BP_X(W^\vee)$. \\

\subsection{}
\label{sub:locus}

When $\CU$ is the universal sheaf on $\CM \times S$, which is $\ \cong \CV/\CW$ by Proposition \ref{prop:length 1}, we may apply the definitions in the previous Subsection and introduce the following: \\

\begin{definition}
\label{def:tower}

Consider the dg scheme:
\begin{equation}
\label{eqn:tower 1}
\fZ = \BP_{\CM \times S} \left(\CU \right)
\end{equation}

\end{definition}

\bigskip

\noindent At the level of usual (non-dg) schemes, the fiber of $\fZ$ over a point $(\CF, x) \in \CM \times S$ parametrizes surjective maps $\phi : \CF \twoheadrightarrow \BC_x$ up to rescaling. The datum of such a map is equivalent to the datum of a colength 1 subsheaf $\CF' \subset \CF$. As we will show in Proposition \ref{prop:hecke}, the sheaf $\CF'$ is stable if and only if $\CF$ is stable, so we conclude that the dg scheme $\fZ$ is supported on the usual scheme: 
\begin{equation}
\label{eqn:locus}
\bfZ = \Big\{ (\CF,\CF') \text{ s.t. } \CF' \subset \CF \text{ and } \CF/\CF' \cong \BC_x \text{ for some } x \in S \Big\} \subset \CM \times \CM'
\end{equation}
where $\BC_x$ denotes the skyscraper sheaf over the closed point $x$. Recall that we write $\CM = \CM'$ for the moduli space of stable sheaves with all possible $c_2$, but we use different notations to emphasize the fact that $\CF$ and $\CF'$ of \eqref{eqn:locus} lie in different copies $\CM$ and $\CM'$ of this moduli space. We will give a precise definition of the closed subscheme \eqref{eqn:locus} in the Appendix, by showing that it represents the functor $\sZ$ of Subsection \ref{sub:flag}. Because of this, $\bfZ \times S$ admits universal sheaves $\CU$, $\CU'$, together with an inclusion:
\begin{equation}
\label{eqn:incl}
\CU' \hookrightarrow \CU \quad \text{on } \bfZ \times S
\end{equation}
In fact, the quotient of \eqref{eqn:incl} is supported on the graph $\Gamma : \bfZ \hookrightarrow \bfZ \times S$ of the projection $p_S : \bfZ \rightarrow S$, $p_S(\CF, \CF') = x$. Therefore, we have a short exact sequence: 
\begin{equation}
\label{eqn:ses universal}
0 \rightarrow \CU' \rightarrow \CU \rightarrow \Gamma_*(\CL) \rightarrow 0
\end{equation}
of sheaves on $\bfZ \times S$, where the so-called \textbf{tautological line bundle} $\CL$ has fibers:
\begin{equation}
\label{eqn:tautological line}
\xymatrix{
\CL \ar@{.>}[d] \\
\bfZ} \qquad \qquad \CL|_{(\CF,\CF',x)} = \CF_x / \CF'_x
\end{equation}
More rigorously, $\CL$ is defined as the push-forward of $\CU/\CU'$ from $\bfZ \times S$ to $\bfZ$. It is not hard to see that $\CL$ coincides with $\CO(1)$ on the projectivization \eqref{eqn:tower 1}. \\

\subsection{} 

The previous Subsection states that the fiber of the map $\bfZ \rightarrow \CM \times S$ over a closed point $(\CF,x)$ is the projective space $\BP \Hom(\CF,\BC_x)$. We wish to obtain a similar description for the dg scheme $\fZ$. Combining \eqref{eqn:derived proj 1} with \eqref{eqn:tower 1}, it follows that the fiber of $\fZ \rightarrow \CM \times S$ over $(\CF,x)$ coincides with the two step complex:
\begin{equation}
\label{eqn:vanc}
\fZ\Big|_{(\CF,x) \in \CM \times S} = \BP\Big[ \Hom(\CV,\BC_x) \rightarrow \Hom(\CW,\BC_x) \Big]
\end{equation}
However, we have the long exact sequence associated to \eqref{eqn:length 1}:
$$
0 \rightarrow \Hom(\CF, \BC_x) \rightarrow \Hom(\CV,\BC_x) \rightarrow \Hom(\CW,\BC_x) \rightarrow \Ext^1(\CF,\BC_x) \rightarrow 0
$$
since $\Ext^1(\CV,\BC_x) = 0$ for any locally free sheaf $\CV$ on a smooth surface (this is an easy exercise that we leave to the interested reader). Therefore, we conclude that the complex in the right-hand side of \eqref{eqn:vanc} is quasi-isomorphic to $\RHom(\CF,\BC_x)$, which is to say that we have the following quasi-isomorphism of dg vector spaces:
\begin{equation}
\label{eqn:ouver}
\fZ\Big|_{(\CF,x) \in \CM \times S} \cong \BP \RHom(\CF,\BC_x)
\end{equation}
The discussion above and below is given in terms of closed points to keep the notation simple. The interested reader may readily translate it in terms of dg scheme-valued points. We wish to show that \eqref{eqn:ses universal} holds for the dg scheme $\fZ$ as well, but to this end, we need to better understand the complex \eqref{eqn:vanc}. Recall from the proof of Proposition \ref{prop:length 1} that $\CV|_{\{\CF\} \times S} = H^0(\CF(mH)) \otimes \CO(-mH)$ for some large enough natural number $m$. Therefore, a ``point" in the complex \eqref{eqn:vanc} comes from a pair of homomorphisms $\bar{\phi}$ and $\bar{\bar{\phi}}$ that make the following diagram commute:

$$
\xymatrix{ & 0 & 0 & 0 & \\
0 \ar[r] & \CF' \ar[r] \ar[u]  & \CF  \ar[r]^-\phi \ar[u] & \BC_x \ar[u] \ar[r] & 0\\
0 \ar[r] & \sharp \ar[r] \ar[u] & H^0(\CF(mH))(-mH) \ar[r]^-{\bar{\phi}} \ar[u]^{\text{ev}} & \CO(-mH) \ar[u] \ar[r] & 0 \\
0 \ar[r] & \flat \ar[r] \ar[u] & \CW \ar[r]^-{\bar{\bar{\phi}}} \ar[u] & \fm (-mH) \ar[u] \ar[r] & 0 \\
& 0 \ar[u] & 0 \ar[u] & 0 \ar[u] & } 
$$
The three sheaves in the leftmost column are defined as the kernels of $\phi$, $\bar{\phi}$, $\bar{\bar{\phi}}$. Since $\bar{\phi}$ is a constant matrix, the kernel $\sharp$ is a codimension 1 subspace of $H^0(\CF(mH))$, tensored with $\CO(-mH)$. However, \cite{HL} show that $m$ can be chosen large enough so that if $\CF$ is globally generated by $\CO(mH)$, then so is any stable colength 1 subsheaf $\CF'$. This implies that $\sharp$ must be equal to $H^0(\CF'(mH)) \otimes \CO(-mH)$, and the map $\sharp \rightarrow \CF'$ must be equal to the evaluation map. This also implies that $\flat = \CW'$, and so we conclude that a point in the fiber \eqref{eqn:vanc} corresponds to an entire diagram:
\begin{equation}
\label{eqn:huge}
\xymatrix{ & 0 & 0 & 0 & \\
0 \ar[r] & \CF' \ar[r] \ar[u]  & \CF  \ar[r]^-\phi \ar[u] & \BC_x \ar[u] \ar[r] & 0\\
0 \ar[r] & H^0(\CF'(mH))(-mH) \ar[r] \ar[u]^{\text{ev}} & H^0(\CF(mH))(-mH) \ar[r]^-{\bar{\phi}} \ar[u]^{\text{ev}} & \CO(-mH) \ar[u] \ar[r] & 0 \\
0 \ar[r] & \CW' \ar[r] \ar[u]^{\iota'} & \CW \ar[r]^-{\bar{\bar{\phi}}} \ar[u]^\iota & \fm (-mH) \ar[u] \ar[r] & 0 \\
& 0 \ar[u] & 0 \ar[u] & 0 \ar[u] & } 
\end{equation}
This implies that the inclusion $\CU' \hookrightarrow \CU$ holds on the dg scheme $\fZ \times S$, where the map between the complexes $\CU' \cong [\CW'\rightarrow \CV']$ and $\CU \cong [\CW\rightarrow \CV]$ on $\fZ \times S$ is induced by the horizontal arrows in diagram \eqref{eqn:huge}.

\subsection{}
\label{sub:tower 2}

We have the three forgetful maps:
$$
\xymatrix{
& \fZ \ar[ld]_{p_1} \ar[d]^{p_S} \ar[rd]^{p_2} & \\
\CM & S & \CM'} 
$$
given by sending a flag $(\CF,\CF')$ to $\CF$, $\left( \text{supp } \CF/\CF' \right)$ and $\CF'$, respectively. After defining $\fZ$ as the projectivization of a coherent sheaf over $p_1 \times p_S$, we will now prove that it is also the projectivization of a coherent sheaf over $p_2 \times p_S$. However, this time we will use the formalism of \eqref{eqn:derived proj 2} instead of \eqref{eqn:derived proj 1}. \\

\begin{proposition}
\label{prop:tower}

The projection map $\fZ \rightarrow \CM' \times S$ is the projectivization:
\begin{equation}
\label{eqn:tower 2}
\fZ \cong \BP_{\CM' \times S} \left( {\CU'}^\vee \otimes \CK_S [1] \right) 
\end{equation}
We write $\CK_S$ both for the canonical bundle of $S$ and for its pull-back to $\CM' \times S$. The line bundle $\CL$ on $\fZ$ coincides with $\CO(-1)$ in the right-hand side. \\

\end{proposition}

\begin{proof} To keep the explanation simple, we will prove the Proposition at the level of closed points, and leave the analogous language for arbitrary dg scheme-valued points to the interested reader. As we have seen in the previous Subsection, points of $\fZ$ are in one-to-one correspondence with diagrams \eqref{eqn:huge}. In order to prove Proposition \ref{prop:tower}, we need to show that there is a one-to-one correspondence:
\begin{equation}
\label{eqn:task}
\text{diagrams \eqref{eqn:huge} for given }(\CF',x) \quad \leftrightarrow \qquad \qquad \qquad 
\end{equation}
$$
\qquad \qquad \qquad \qquad \leftrightarrow  \quad \BP \Big[\Hom({\CW'}^\vee \otimes \CK_S, \BC_x) \rightarrow  \Hom({\CV'}^\vee \otimes \CK_S, \BC_x) \Big]
$$
Consider the following commutative diagram, induced from the Ext long exact sequences corresponding to the sequence $0 \rightarrow \fm(-mH) \rightarrow \CO(-mH) \rightarrow \BC_x \rightarrow 0$:
\begin{equation}
\label{eqn:exact sequences}
\xymatrix{ 0 \ar[d] & 0 \ar[d] \\ 
\Ext^1(\CO(-mH), \CW') \ar[r]^\cong \ar[d] & \Ext^1(\CO(-mH), \CV') \ar[d] \\
\Ext^1(\fm(-mH), \CW') \ar[d] \ar[r] & \Ext^1(\fm(-mH), \CV') \ar[d] \\
\Ext^2 (\BC_x, \CW') \ar[d] \ar[r] & \Ext^2 (\BC_x, \CV') \ar[d] \\
\Ext^2 (\CO(-mH), \CW') \ar[d] \ar[r]^\cong & \Ext^2 (\CO(-mH), \CV') \ar[d] \\ 
\Ext^2 (\fm(-mH), \CW') \ar[d] \ar[r]^\cong & \Ext^2 (\fm(-mH), \CV') \ar[d] \\
0 & 0 }
\end{equation}
where $\CV' = H^0(\CF'(mH)) \otimes \CO(-mH)$ and $\CW' = \text{Ker } ( \CV' \stackrel{\text{ev}}\rightarrow \CF')$. The topmost terms are 0 because $\Ext^1(\BC_x,\CE) = 0$ for any locally free sheaf $\CE$, a well-known fact. Meanwhile, the first, fourth and fifth horizontal arrows are isomorphisms because the kernel and cokernel of these maps vanish due to the global generation and cohomology vanishing of $\CF'(mH)$. Therefore, we conclude that the complex:
$$
\Big[ \Ext^1(\fm(-mH), \CW') \rightarrow \Ext^1(\fm(-mH), \CV') \Big]
$$
is quasi-isomorphic to:
$$
\Big[ \Ext^2 (\BC_x, \CW') \rightarrow \Ext^2 (\BC_x, \CV') \Big] \cong \Big[ \Hom (\CW' \otimes \CK_S^{-1}, \BC_x)^\vee \rightarrow \Hom (\CV' \otimes \CK_S^{-1}, \BC_x)^\vee \Big]
$$
(the isomorphism in the latter equation is Serre duality). Therefore, the task \eqref{eqn:task} boils down to constructing a one-to-one correspondence:
\begin{equation}
\label{eqn:task task}
\text{diagrams \eqref{eqn:huge} for given }(\CF',x) \quad \leftrightarrow \qquad \qquad \qquad 
\end{equation}
$$
\qquad \qquad \qquad \qquad \leftrightarrow  \quad \BP \Big[ \Ext^1(\fm(-mH), \CW') \rightarrow \Ext^1(\fm(-mH), \CV') \Big]
$$
The correspondence in the $\rightarrow$ direction is given by assigning to a diagram \eqref{eqn:huge} its bottom-most row, which is an extension $\in \Ext^1(\fm(-mH), \CW')$. The vanishing of this extension when pushed forward via $\CW' \hookrightarrow \CV'$ happens because of the middle row in \eqref{eqn:huge}, which is a short exact sequence of trivial locally free sheaves $\otimes \CO(-mH)$. \\

\noindent As for the correspondence in the $\leftarrow$ direction, we need to show that to any extension in $\Ext^1(\fm(-mH), \CW')$ which vanishes when pushed forward under $\CW' \hookrightarrow \CV'$, we may associate a diagram \eqref{eqn:huge}. In other words, we need to show that one may reconstruct the entire diagram \eqref{eqn:huge} from the solid lines below:
\begin{equation}
\label{eqn:aaa}
\xymatrix{ & 0 &  & 0 & \\
& \CF' \ar[u]  & & \BC_x \ar[u] & \\
0 \ar@{.>}[r] & H^0(\CF'(mH))(-mH) \ar[u] \ar@{.>}[r] & ? \ar@{.>}[r] & \CO(-mH) \ar[u] \ar@{.>}[r] & 0 \\
0 \ar[r] & \CW' \ar[r] \ar[u] & \CW \ar[r] \ar@{-->}[u] & \fm (-mH) \ar[u] \ar[r] & 0 \\
& 0 \ar[u] & & 0 \ar[u] & } 
\end{equation}
and the information that the bottom short exact sequence splits if we push it out under $\CW' \hookrightarrow  H^0(\CF'(mH))(-mH)$. Indeed, this information amounts to the same thing as a split short exact sequence (which we display by dotted arrows in the middle row of \eqref{eqn:aaa}) with $? = $ (\text{vector space}) $\otimes \ \CO(-mH)$, and a vertical map $\CW \dashrightarrow \ ?$ which makes the whole diagram commute. From this datum, we may reconstruct the diagram \eqref{eqn:huge}, which thus allows us to reconstruct the sheaf $\CF := \ ?/\CW$ from $\CF'$. It is obvious that the correspondences $\rightarrow$ and $\leftarrow$ constructed in the present and preceding paragraphs are inverses of each other. 

\end{proof}

\subsection{}
\label{sub:smooth}

Let us now study the particular situation when then moduli space of stable sheaves is smooth. By the well-known Kodaira-Spencer isomorphism, we have:
$$
\Tan_{\CF}(\CM) = \text{Ext}^1(\CF,\CF)
$$
Stable sheaves are simple, i.e. $\Hom(\CF,\CF) = \BC$, and the obstruction to $\CM$ being smooth lies within $\Ext^2(\CF,\CF)$. This group can be computed using Serre duality:
\begin{equation}
\label{eqn:serre}
\Ext^2(\CF,\CF) \cong \Hom(\CF, \CF \otimes \CK_S)^\vee
\end{equation}
Under Assumption S, it is easy to show that the vector space on the right is trivial. Indeed, if $\CK_S \cong \CO_S$, then the fact that stable sheaves are simple implies that the vector space \eqref{eqn:serre} is canonically $\BC$. On the other hand, if $c_1(\CK_S) \cdot H < 0$, then the kernel or cokernel of any non-zero homomorphism $\CF \rightarrow \CF \otimes \CK_S$ would violate the stability of $\CF$, and so the vector space \eqref{eqn:serre} is zero. Since the fact that \eqref{eqn:serre} is 0 or canonically $\BC$ implies that $\CM$ is smooth (\cite{HL}, Theorem 4.5.4), then Assumption S implies that $\CM$ is smooth. Moreover:
\begin{equation}
\label{eqn:dimension 0}
\dim_{\CF}\CM = 1 - \chi(\CF,\CF) + \e
\end{equation}
where $\e$ is 1 or 0, depending on which of the two conditions of Assumption S holds (the summands $1$ and $\e$ in \eqref{eqn:dimension 0} are the dimensions of the vector spaces $\Hom(\CF,\CF)$ and $\Ext^2(\CF,\CF)$, as described in the previous paragraph). The Euler characteristic can be computed using the Hirzebruch-Riemann-Roch theorem, and one obtains:
\begin{equation}
\label{eqn:dimension}
\dim \CM_{(r,c_1,c_2)} = 1 + \e + \text{const} + 2r c_2
\end{equation}
where the constant in \eqref{eqn:dimension} depends only on $S,H,r,c_1$. Together with the fact that $\fZ = \BP_{\CM \times S} (\CU)$ for the rank $r$ universal sheaf $\CU$, this implies that:
\begin{equation}
\label{eqn:expected dimension}
\dim \Big\{ \fZ \cap \left( \CM_{(r,c_1,c_2)} \times \CM_{(r,c_1,c_2+1)} \right) \Big\} = 1 + \e + \text{const} + 2r c_2 + r + 1
\end{equation}
By definition, the number above is the ``expected dimension" of the scheme $\bfZ$ on which the dg scheme $\fZ$ is supported, but in general this need not equal to the actual dimension of $\bfZ$. However, have the following result. \\

\begin{proposition}
\label{prop:smooth}

Under Assumption S, the scheme $\bfZ$ is smooth of expected dimension \eqref{eqn:expected dimension}, and it coincides with the dg scheme $\fZ$. Moreover, the map:
$$
p_S : \fZ \longrightarrow S, \qquad (\CF' \subset \CF) \mapsto \emph{supp }\CF/\CF'
$$
is a smooth morphism. \\

\end{proposition}

\begin{proof} Recall that $\bfZ$ is the set-theoretic zero locus of a section $s$ of a locally free sheaf on a smooth space (since $\CM$ is smooth, so are projective bundles over it), and $\fZ$ is the virtual zero locus of $s$. To show that $\fZ \cong \bfZ$, we must show that the section is regular, and to do so it suffices to show that $\bfZ$ has expected dimension. In fact, we will even show that the dimensions of the tangent spaces of $\bfZ$ are equal to the expected dimension \eqref{eqn:expected dimension}, which will also prove the smoothness of $\bfZ \cong \fZ$. By a general argument pertaining to moduli spaces of flags of sheaves, the space:
$$
\Tan_{(\CF' \subset \CF)} \bfZ \subset \Tan_{(\CF, \CF')} \left( \CM \times \CM' \right) = \Ext^1(\CF,\CF) \oplus \Ext^1(\CF',\CF')
$$
consists of pairs of extensions which are compatible under the inclusion $\CF' \subset \CF$:
\begin{equation}
\label{eqn:2 tan}
\xymatrix{0 \ar[r] & \CF \ar[r]  & \CS  \ar[r] & \CF \ar[r] & 0\\
0 \ar[r] & \CF' \ar[r] \ar@{^{(}->}[u] & \CS' \ar[r] \ar@{^{(}->}[u] & \CF' \ar@{^{(}->}[u] \ar[r] & 0}
\end{equation}
A simple diagram chase shows that such pairs of extensions are precisely those which map to the same extension in $\Ext^1(\CF',\CF)$, and so we conclude that:
\begin{equation}
\label{eqn:sigma}
\Tan_{(\CF' \subset \CF)} \bfZ = \Ker \Big[ \Ext^1(\CF,\CF) \oplus \Ext^1(\CF',\CF') \stackrel{\sigma}\longrightarrow \Ext^1(\CF',\CF) \Big] 
\end{equation}
where the map $\sigma$ is the difference of the two natural maps  induced by $\CF' \subset \CF$. The Hirzebruch-Riemann-Roch theorem implies the following equalities:
\begin{align}
\dim \ \Ext^1(\CF,\CF) &=  1 + \e + 2rc_2 + \text{const} \label{eqn:est 1} \\
\dim \ \Ext^1(\CF',\CF') &= 1 + \e + 2rc_2 + 2r + \text{const} \label{eqn:est 2} \\
\dim \ \Ext^1(\CF,\CF')  &= 0 + \e + 2rc_2 + r + \text{const} \label{eqn:est 3} \\
\dim \ \Ext^1(\CF',\CF) &= 1 + 0 + 2rc_2 + r + \text{const} \label{eqn:est 4}
\end{align}
where const only depends on $S,H,r,c_1$, and the first two summands in each term in the right-hand side come from the dimensions of $\Hom$ and $\Ext^2$ (the number $\e$ is 1 or 0, depending on which of the two conditions of Assumption S holds). Because of these dimension estimates and \eqref{eqn:sigma}, we conclude that the tangent spaces to $\bfZ$ have dimension \eqref{eqn:expected dimension} if and only if the map $\sigma$ of \eqref{eqn:sigma} has cokernel of dimension $1 - \e$. To analyze this cokernel, recall that $\sigma$ is the difference of the maps $\sigma_1$ and $\sigma_2$ in the following commutative diagram with all rows and columns exact:
$$
\xymatrix{ & \Hom(\CF,\BC_x) \ar[d]_-{\lambda_2} & & & \\
\Ext^1(\BC_x,\CF') \ar[r]^{\lambda_1} & \Ext^1(\CF,\CF') \ar[r]^{\rho_1} \ar[d]_{\rho_2} & \Ext^1(\CF',\CF') \ar[d]_{\sigma_2} \ar[r]^\mu & \Ext^2(\BC_x, \CF') \ar[d]_\zeta \ar@{->>}[r] & \BC^\e \ar[d]_\cong \\
& \Ext^1(\CF, \CF) \ar[r]^{\sigma_1} & \Ext^1(\CF',\CF) \ar[r]^{\tau_1} \ar[d]_{\tau_2} & \Ext^2(\BC_x,\CF) \ar@{->>}[r] \ar[d]_{p_2} & \BC^\e \\
& & \Ext^1(\CF',\BC_x) \ar[r]^{p_1} & \Ext^2(\BC_x,\BC_x) & &}
$$
where $\BC_x = \CF/\CF'$. The vector spaces $\BC^\e$ in the right hand side of the diagram are $\Ext^2(\CF,\CF')$ and $\Ext^2(\CF,\CF)$, respectively, and the map between them is the Serre dual of the isomorphism $\Hom(\CF,\CF) \rightarrow \Hom(\CF', \CF)$. We will refer to the display above as the \textbf{big diagram}, and use the notations therein for the remainder of this proof. Then we have: \\

\begin{itemize}[leftmargin=*]

\item If $\CK_S \cong \CO_S$, then $\e = 1$ and we must show that $\sigma$ is surjective. To this end, we claim that $p_2 \circ \tau_1 = 0$, because the Serre dual of this map is:
$$
\Hom(\BC_x, \BC_x) \rightarrow \Hom(\CF,\BC_x) \rightarrow \Ext^1(\CF,\CF')
$$
and the generator of the vector space $\Hom(\BC_x,\BC_x)$ goes to the extension:
$$
0 \longrightarrow \CF' \stackrel{\text{incl}_1}\longrightarrow \text{Ker} (\CF \oplus \CF \longrightarrow \BC_x) \stackrel{\text{pr}_2}\longrightarrow \CF \longrightarrow 0 
$$
where $\text{incl}_1$ is inclusion into the first factor and $\text{pr}_2$ is projection onto the second factor. The above extension is split by the map $\CF \rightarrow \text{Ker} (\CF \oplus \CF \rightarrow \BC_x)$ given by the formula $f \mapsto (f,-f)$ for all local sections $f$ of $\CF$. Therefore, $p_2 \circ \tau_1 = 0$, which means that for any $c \in \Ext^1(\CF',\CF)$, there exists $d$ such that $\tau_1(c) = \zeta(d)$. However, $d$ must be in the image of $\mu$, because of the fact that the (top, right)--most vertical arrow is an isomorphism. Therefore, $d = \mu(e_2)$ for some $e_2$ and thus $\tau_1(c) = \tau_1(\sigma_2(e_2))$. Therefore $\exists \ e_1$ such that $c = \sigma_1(e_1) + \sigma_2(e_2)$. \\

\item If $c_1(\CK_S) \cdot H < 0$, then $\e = 0$ and we must show that the map $\sigma$ of \eqref{eqn:sigma} has 1-dimensional cokernel. Since $\Ext^2(\BC_x, \BC_x) \cong \BC$, it is enough to show that any $c \in \Ker(p_2 \circ \tau_1)$ lies in $\text{Im } \sigma_1 + \text{Im } \sigma_2$. This is done by repeating the argument in the previous bullet after the words ``for any $c$". \\ 

\end{itemize}

\noindent Since $\fZ$ and $S$ are smooth equidimensional varieties over a field, in order to show that the map $\fZ \rightarrow S$ is a smooth morphism, it suffices to show that the differential:
$$
p_{S*} : \Tan_{(\CF' \subset \CF)}\fZ \longrightarrow \Tan_x S
$$
is surjective. If one interprets tangent vectors to $\fZ$ as diagrams \eqref{eqn:2 tan}, then $p_{S*}$ applied to such a vector is given by the short exact sequence:
$$
\Big\{ 0 \rightarrow \CF/\CF' \rightarrow \CS/\CS' \rightarrow \CF/\CF' \rightarrow 0 \Big\} \in \Ext^1(\BC_x,\BC_x) = \Tan_x S
$$
Therefore, a tangent vector \eqref{eqn:2 tan} is in the kernel of $p_{S*}$ if and only if $\CS/\CS' \cong \BC_x \oplus \BC_x$, and this happens if and only if there exists a subsheaf $\CS' \subset \CT \subset \CS$ which fits as an intermediate step in \eqref{eqn:2 tan}, see below:
$$
\xymatrix{0 \ar[r] & \CF \ar[r]  & \CS  \ar[r] & \CF \ar[r] & 0 \\
0 \ar[r] & \CF' \ar[r] \ar@{^{(}->}[u] & \CT \ar@{^{(}->}[u] \ar[r] & \CF \ar[r] \ar@{=}[u] & 0 \\
0 \ar[r] & \CF' \ar[r] \ar@{=}[u]  & \CS' \ar[r] \ar@{^{(}->}[u] & \CF' \ar@{^{(}->}[u] \ar[r] & 0}
$$
This is precisely saying that the two extensions in diagram \eqref{eqn:2 tan} come from the maps $(\rho_1,\rho_2)$ applied to the extension $0 \rightarrow \CF' \rightarrow \CT \rightarrow \CF \rightarrow 0$, and so we conclude:
\begin{equation}
\label{eqn:ker}
\text{Ker } p_{S*} = \text{Im }( \rho_1, \rho_2)
\end{equation}
where we think of $(\rho_1,\rho_2)$ as a map $\Ext^1(\CF,\CF') \rightarrow \Ext^1(\CF,\CF) \oplus \Ext^1(\CF',\CF')$. \\

\begin{claim}
\label{claim:luna}

The map $(\rho_1, \rho_2)$ is injective, and thus together with \eqref{eqn:ker}, we have:
\begin{equation}
\label{eqn:ker 2}
\emph{Ker } p_{S*} \cong \emph{Ext}^1(\CF,\CF')
\end{equation}

\end{claim}

\bigskip

\noindent Together with the dimension computations in \eqref{eqn:expected dimension} and \eqref{eqn:est 3}, the above Claim implies that the dimension of $\Tan_{(\CF'\subset \CF)} \fZ / \text{Ker }p_{S*}$ is 2. Therefore, the map:
$$
\frac {\Tan_{(\CF'\subset \CF)} \fZ}{\text{Ker }p_{S*}} \stackrel{p_{S*}}\longrightarrow \Tan_x S
$$
is an injection of vector spaces of dimension 2, and therefore surjective. \\

\noindent Let us now prove Claim \ref{claim:luna}. With the notation as in the big diagram, it is enough to show that the subspaces $\text{Im }\lambda_1$ and $\text{Im } \lambda_2$ have trivial intersection in $\Ext^1(\CF,\CF')$. To do so, let us consider the reflexive hull:
\begin{equation}
\label{eqn:sha}
0 \longrightarrow \CF \longrightarrow \CV \longrightarrow \CV/\CF \longrightarrow 0
\end{equation}
where $\CV$ is locally free and $\CV/\CF$ has finite length. We similarly have a short exact sequence $0 \rightarrow \CF' \rightarrow \CV \rightarrow \CV/\CF' \rightarrow 0$, and the associated long exact sequence is:
$$
... \rightarrow \Hom(\BC_x,\CV) \rightarrow \Hom(\BC_x, \CV/\CF') \rightarrow \Ext^1(\BC_x, \CF') \rightarrow \Ext^1(\BC_x, \CV) \rightarrow ...
$$
Because $\CV$ is locally free, the vector spaces at the endpoints of the above sequence are 0, and so we have an isomorphism $\Hom(\BC_x, \CV/\CF') \cong \Ext^1(\BC_x,\CF')$. Moreover, this isomorphism fits into the following commutative diagram, where all four maps are induced from various Ext long exact sequences:
$$
\xymatrix{\Ext^1(\BC_x,\CF') \ar[r]^{\lambda_1} & \Ext^1(\CF,\CF') \\
\Hom(\BC_x,\CV/\CF') \ar[r]^\omega \ar[u]^\cong & \Hom(\CF,\CV/\CF') \ar[u]_\lambda}
$$
Moreover, the map $\lambda_2 : \Hom(\CF,\BC_x) \rightarrow  \Ext^1(\CF, \CF')$ factors as $\lambda \circ \psi$, according to the map of sheaves $\BC_x \cong \CF/\CF' \hookrightarrow \CV/\CF'$ and the induced arrows below:
$$
\xymatrix{\Hom(\BC_x,\CV/\CF') \ar[r]^\omega & \Hom(\CF,\CV/\CF') \\
 & \Hom(\CF,\BC_x) \ar[u]_\psi}
$$
Homomorphisms in $\text{Im }\psi$ have length 1 image, namely $\BC_x \cong \CF/\CF' \hookrightarrow \CV/\CF'$. Homomorphisms in $\text{Im }\omega$ annihilate the colength 1 sheaf $\CF' \subset \CF$. Therefore, the intersection $\text{Im }\psi \cap \text{Im } \omega$ is one-dimensional, spanned by the homomorphism $\CF \twoheadrightarrow \CF/\CF' \hookrightarrow \CV/\CF'$. Since this homomorphism goes to 0 under $\lambda$, then $\text{Im }\lambda_1$ and $\text{Im }\lambda_2$ have 0 intersection in $\text{Ext}^1(\CF,\CF')$, thus proving the claim. 

\end{proof}

\section{$K$-theory of the moduli space of sheaves}
\label{sub:k}

\medskip

\subsection{} 
\label{sub:corr}

Consider the projection maps from $\fZ$ of \eqref{eqn:tower 1} to the moduli spaces of sheaves:
\begin{equation}
\label{eqn:projection maps}
\xymatrix{
& \fZ \ar[ld]_{p_1} \ar[d]^{p_S} \ar[rd]^{p_2} & \\
\CM & S & \CM'} 
\end{equation}
and we will also write $p_{1S}$ and $p_{2S}$ for the projections from $\fZ$ to $\CM \times S$ and $\CM' \times S$, respectively. This allows us to define the following operators:
\begin{equation}
\label{eqn:def e}
\km \stackrel{e(z)}\longrightarrow \kmms \pa{z}, \qquad \qquad \qquad e(z) =  p_{2S*}\left( \delta \left(\frac {\CL}z \right) \cdot p_1^* \right)
\end{equation}
\begin{equation}
\label{eqn:def f}
K_{\CM'} \stackrel{f(z)}\longrightarrow \kms \pa{z}, \ \qquad f(z) = \frac {\det \CU}{(-z)^r q^{r-1}} \cdot  p_{1S*}\left( \delta \left(\frac {\CL}z \right) \cdot p_2^* \right)
\end{equation}
where $\delta(z) = \sum_{n \in \BZ} z^n$ denotes the delta function as a formal series. Therefore, \eqref{eqn:def e} and \eqref{eqn:def f} are formal power series of operators, which encode all powers of the tautological line bundle $\CL$ viewed as correspondences between $\CM$, $\CM'$ and $S$. \\

\begin{proposition}
\label{prop:comm rel 1}

We have the following identity of operators $\km \rightarrow K_{\CM \times S \times S}$:
\begin{equation}
\label{eqn:comm rel 1}
\zeta^S \left( \frac wz \right) e(z) e(w) = \zeta^S \left( \frac zw \right)  e(w) e(z) 
\end{equation}
where the \textbf{zeta function} associated to the surface $S$ is defined as:
\begin{equation}
\label{eqn:def zeta s}
\zeta^S (x) = \wedge^\bullet \left( - x \cdot \CO_\Delta \right) \ \in \ \kss (x) 
\end{equation}
If one replaces $e \leftrightarrow f$, then \eqref{eqn:comm rel 1} holds with the opposite product. \\

\end{proposition}

\noindent Proposition \ref{prop:comm rel 1} will be proved in Subsection \ref{sub:proof}, but before we lay the groundwork, let us explain two things about relation \eqref{eqn:comm rel 1}: how to interpret the composition of $e(z)$ and $e(w)$ (which will be done in the current paragraph), and how to make sense of the relation as an equality of bi-infinite formal series (which will be done in the next paragraph). By definition, we have the correspondences:
$$
e(z) : \km \longrightarrow K_{\CM \times S_1} \qquad \text{and} \qquad e(w) : \km \longrightarrow K_{\CM \times S_2}
$$
where we use the notation $S_1 = S_2 = S$ as a convenient way to keep track of two factors of the surface $S$ involved in the definition. By pulling back correspondences, we may think of $e(z)$ as an operator $K_{\CM \times S_2} \rightarrow K_{\CM \times S_1 \times S_2}$ acting trivially on the $S_2$ factor, which allows us to define the composition as:
$$
e(z) e(w) : \km \longrightarrow K_{\CM \times S_2} \longrightarrow K_{\CM \times S_1 \times S_2}
$$
We can take the tensor product on the target with the pull-back of $\zeta^S (w/z) \in K_{S_1 \times S_2}$, where $\Delta \hookrightarrow S_1 \times S_2$ denotes the diagonal. This gives rise to an operator:
$$
\zeta^S \left( \frac wz \right) e(z) e(w) : \km \longrightarrow K_{\CM \times S_1 \times S_2}
$$
which is precisely the left-hand side of \eqref{eqn:comm rel 1}. The right-hand side is defined analogously, by replacing $(z,S_1) \leftrightarrow (w,S_2)$ and identifying $\CM \times S_1 \times S_2 = \CM \times S_2 \times S_1$ via the permutation of the factors $S = S_1 = S_2$. Thus the important thing to keep in mind is the fact that the variables $z$ and $w$ must each correspond to the same copy of the surface $S$ in the left as in the right-hand sides of equation \eqref{eqn:comm rel 1}. \\

\noindent Let us now explain how to make sense of \eqref{eqn:comm rel 1} as an equality of bi-infinite formal series. According to Subsection \ref{sub:not}, $\zeta^S(x)$ is a rational function in $x$ with coefficients in $\kss$, so relation \eqref{eqn:comm rel 1} can be interpreted by multiplying it with the denominators of $\zeta^S(z/w)$ and $\zeta^S(w/z)$ and then equating coefficients in $z$ and $w$. Fortunately, this can be made even more explicit, as Proposition \ref{prop:codim 2} implies that:
\begin{equation}
\label{eqn:diag codim 2}
\zeta^S(x) = 1 +  \frac {[\CO_\Delta] \cdot x}{(1-x)(1-x q)} \ \in \ \kss(x)
\end{equation}
where $q = [\CK_S] \in \ks$ can be pulled back to $\kss$ via either projection (it is immaterial which, because of the factor $[\CO_\Delta]$ in \eqref{eqn:diag codim 2}). Therefore, relation \eqref{eqn:comm rel 1} should be interpreted as the following equality of operators $\km \rightarrow K_{\CM \times S \times S} \pa{z,w}$:
$$
\left[ (z - w)( zq_1 - w) (z-wq_2) + [\CO_\Delta] \cdot zw (zq-w) \right] e(z) e(w) = 
$$
\begin{equation}
\label{eqn:commy}
= \left[ (z - w)(zq_1 - w)(z-wq_2) + [\CO_\Delta] \cdot zw (z-wq) \right] e(w) e(z)
\end{equation}
(we write $q_1$ and $q_2$ for the pull-back of the canonical class on the two factors of $S \times S$, and we also write $[\CO_\Delta] q$ for either of $[\CO_\Delta] q_1 = [\CO_\Delta] q_2$). Relation \eqref{eqn:commy} is equivalent with the following collection of equalities for the coefficients \eqref{eqn:coeff series}:
$$
q_1 [e_{n+3}, e_m] - ( q_1q_2 + q_1 + 1 ) [e_{n+2}, e_{m+1}] + ( q_1 q_2 + q_2 + 1 ) [e_{n+1}, e_{m+2}] - 
$$
\begin{equation}
\label{eqn:comm rel 11}
- q_2 [e_n, e_{m+3}] = [\CO_\Delta] \cdot \left( [e_{n+1}, e_{m+2}]_{q} + [e_{m+1}, e_{n+2}]_{q} \right)
\end{equation}
for all $m,n \in \BZ$, where $[x,y]_q = xy - q yx$. Any composition of the form $e_{n+i} e_{m+j}$ or $e_{m+j} e_{n+i}$ that appears in \eqref{eqn:comm rel 11} is an operator $\km \rightarrow K_{\CM \times S_1 \times S_2}$, with:
$$
e_{n+i} : \km \rightarrow K_{\CM \times S_1} \qquad (\text{respectively } e_{m+j} : \km \rightarrow K_{\CM \times S_2})
$$
extended trivially to the $S_2$ (respectively $S_1$ factor). Therefore, the composition of operators only involves the $\km$ factor, while $K_{S_1}$ and $K_{S_2}$ behave as coefficients that do not interact with each other except through $[\CO_\Delta]$. \\

\subsection{}
\label{sub:proof}

Given sheaves $\CF, \CF'$ and a point $x \in S$, we will write $\CF' \subset_x \CF$ if $\CF' \subset \CF$ and $\CF/\CF' \cong \BC_x$. In order to prove Proposition \ref{prop:comm rel 1}, consider the following diagram:
\begin{equation}
\label{eqn:comp corr}
\xymatrix{
& & \fZ_2 \ar@/_2pc/[lldd]_{\pi_1} \ar[ld] \ar[rd] \ar@/^2pc/[rrdd]^{\pi_2} & & \\
& \fZ \times S_2 \ar[ld] \ar[rd] & & \fZ \ar[ld] \ar[rd] & \\
\CM'' \times S_1 \times S_2 & & \CM' \times S_2 & \qquad \qquad \qquad \qquad & \CM} 
\end{equation}
If the two copies of $\fZ$ on the second row parametrize pairs of sheaves $(\CF'' \subset_{x_1} \CF')$ and $(\CF' \subset_{x_2} \CF)$, respectively (denote by $S_1$ and $S_2$ the copies of $S$ where the points $x_1$ and $x_2$ lie) then the top of the diagram is defined as the derived fiber product:
\begin{equation}
\label{eqn:fz2}
\fZ_2 = \fZ \times_{\CM'} \fZ
\end{equation}
which parametrizes triples of sheaves $\{ \CF'' \subset_{x_1} \CF' \subset_{x_2} \CF, \text{ for some }x_1,x_2 \in S \}$. Rigorously speaking, $\fZ_2$ is defined in \eqref{eqn:fz2} as the projectivization of the same coherent sheaf over the first (respectively, second) factor of $\fZ \times_{\CM'} \fZ$ as the second (respectively, first) factor is over $\CM'$. The coherent sheaf in question is defined, in either \eqref{eqn:tower 1} or \eqref{eqn:tower 2}, as being quasi-isomorphic to an explicit complex of two locally free sheaves. Therefore, this defines $\fZ_2$ in \eqref{eqn:fz2} as a dg scheme. Moreover, we also have the line bundles $\CL_1$ and $\CL_2$ on $\fZ_2$, with fibers:
$$
\CL_1|_{\{ \CF'' \subset_{x_1} \CF' \subset_{x_2} \CF \}} = \CF'_{x_1}/\CF''_{x_1}
$$
$$
\CL_2|_{\{ \CF'' \subset_{x_1} \CF' \subset_{x_2} \CF \}} = \CF_{x_2}/\CF'_{x_2}
$$
Then the usual rules of composing correspondences imply that $e_n e_m$ is given by:
\begin{equation}
\label{eqn:corr fz2}
\pi_{1*} \left( \CL_1^n \CL_2^m \cdot \pi_2^* \right)
\end{equation}
Throughout the remainder of this Subsection, $\CM$, $\CM'$, $\CM''$ will refer to the three copies of the moduli space of stable sheaves where $\CF$, $\CF'$, $\CF''$ lie, respectively. \\

\begin{proof} \emph{of Proposition \ref{prop:comm rel 1}:} We refer the reader to Subsections \ref{sub:push} and \ref{sub:blow} for certain computations in $K$--theory that we will use in the course of the proof. We may use the setup therein because $\fZ$ is the projectivization of the coherent sheaf $\CU$ on $\CM \times S_2$ (since $\CU$ has a two term locally free resolution, its projectivization is defined as in Section \ref{sub:derived bundles}). Similarly, $\fZ_2$ is the projectivization of $\CU'$ on $\fZ \times S_1$:
\begin{equation}
\label{eqn:two stories}
\xymatrix{
\fZ_2 = \BP_{\fZ \times S_1} (\CU') \ar[d] & \\
\fZ \times S_1 \ar[d] & \\
\CM \times S_1 \times S_2}
\end{equation}
where $\CU'$ and $\CU$ are connected by the short exact sequence \eqref{eqn:ses universal} of sheaves on $\fZ \times S_1$:
\begin{equation}
\label{eqn:ses}
0 \rightarrow \CU' \rightarrow \CU \rightarrow \CL_2 \otimes \CO_\Delta \rightarrow 0
\end{equation}
In formula \eqref{eqn:ses}, we abuse notation in two ways in other to simplify the exposition: \\

\begin{itemize}

\item $\CO_\Gamma$ which appears in \eqref{eqn:ses universal} is the pull-back to $\fZ \times S_1$ of the structure sheaf of the diagonal $\CO_\Delta$ in $S_1 \times S_2$, and so we write ``$\CO_\Delta$" insetad of ``$\CO_\Gamma$" \\

\item we denote the line bundle $\CL_2$ on $\fZ$ and its pull-back to $\fZ \times S_1$ by the same symbol, and therefore we have $\CL_2 \otimes \CO_\Gamma = \CL_2 \otimes \Gamma_*(\CO) = \Gamma_*(\CL_2)$. 

\end{itemize}

\tab  
Recall that the line bundle $\CL_1$ equals the invertible sheaf $\CO(1)$ on the $\BP$ on the top of diagram \eqref{eqn:two stories}. Let us consider the diagram \eqref{eqn:nice nice diagram} associated to the short exact sequence \eqref{eqn:ses} on $X = \fZ \times S_1$:
\begin{equation}
\label{eqn:nice diagram proof}
\xymatrix{& Y \ar[ld]_{p'} \ar[rd]^p & \\
\BP_X(\CU') \ar[rd] & & \BP_X(\CU) \ar[ld]\\
& X &}
\end{equation}
Note that the square is not derived Cartesian. Instead, recall that $\fZ_2 = \BP_X(\CU')$ parametrizes triples $\CF'' \subset \CF'_1 \subset \CF$, while $\BP_X(\CU)$ parametrizes triples $\CF_1', \CF_2' \subset \CF$. According to \eqref{eqn:nice diagram} and \eqref{eqn:nice square}, the space $Y$ then parametrizes quadruples:
\begin{equation}
\label{eqn:quadruples}
\xymatrix{& \CF_1' \ar@{^{(}->}[rd]^{x_2}_{\CL_2} & \\
\CF'' \ar@{^{(}->}[ru]^{x_1}_{\CL_1} \ar@{^{(}->}[rd]_{x_2}^{\bar{\CL}_2}
 & & \CF \\
& \CF_2' \ar@{^{(}->}[ru]_{x_1}^{\bar{\CL}_1} &}
\end{equation}
where each inclusion is a length 1 sheaf, with the support points $x_1$ or $x_2$, depending on the label on the arrows. We also write $\CL_1, \CL_2, \bar{\CL}_1, \bar{\CL}_2$ for the corresponding quotient line bundles, as in diagram \eqref{eqn:quadruples}. The maps $p'$ and $p$ in \eqref{eqn:nice diagram proof} are given by forgetting $\CF_2'$ and $\CF''$, respectively. According to Proposition \ref{prop:blow up}, the map $p'$ is explicitly the projectivization:
\begin{equation}
\label{eqn:proj y}
Y = \BP_{\fZ_2} \left( \CL_1 ``\oplus" \CL_2 \otimes \CO_\Delta \right) 
\end{equation} 
where the notation $``\oplus"$ stands for a non-trivial extension of the line bundle $\CL_1$ with $\CL_2 \otimes \CO_\Delta$. Moreover, the line bundle $\CO(1)$ of the projectivization \eqref{eqn:proj y} is precisely $\bar{\CL}_1$ in \eqref{eqn:quadruples}. Then we may invoke Proposition \ref{prop:general proj} to obtain:
\begin{equation}
\label{eqn:chosen push}
p'_* \Big[ \bar{\CL}_1^n \CL_2^m (\bar{\CL}_1 - \CL_2)(\bar{\CL}_1 q_{(1)} - \CL_2)(\bar{\CL}_1 - \CL_2 q_{(2)}) \Big] = 
\end{equation}
$$
= \i z^n \CL_2^m (z - \CL_2)(z q_{(1)} - \CL_2)( z - \CL_2 q_{(2)}) \cdot \frac { \zeta \left( \frac {\CL_2}z \right) }{\left(1- \frac {\CL_1} z\right)}
$$
where $q_{(1)}$ and $q_{(2)}$ denote the canonical classes (in $K$--theory) of the two factors of $S_1 \times S_2$, and their pull-backs to $\fZ_2$ and $Y$. We can apply \eqref{eqn:diag codim 2} to rewrite \eqref{eqn:chosen push} as:
$$
= \i \frac {z^n \CL_2^m(z q_{(1)} - \CL_2)}{\left(1- \frac {\CL_1} z\right)} \Big[ (z - \CL_2)(z - \CL_2 q_{(2)}) + [\CO_\Delta] z \CL_2 \Big] =
$$
$$
= \CL_1^n \CL_2^m (\CL_1 q_{(1)} - \CL_2)\Big[ (\CL_1 - \CL_2)(\CL_1 - \CL_2 q_{(2)}) + [\CO_\Delta] \CL_1 \CL_2 \Big]
$$
where the last equality follows from the fact that the only pole of the integral, apart from $0$ and $\infty$, is $z = \CL_1$. The right-hand side is a class on $\fZ_2$, which according to \eqref{eqn:corr fz2}, precisely produces the operator $\km \rightarrow K_{\CM'' \times S_1 \times S_2}$ that appears in the left-hand side of \eqref{eqn:commy}. However, the space $Y$ parametrizing quadruples \eqref{eqn:quadruples} is symmetric in $\CF_1'$ and $\CF_2'$, up to replacing $x_1 \leftrightarrow x_2$, $\CL_1 \leftrightarrow \bar{\CL}_2$, $\CL_2 \leftrightarrow \bar{\CL}_1$. Since up to sign, so is the class in the left-hand side of \eqref{eqn:chosen push}, we conclude that the left-hand side of \eqref{eqn:commy} is antisymmetric, which is precisely what equality \eqref{eqn:commy} states.

\end{proof}

\subsection{}
\label{sub:enter h}

Recall that $q = [\CK_S] \in \ks$. To complete the picture given by the operators of \eqref{eqn:def e} and \eqref{eqn:def f}, let us define the operators of multiplication:
$$
h^\pm(z) : \km \longrightarrow \kms[[z^{\mp 1}]],
$$
\begin{equation}
\label{eqn:def h}
h^\pm(z) = \text{multiplication by } \wedge^\bullet \left( \frac {\CU (q^{-1}-1)}z  \right)
\end{equation}
where we expand the currents $h^+(z)$ and $h^-(z)$ in opposite powers of $z$:
\begin{equation}
\label{eqn:coeff h}
h^+(z) = \sum_{n = 0}^\infty \frac {h^+_n}{z^n} \qquad \qquad h^-(z) = \sum_{n = 0}^\infty \frac {h^-_n}{z^{-n}}
\end{equation}
Note that $h^+_0 = 1$ and $h^-_0 = q^{-r}$, where $r$ denotes the rank of our sheaves. \\

\begin{proposition}
\label{prop:comm rel 2}

We have the following commutation relations:
\begin{equation}
\label{eqn:comm rel 2}
\zeta^S \left( \frac wz \right) e(z) h^\pm(w) = \zeta^S \left( \frac zw \right) h^\pm(w) e(z) 
\end{equation}
of operators $\km \rightarrow \kmss \pa{z} [[w^{\mp 1}]]$ and the opposite relation with $e \leftrightarrow f$.

\end{proposition}

\noindent We make sense of \eqref{eqn:comm rel 2} by expanding first in $w$ and then in $z$, so one may translate it into a collection of commutation relations between the operators $e_n$ and $h_m^\pm$, i.e.:
$$
e_n h^+_1 + e_{n+1} [\CO_\Delta^\vee] = h_1^+ e_n + e_{n+1} [\CO_\Delta]
$$
$$
e_n  h^+_2 + e_{n+1}  h^+_1 [\CO_\Delta^\vee] + e_{n+2} [S^2 \CO_\Delta^\vee] = h_2^+  e_n + h_1^+  e_{n+1} [\CO_\Delta] + e_{n+2} [S^2 \CO_\Delta]
$$
and so on, for all $n \in \BZ$. The corresponding relations for $h^-_m$ are analogous. \\


\begin{proof} By definition, the left-hand side of \eqref{eqn:comm rel 2} is given by the $K$--theory class:
\begin{equation}
\label{eqn:compare}
\wedge^\bullet \left(- \frac wz \cdot \CO_\Delta \right) \delta \left(\frac {\CL}z \right) \wedge^\bullet \left( \frac {\CU(q^{-1}-1)}w \right)
\end{equation}
on $\fZ \times S$, while the right-hand side of \eqref{eqn:comm rel 2} is given by the $K$--theory class:
\begin{equation}
\label{eqn:compare 2}
\wedge^\bullet \left(- \frac zw \cdot \CO_\Delta \right) \delta \left(\frac {\CL}z \right) \wedge^\bullet \left( \frac {\CU'(q^{-1}-1)}{w} \right)
\end{equation}
on $\fZ \times S$. Note that \eqref{eqn:ses universal} implies that $[\CU] = [\CU'] + [\CL] \cdot [\CO_\Delta]$ (here we note that the sheaf $\CO_\Gamma$ on $\fZ \times S$ matches the pull-back of $\CO_\Delta$ on $S \times S$, and we abuse notation by writing $\CL$ both for the tautological line bundle on $\fZ$ and for its pull-back to $\fZ \times S$). Then the expression \eqref{eqn:compare 2} equals: 
$$
\wedge^\bullet \left(- \frac zw \cdot \CO_\Delta \right) \delta \left(\frac {\CL}z \right) \wedge^\bullet \left( \frac {\CU(q^{-1}-1)}{w} \right) \wedge^\bullet \left(- \frac {\CL(q^{-1}-1)}w \cdot \CO_\Delta \right) = 
$$
$$
= \wedge^\bullet \left(- \frac zw \cdot \CO_\Delta \right) \delta \left(\frac {\CL}z \right) \wedge^\bullet \left( \frac {\CU(q^{-1}-1)}{w} \right) \wedge^\bullet \left(- \frac {z(q^{-1}-1)}w \cdot \CO_\Delta \right)
$$
where the equality uses the fundamental property \eqref{eqn:fundamental property} of the $\delta$ function. The right-hand side of the expression above is equal to \eqref{eqn:compare} once we use the relation: 
\begin{equation}
\label{eqn:flip}
\wedge^\bullet\left(-\frac wz \cdot \CO_\Delta \right) = \wedge^\bullet\left(-\frac z{wq} \cdot \CO_\Delta \right)
\end{equation}
itself a consequence of \eqref{eqn:dual wedge} and $[\CO_\Delta] = q [\CO_\Delta^\vee]$.


\end{proof}

\begin{proposition}
\label{prop:comm rel 3}

We have the following commutation relation:
\begin{equation}
\label{eqn:comm rel 3}
[e(z), f(w)] = \delta \left( \frac zw \right) \Delta_* \left( \frac {h^+(z) - h^-(w)}{1-q^{-1}} \right)
\end{equation}
where the right-hand side denotes the operator of multiplication with a certain class on $\CM \times S$, as in \eqref{eqn:def h}, followed by the diagonal map $\Delta : S \hookrightarrow S \times S$. \\

\end{proposition}

\begin{remark} We note that the object inside $\Delta_*$ in \eqref{eqn:comm rel 3} is an actual $K$--theory class, in spite of the denominator $1-q^{-1}$. More specifically, the coefficients of relation \eqref{eqn:comm rel 3} in $z$ and $w$ give rise to the family of relations:
$$
[ e_n, f_m ] = \frac 1{1-q^{-1}} \cdot 
\begin{cases}
h^+_{n+m} & \text{if } n+m > 0 \\
h_0^+ - h_0^- & \text{if } n+m = 0 \\
- h^-_{-n-m} & \text{if } n+m < 0 
\end{cases}
$$
The fact that the operators in the bracket in the right-hand side are multiples of $1-q^{-1}$ follows from the definition in \eqref{eqn:def h} and the fact that $h^+_0 = 1$ and $h^-_0 = q^{-r}$. \\

\end{remark}

\begin{proof} The compositions $e(z) f(w)$ and $f(w) e(z)$ are given by the correspondences:
\begin{equation}
\label{eqn:zoe}
\frac {\det \CU}{(-w)^r q^{r-1}} \cdot (\rho^+_* \text{ and } \rho^-_*) \left[ \delta \left( \frac {\CL_1}z \right) \delta \left( \frac {\CL_2}w \right) \right] 
\end{equation}
respectively, where we consider the following derived fiber products $\fZ \times_{\CM'} \fZ$:
\begin{equation}
\label{eqn:2+}
\fZ_2^+ = \{ \CF \subset \CF' \supset \CF'' \text{ such that } \CF'/\CF \cong \BC_{x_1} \text{ and } \CF'/\CF'' \cong \BC_{x_2} \}
\end{equation}
\begin{equation}
\label{eqn:2-}
\fZ_2^- = \{ \CF \supset \CF' \subset \CF'' \text{ such that } \CF/\CF' \cong \BC_{x_2} \text{ and } \CF''/\CF' \cong \BC_{x_1} \}
\end{equation}
We let $\CL_1$, $\CL_2$ denote the line bundles on $\fZ_2^+$, $\fZ_2^-$ that parametrize the one dimensional quotients denoted by $\BC_{x_1}$, $\BC_{x_2}$ in either \eqref{eqn:2+} or \eqref{eqn:2-}. Finally, define:
$$
\rho^\pm : \fZ_2^\pm \longrightarrow \CM \times \CM'' \times S \times S
$$
to be the maps which remember $\CF, \CF'', x_1, x_2$. If we were tracking the connected components of the moduli spaces $\CM$ and $\CM''$ (which, we recall, are indexed by the second Chern class $c_2$) we ought to replace the codomain of the maps $\rho^\pm$ by $\sqcup_{c_2 \in \BZ} \ \CM_{(r,c_1,c_2)} \times \CM_{(r,c_1,c_2)} \times S \times S$. The key observation is the following: \\

\begin{claim}

The dg schemes $\fZ_2^\pm$ are isomorphic on the complement of the diagonal: 
\begin{equation}
\label{eqn:marc}
\CM \times S \hookrightarrow \CM \times \CM'' \times S \times S
\end{equation}
This isomorphism sends the bundles $\CL_1$, $\CL_2$ on $\fZ_2^+$ to the bundles $\CL_1$, $\CL_2$ on $\fZ_2^-$.

\end{claim}

\tab 
Indeed, the isomorphism is given by the following obviously inverse assignments: 
\begin{equation}
\label{eqn:ass 1}
\fZ_2^- \ni (\CF \supset \CF' \subset \CF'') \mapsto (\CF \subset \CF \cup \CF'' \supset \CF'') \in \fZ_2^+
\end{equation}
\begin{equation}
\label{eqn:ass 2}
\fZ_2^+ \ni (\CF \subset \CF' \supset \CF'') \mapsto (\CF \supset \CF \cap \CF'' \subset \CF'') \in \fZ_2^-
\end{equation}
These formulas should be read by picturing all the sheaves involved as subsheaves of their double dual $\CV$ (the double duals of the torsion-free sheaves $\CF,\CF',\CF''$ as in \eqref{eqn:ass 1}, \eqref{eqn:ass 2} are stable locally free sheaves on $S$, all uniquely isomorphic up to scalar multiple, and thus naturally identifiable with each other). For example, in \eqref{eqn:ass 1} we take two subsheaves $\CF, \CF'' \subset \CV$ whose intersection is colength 1 in each of them, and claim that the union $\CF \cup \CF'' \subset \CV$ contains $\CF,\CF''$ as colength 1 subsheaves. For these assignments to be well-defined, it is important that $\CF \neq \CF''$ as subsheaves of their double dual, which is equivalent to requiring that $\CF \not \cong \CF''$.

\tab 
As a consequence of the claim, and the excision long exact sequence in algebraic $K$--theory, we conclude that the commutator $[e(z), f(w)]$ is given by a class supported on the diagonal \eqref{eqn:marc} of $\CM \times \CM'' \times S \times S$. Therefore, the commutator acts as:
$$
[e(z), f(w)] = \text{multiplication by } \Delta_*(\gamma)
$$
for some $\gamma \in \kms$. Above, we abuse notation by writing $\Delta : \CM \times S \rightarrow \CM \times S \times S$ instead of the more appropriate notation $\text{Id}_{\CM} \times \Delta$. We note that $\Delta_*$ is injective, because it has a left-inverse given by projection onto the first component $S \times S \rightarrow S$. By applying the formula above to the unit class, we obtain:
\begin{equation}
\label{eqn:anto}
[e(z), f(w)] \cdot 1 = \Delta_*(\gamma)
\end{equation}
and it remains to prove that:
\begin{equation}
\label{eqn:remains}
\gamma = \frac {\delta \left(\frac zw \right)}{1-q^{-1}} \left[ \wedge^\bullet \left( \frac {\CU(q^{-1}-1)}z \right) - \wedge^\bullet \left( \frac {\CU(q^{-1}-1)}w \right) \right] 
\end{equation}
where the first wedge product is expanded in non-positive powers of $z$ and the second wedge product is expanded in non-negative powers of $w$. To prove \eqref{eqn:remains}, let us combine Definition \ref{def:tower} and Proposition \ref{prop:tower} with Proposition \ref{prop:general proj}:
\begin{align*}
&e(z) \cdot 1 = \wedge^\bullet \left(\frac {zq}{\CU}  \right) \Big|_{0 - \infty} = - \wedge^\bullet \left(\frac {zq}{\CU}  \right) \b \\
&f(w) \cdot 1 = \frac {\det \CU'}{(-w)^r q^{r-1}} \wedge^\bullet \left( - \frac {\CU'}w \right) \b = q^{1-r} \wedge^\bullet \left( - \frac w{\CU'} \right) \b
\end{align*}
(the line bundle $\CL$ on $\fZ$ is identified with $\CO(1)$ and $\CO(-1)$ in the projectivizations \eqref{eqn:tower 1} and \eqref{eqn:tower 2}, respectively, and then we apply \eqref{eqn:push taut gen 1} and \eqref{eqn:push taut gen 2}). Then:
$$
e(z)f(w) \cdot 1 = e(z) \cdot q^{1-r}\wedge^\bullet \left(- \frac w{\CU'} \right) \b
$$
Note that $e(z)$ is given by the correspondence $\fZ = \{\CF \subset \CF'\}$, with the sheaves $\CF$ and $\CF'$ associated to the target and the domain of the correspondence, respectively. Since $[\CU'] = [\CU] + [\CL_1] \cdot [\CO_{\Delta}]$ where $\CL_1$ is the tautological line bundle, we have:
$$
e(z)f(w) \cdot 1 = - q_{(2)}^{1-r} \wedge^\bullet \left( - \frac w{\CU_2} \right) \wedge^\bullet \left( - \frac wz \cdot \CO_\Delta^\vee \right) e(z) \cdot 1  \ \b = 
$$
\begin{equation}
\label{eqn:rr1}
= - q_{(2)}^{1-r} \wedge^\bullet \left(\frac {zq_{(1)}}{\CU_1}  \right) \wedge^\bullet \left( - \frac w{\CU_2} \right) \wedge^\bullet \left( - \frac zw \cdot \CO_\Delta \right) \b
\end{equation}
In the formula above, we write $\CU_1$ (respectively $\CU_2$) for the universal sheaves on the product $\CM \times S \times S$ that are pulled back from the product of the first and second (respectively first and third) factors. Similarly, $q_{(1)}$ and $q_{(2)}$ denote the canonical class $[\CK_S]$ pulled back from the first and second (respectively) factor of $S \times S$. 
By an analogous computation, we have:
\begin{equation}
\label{eqn:rr2}
f(w)e(z) \cdot 1 = - q_{(2)}^{1-r} \wedge^\bullet \left(\frac {zq_{(1)}}{\CU_1}  \right) \wedge^\bullet \left( - \frac w{\CU_2} \right) \wedge^\bullet \left( - \frac zw \cdot \CO_\Delta \right) \b
\end{equation}
Comparing the right-hand sides of \eqref{eqn:rr1} and \eqref{eqn:rr2}, one is tempted to conclude that the two expressions are equal. However, note that the order in which the operators are applied means that in \eqref{eqn:rr1} we first compute the residue in $w$ and then the residue in $z$, while in \eqref{eqn:rr2} we compute the residues in the opposite order. Therefore, the difference between \eqref{eqn:rr1} and \eqref{eqn:rr2} comes from the poles in $z/w$:
\begin{multline}
\label{eqn:dire}
[e(z), f(w)] \cdot 1 = \sum_{\alpha \text{ pole of } \zeta^S} \underset{x = \alpha}{\text{Res}} \ \zeta^S(x) \cdot \\ \delta \left(\frac z{w\alpha} \right) q_{(2)}^{1-r} \wedge^\bullet \left(\frac {zq_{(1)}}{\CU_1}  \right) \wedge^\bullet \left( - \frac w{\CU_2} \right)   \b
\end{multline}
Recall that $\zeta^S(x) = \wedge^\bullet( - x \cdot \CO_\Delta)$ was defined in \eqref{eqn:def zeta s}, while in \eqref{eqn:diag codim 2} we established that the only poles of $\zeta^S$ are $\alpha = 1$ and $\alpha = q^{-1}$. Moreover, the corresponding residue is a multiple of $[\CO_\Delta]$, so we may write $q_{(1)} = q_{(2)} = q$; thus, \eqref{eqn:dire} becomes:
$$
[e(z), f(w)] \cdot 1 = \delta \left(\frac zw \right) q^{1-r} \wedge^\bullet \left(\frac {w q}{\CU_1}  \right) \wedge^\bullet \left( - \frac w{\CU_2} \right)  \frac {[\CO_\Delta]}{q-1} \ \b - 
$$
$$
- \delta \left(\frac {zq}{w} \right) q^{1-r} \wedge^\bullet \left(\frac {w}{\CU_1}  \right) \wedge^\bullet \left( - \frac w{\CU_2} \right)   \frac {[\CO_\Delta]}{q-1}  \ \b
$$
Because of the factor $[\CO_\Delta]$, we may identify $\CU_1=\CU_2 = \CU$, and so the second line vanishes (a constant function does not have any poles between 0 and $\infty$). Therefore:
\begin{align*}
[e(z), f(w)] \cdot 1 &= \delta \left( \frac zw \right) q^{1-r } \frac {[\CO_\Delta]}{q-1} \wedge^\bullet \left(\frac {w(q-1)}{\CU}  \right) \b \\ &= \delta \left( \frac zw \right) \frac {[\CO_\Delta]}{1-q^{-1}} \wedge^\bullet \left(\frac {\CU(q^{-1}-1)}{w}  \right) \b
\end{align*}
which combined with \eqref{eqn:anto} implies \eqref{eqn:remains}.

\end{proof}

\subsection{}
\label{sub:collect}

Let us now put together the results of Propositions \ref{prop:comm rel 1}, \ref{prop:comm rel 2} and \ref{prop:comm rel 3}. Note that the commutation relations in question only depend on two pieces of data, namely:
$$
\zeta^S(x) = \wedge^\bullet(-x \cdot \CO_\Delta) \in \kss (x)
$$
and the class of the canonical bundle $q = [\CK_S] \in \ks$. With this in mind, we proved that the operators $e(z), f(z), h^\pm(z)$ of \eqref{eqn:def e}, \eqref{eqn:def f}, \eqref{eqn:def h} satisfy the relations:
\begin{equation}
\label{eqn:comm rel 1 intro}
\zeta^S \left( \frac wz \right) e(z)e(w) \ \ = \ \ \zeta^S \left(\frac zw \right) e(w)e(z)
\end{equation}
\begin{equation}
\label{eqn:comm rel 2 intro}
\zeta^S \left( \frac wz \right) e(z)h^\pm(w) = \zeta^S \left(\frac zw \right) h^\pm(w)e(z)
\end{equation}
\begin{equation}
\label{eqn:comm rel 3 intro}
[e(z), f(w)] = \delta \left( \frac zw \right) \Delta_* \left( \frac {h^+(z) - h^-(w)}{1-q^{-1}} \right)
\end{equation}
The algebra described by relations \eqref{eqn:comm rel 1 intro}, \eqref{eqn:comm rel 2 intro}, \eqref{eqn:comm rel 3 intro} will be the blueprint for the universal shuffle algebras we will define in Section \ref{sec:shuffle}. Let us show how the relations above restrict to the diagonal. To this end, recall that:
\begin{equation}
\label{eqn:restriction}
[\CO_\Delta] \Big|_\Delta = [\CO_S] - [\Omega_S^1] + [\CK_S] \in \ks
\end{equation}
where $\Delta : S \hookrightarrow S \times S$ denotes the diagonal. Therefore, it follows that:
$$
\zeta^S(x) \Big|_\Delta = \frac {(1-xa)(1-xb)}{(1-x)(1-xq)}
$$
where $a,b$ are the Chern roots of $\Omega_S^1$, and therefore $q = ab$. Thus we conclude that, after restricting to the diagonal, the algebra generated by $e(z), f(z), h^\pm(z)$ subject to relations \eqref{eqn:comm rel 1 intro}, \eqref{eqn:comm rel 2 intro}, \eqref{eqn:comm rel 3 intro} is nothing but the integral form (that is, over the ring $\ks$ instead of over the field $\BQ(a,b)$) of the Ding-Iohara-Miki algebra:
\begin{equation}
\label{eqn:dim 1}
(z-wa)(z-wb) \left(z - \frac wq \right) e(z)e(w) = \left( z - \frac wa \right) \left(z - \frac wb \right) (z - wq) e(w)e(z)
\end{equation}
\begin{equation}
\label{eqn:dim 2}
\frac {(z-wa)(z-wb)}{(z-w)(z-wq)} e(z)h^\pm(w) \ = \ \frac {(za-w)(zb-w)}{(z-w)(zq-w)} h^\pm(w)e(z)
\end{equation}
\begin{equation}
\label{eqn:dim 3}
[e(z), f(w)] = \delta \left( \frac zw \right) (1-a)(1-b) \left( \frac {h^+(z) - h^-(w)}{1-q^{-1}} \right)
\end{equation}
For background on this algebra, we refer the reader to \cite{DI}, \cite{FJMM}, \cite{FT}, \cite{M}. Note that much of the existing literature on the subject is done over the field $\BQ(a,b)$, which is quite different from our geometric setting, where the ring $\ks$ has zero-divisors. \\

\subsection{}
\label{sub:computations}

Let us now give explicit computations for the operators $e(z)$, $f(z)$, $h^\pm(z)$, under the additional Assumption B from \eqref{eqn:assumption b}. As an application, we will give a computational proof of Proposition \ref{prop:comm rel 1}. Consider the bilinear form:
\begin{equation}
\label{eqn:inner}
( \cdot, \cdot ) : \ks \otimes \ks \rightarrow \BZ \qquad \qquad (x, y) = \chi(S, x \otimes y)
\end{equation}
Since we assume the diagonal $\Delta \hookrightarrow S \times S$ is decomposable, it can be written as:
\begin{equation}
\label{eqn:diagonal}
[\CO_\Delta] = \sum_{i} l_i \boxtimes l^i \in \kss
\end{equation}
for some $l_i, l^i \in \ks$, where $a \boxtimes b$ refers to the exterior product $p_1^* a \otimes p_2^* b$. If \eqref{eqn:diagonal} happens, then $\{l_i\}$ and $\{l^i\}$ are dual bases with respect to the inner product \eqref{eqn:inner}. Moreover, according to Theorem 5.6.1 of \cite{CG}, we have Kunneth decompositions $\kss = \ks \boxtimes \ks$ and $K_{\CM \times S} = \km \boxtimes \ks$ which we will tacitly use from now on. \\

\begin{example}
\label{ex:basic plane}

When $S = \BP^2$, Beilinson considered the following resolution of the structure sheaf of the diagonal:
$$
0 \longrightarrow \CO(-2) \boxtimes \wedge^2 Q \longrightarrow \CO(-1) \boxtimes Q \longrightarrow \CO \boxtimes \CO \longrightarrow \CO_\Delta\longrightarrow 0
$$ 
where $\CO, \CO(-1), \CO(-2)$ are the usual line bundles on $\BP^2$, while $Q = \Omega_{\BP^2}^1(1)$ . Then:
\begin{equation}
\label{eqn:diag plane}
[\CO_\Delta] = [\CO] \boxtimes [\CO] - [\CO(-1)] \boxtimes [Q] + [\CO(-2)] \boxtimes [\wedge^2 Q]
\end{equation}
Alternatively, it is straightforward to check that $[\CO], - [Q], [\wedge^2 Q]$ is the dual basis to $[\CO], [\CO(-1)], [\CO(-2)]$ with respect to the bilinear form \eqref{eqn:inner}. 

\end{example}

\tab 
Assumption B also entails the fact that the universal bundle of the moduli space $\CM$ of stable = semistable sheaves on $S$ with invariants $(r,c_1)$ can be written as:
\begin{equation}
\label{eqn:kunneth}
[\CU] = \sum_i [\CT_i] \boxtimes l^i \ \quad \in \km \boxtimes \ks
\end{equation}
where $\CT_i$ are certain locally free sheaves on $\CM$. One may rewrite this relation as:
\begin{equation}
\label{eqn:taut from univ}
[\CT_i] = p_{1*} \left([\CU] \otimes p_2^* l_i \right) \qquad \in \km
\end{equation}
since $l_i$ and $l^i$ are dual bases with respect to the Euler characteristic pairing \eqref{eqn:inner}, and $p_1 : \CM \times S \rightarrow \CM$ and $p_2 : \CM \times S \rightarrow S$ denote the standard projections. \\

\begin{example}
\label{ex:monad plane}

Assume $S = \BP^2$ and $\gcd(r,c_1) = 1$, $-r < c_1 < 0$. The latter condition is more like a normalization than a restriction, as the moduli space remains unchanged under changing $c_1 \mapsto c_1+r$, which amounts to tensoring sheaves $\CF \mapsto \CF(1)$. As a consequence of this normalization, we have:
$$
H^j(\BP^2, \CF(-i)) = 0 \ \qquad \forall \ i \in \{0,1,2\}, \ j \in \{0,2\}
$$
for any stable sheaf $\CF$. Therefore, the derived direct images:
\begin{equation}
\label{eqn:taut 0}
\CT_i = R^1p_{1*}(\CU \otimes p_2^*\CO(-i)) \qquad \qquad \forall \ i \in \{0,1,2\}
\end{equation}
are locally free sheaves on $\CM$, whose fibers over a point $\CF$ are given by the cohomology groups $H^1(\BP^2,\CF(-i))$. Beilinson proved that there exists a monad:
\begin{equation}
\label{eqn:monad}
\CT_2 \boxtimes \wedge^2 Q \hookrightarrow \CT_1 \boxtimes Q \twoheadrightarrow \CT_0 \boxtimes \CO
\end{equation}
on $\CM \times S$, meaning a chain complex with the first map injective, the last map surjective, and the middle map having cohomology equal to the universal sheaf $\CU$. Therefore, we have the following explicit decomposition of the $K$--theory class of the universal sheaf in $\kms = \km \boxtimes \ks$:
\begin{equation}
\label{eqn:k explicit}
[\CU] = - [\CT_0] \boxtimes [\CO] + [\CT_1] \boxtimes [Q] - [\CT_2] \boxtimes [\wedge^2 Q]
\end{equation}
Compare \eqref{eqn:k explicit} with \eqref{eqn:diag plane}. Formula \eqref{eqn:k explicit} establishes the fact that Assumption B applies to $\BP^2$. Historically, monads were also used by Horrocks in a related context, and we refer the reader to \cite{LePot} or \cite{OSS} for more detailed background. \\

\end{example}

\subsection{} 

The $K$--theory classes of the locally free sheaves $\CT_i$ and their exterior powers are called \textbf{tautological classes}, and we will often abuse this terminology to refer to any polynomial in such classes. Products of tautological classes are well-defined in $K$--theory due to Proposition \ref{prop:length 1}, and we can therefore consider the groups:
\begin{equation}
\label{eqn:tautological 0}
K'_\CM = \bigoplus_{c_2 = \left \lceil \frac {r-1}{2r} c_1^2 \right \rceil}^\infty K'_{\CM_{(r,c_1,c_2)}}
\end{equation}
where $K'_{\CM_{(r,c_1,c_2)}} \subset K_{\CM_{(r,c_1,c_2)}}$ is the subgroup consisting of all classes:
\begin{equation}
\label{eqn:tautological}
\Psi(...,\CT_i,...)
\end{equation}
as $\Psi$ goes over all symmetric Laurent polynomials in the Chern roots of the locally free sheaves $\CT_i$. To be more specific, if $\CT_i$ has rank $r_i$, then we may formally write its $K$--theory class as $[\CT_i] = t_{i,1}+...+t_{i,r_i}$. Even though the individual $t_{i,j}$ are not well-defined $K$--theory classes, symmetric polynomials in them are. Therefore, we think of $\Psi(...,\CT_i,...)$ in \eqref{eqn:tautological} as being a function whose inputs are the Chern roots of all the sheaves $\CT_i$, and it is required to be symmetric for all $i$ separately. \\

\noindent Recall the dg scheme $\fZ$ defined in \eqref{eqn:tower 1}, the tautological line bundle $\CL$ of \eqref{eqn:tautological line}, and the short exact sequence \eqref{eqn:ses universal}. We have the following equality in $K_{\fZ \times S}$: 
$$
[\CU] = [\CU'] + [\CL] \cdot [\CO_{\Gamma}]
$$
where $\Gamma \subset \fZ \times S$ is the graph of the projection $p_S : \fZ \rightarrow S$. We abuse notation by using the notation $\CL$ both for the tautological line bundle on $\fZ$ and its pull-back to $\fZ \times S$. Moreover, we have the locally free sheaves $\CT_i$ and $\CT_i'$ on $\fZ$ that are pulled back from the spaces $\CM$ and $\CM'$, respectively, via the maps \eqref{eqn:projection maps}. Because of formula \eqref{eqn:taut from univ}, we have the following equality of $K$--theory classes on $\fZ$:
\begin{equation}
\label{eqn:taut on z}
[\CT_i] = [\CT'_i] + [\CL] \cdot p_S^*(l_i)
\end{equation}
for all $i$ in the indexing set \eqref{eqn:diagonal}. \\

\begin{lemma}
\label{lem:op taut}

In terms of the tautological classes \eqref{eqn:tautological}, we have:
\begin{align}
&e(z) \Psi(..., \CT_i,...) = \Psi(..., \CT_i' + z \cdot l_i,...) \prod_i \wedge^\bullet \left( \frac z{\CT_i' \boxtimes l^i q^{-1}} \right) \Big|_{0 - \infty} \label{eqn:lem e} \\
&f(z) \Psi(..., \CT_i',...) = \Psi(..., \CT_i - z \cdot l_i,...) \prod_i \wedge^\bullet \left( - \frac z{\CT_i \boxtimes l^i} \right) \b \cdot q^{1-r} \label{eqn:lem f}
\end{align}
where the expressions in the right-hand side take values in $\km \boxtimes \ks$, with the $\CT_i$ lying in the first tensor factor and $l_i, l^i, q$ in the second tensor factor. We recall the notation \eqref{eqn:frac notation}, where ``dividing" by a $K$--theory class means multiplying by its dual. \\ 

\end{lemma}

\begin{proof} We will use the notation in \eqref{eqn:projection maps}. By definition, $e(z) \Psi(..., \CT_i,...)$ equals:
$$
p_{2S*} \left( \delta \left(\frac {\CL}z \right)  \Psi(..., \CT_i,...) \right) =  p_{2S*} \left( \delta \left(\frac {\CL}z \right)  \Psi(..., \CT'_i + \CL \cdot p_S^* (l_i) ,...) \right) 
$$
where the last equality is a consequence of \eqref{eqn:taut on z}. Using property \eqref{eqn:fundamental property} of the $\delta$ function and the fact that $\CT_i'$ and $l_i$ are both pulled back via $p_{2S}^*$, we have:
\begin{equation}
\label{eqn:pushy pushy}
e(z) \Psi(..., \CT_i,...) =  \Psi(..., \CT_i' + z \cdot l_i ,...) \cdot p_{2S*} \left( \delta \left(\frac {\CL}z \right) \right)
\end{equation}
To obtain the desired result, we must compute $p_{2S*}$ applied to the formal series $\delta \left(\frac {\CL}z \right)$. To this end, recall that the map $p_{2S*}$ is described as a projectivization in \eqref{eqn:tower 2}, and that the line bundle $\CL$ is the same as the Serre twisting sheaf $\CO(-1)$ with respect to the projectivization. Then formulas \eqref{eqn:tower 2} and \eqref{eqn:push taut gen 2} imply that:
$$
e(z) \Psi(..., \CT_i,...) = \Psi(..., \CT_i' + z \cdot l_i ,...) \cdot \wedge^\bullet \left(\frac z{\CU' \otimes p_S^*\CK^{-1}} \right) \Big|_{0 - \infty}
$$
Using \eqref{eqn:kunneth}, we obtain precisely \eqref{eqn:lem e}. Formula \eqref{eqn:lem f} is proved analogously.

\end{proof}

\subsection{} As a consequence of Lemma \ref{lem:op taut}, the operators $e(z)$ and $f(z)$ of \eqref{eqn:def e} and \eqref{eqn:def f} map $\km'$ to $\kms' = K_{\CM}' \boxtimes \ks$, where $K_{\CM}' \subset \km$ is the subgroup \eqref{eqn:tautological 0} of tautological classes. The same is clearly true for the operators $h^\pm(z)$ of \eqref{eqn:def h}. Subject to Assumption B, Propositions \ref{prop:comm rel 1}, \ref{prop:comm rel 2} and \ref{prop:comm rel 3} can be proved by direct computation, and we present the proof of the first of these below (the other two are analogous). However, note that we are only proving a weaker version of Proposition \ref{prop:comm rel 1}, because our argument only establishes formula \eqref{eqn:comm rel 1} for the restricted operators: 
$$
e(z) : \km' \rightarrow \kms' \pa{z}
$$

\begin{proof} \emph{of Proposition \ref{prop:comm rel 1} subject to Assumption B}: Formula \eqref{eqn:lem e} implies that:
$$
e(z) e(w) \Psi(..., \CT_i,...) = e(z) \left[ \Psi(..., \CT_i + w l_{i(2)},...) \prod_i \wedge^\bullet \left(\frac w{\CT_i \boxtimes l^i_{(2)}q_{(2)}^{-1}} \right) \right] = 
$$
$$
\Psi(..., \CT_i + z l_{i(1)} + w l_{i(2)},...) \prod_i \wedge^\bullet \left(\frac z{\CT_i \boxtimes l^i_{(1)}q_{(1)}^{-1}} \right) \wedge^\bullet \left(\frac {w}{\CT_i \boxtimes l^i_{(2)}q_{(2)}^{-1}} \right) \wedge^\bullet \left( \frac {wq}z \l_{i(1)}^\vee \boxtimes l^{i\vee}_{(2)} \right)
$$
In the right-hand side, $l_{i(1)}$ and $l_{i(2)}$ refer to the pull-backs to $\kss$ of the $K$--theory class $l_i \in \ks$ via the first and second projections, respectively. For brevity, we suppress the notation $|_{0 - \infty}$ in the right-hand side. The facts that $\wedge^\bullet$ is multiplicative and $\sum_i l_{i(1)} \boxtimes l^i_{(2)} = [\CO_\Delta] = q[\CO_\Delta^\vee]$ imply:
$$
\wedge^\bullet \left( \frac {wq}z \cdot \l_{i(1)}^\vee \boxtimes l^{i\vee}_{(2)} \right) = \wedge^\bullet \left(\frac {wq}z \cdot \CO_\Delta^\vee\right) = \wedge^\bullet \left(\frac wz \cdot \CO_\Delta \right)
$$
We thus conclude the following relation for the composition of the operators $e(z)$:
\begin{equation}
\label{eqn:double e}
e(z) e(w) \Psi(..., \CT_i,...) =
\end{equation}
$$
 = \Psi(..., \CT_i + z l_{i(1)} + w l_{i(2)},...) \prod_i \wedge^\bullet \left(\frac z{\CT_i \boxtimes l^i_{(1)} q_{(1)}^{-1}} \right)  \wedge^\bullet \left(\frac w{\CT_i \boxtimes l^i_{(2)} q_{(2)}^{-1}} \right) \wedge^\bullet \left( \frac wz \cdot \CO_\Delta \right)
$$
Comparing \eqref{eqn:double e} with the analogous formula for $z \leftrightarrow w$ and $1 \leftrightarrow 2$ implies \eqref{eqn:comm rel 1}.

\end{proof}

\subsection{}

Relation \eqref{eqn:double e} is an equality of formal series of operators. Taking integrals as in Subsection \ref{sub:convergence}, it allows us to obtain formulas for the operators \eqref{eqn:coeff series}:
$$
e_n = \int_{0 - \infty} z^n e(z) \quad : \quad \km' \rightarrow \kms'
$$
specifically:
$$
e_n \cdot \Psi(...,\CT_i,...) = \int_{0 - \infty}  z^{n} \Psi(..., \CT_i + z l_{i},...) \prod_i \wedge^\bullet \left( \frac {z}{\CT_i \boxtimes l^i q^{-1}} \right) 
$$
Composing two such relations in the variables $z_1$ and $z_2$, we obtain:
$$
e_{n_1} e_{n_2} \cdot \Psi(...,\CT_i,...) = \int_{0 - \infty}^{z_1 \prec z_2} z_1^{n_1} z_2^{n_2} \wedge^\bullet \left(\frac {z_2}{z_1} \cdot \CO_\Delta \right)
$$
$$
\Psi(..., \CT_i + z_1 l_{i(1)} + z_2 l_{i(2)},...) \prod_i \wedge^\bullet \left( \frac {z_1}{\CT_i \boxtimes l_{(1)}^i q_{(1)}^{-1}} \right) \wedge^\bullet \left( \frac {z_2}{\CT_i \boxtimes l_{(2)}^i q_{(2)}^{-1}} \right)
$$
where the notation $\int_{0 - \infty}^{z_1 \prec z_2}$ means that we take the residues at both 0 and $\infty$ first in $z_2$ and then in $z_1$ (intuitively, $z_2$ is closer to $0$ and $\infty$ than $z_1$). Iterating this computation implies that an arbitrary composition of $e_n$ operators is given by:
\begin{equation}
\label{eqn:non symmetric e}
e_{n_1}...e_{n_k} \cdot \Psi(...,\CT_i,...) = \int_{0 - \infty}^{z_1 \prec ... \prec z_k} z_1^{n_1} ... z_k^{n_k}\prod_{1 \leq i < j \leq k} \wedge^\bullet \left( \frac {z_j}{z_i} \cdot \CO_{\Delta_{ij}} \right)
\end{equation}
$$
\Psi(..., \CT_i + z_1 l_{i(1)} + ... + z_k l_{i(k)} ,...) \prod_i \wedge^\bullet \left( \frac {z_1}{\CT_i \boxtimes l_{(1)}^i q_{(1)}^{-1}} \right)...  \wedge^\bullet \left( \frac {z_k}{\CT_i \boxtimes l_{(k)}^i q_{(k)}^{-1}} \right)
$$
where $\Delta_{ij}$ is the codimension 2 diagonal in $S \times ... \times S$ corresponding to the $i$ and $j$ factors. In Remark \ref{rem:move contours} below, we will explain how to ensure that the $k$ contours that appear in \eqref{eqn:non symmetric e} can be moved around until they coincide. Once we do so, both the contours and the second line of \eqref{eqn:non symmetric e} will be symmetric with respect to the \textbf{simultaneous} permutation of the variables $z_1,...,z_k$ and the $k$ factors of $K_{S \times ... \times S}$. Therefore, the value of the integral is unchanged if we replace the rational function in $z_1,...,z_k$ on the first line of \eqref{eqn:non symmetric e} with its symmetrization:
\begin{equation}
\label{eqn:symmetrization}
\sym \ R(z_1,...,z_k) = \sum_{\sigma \in S(k)} (\sigma \cdot R)(z_{\sigma(1)},..., z_{\sigma(k)})
\end{equation}
for any rational function $R$ with coefficients in $K_{S \times ... \times S}$. Note that $\sigma \in S(k)$ acts on the coefficients of $R$ via permutation of the factors in $S \times ... \times S$. We obtain:
$$
e_{n_1}...e_{n_k} \cdot \Psi(...,\CT_i,...) = \frac 1{k!} \int^{z_1,...,z_k}_{\text{same contour}} \sym \left[ z_1^{n_1} ... z_k^{n_k}\prod_{1 \leq i < j \leq k} \wedge^\bullet \left( \frac {z_j}{z_i} \cdot \CO_{\Delta_{ij}} \right) \right]
$$
\begin{equation}
\label{eqn:symmetric e}
\Psi(..., \CT_i + z_1 l_{i(1)} + ... + z_k l_{i(k)} ,...) \prod_i \wedge^\bullet \left( \frac {z_1}{\CT_i \boxtimes l_{(1)}^i q_{(1)}^{-1}} \right)...  \wedge^\bullet \left( \frac {z_k}{\CT_i \boxtimes l_{(k)}^i q_{(k)}^{-1}} \right)
\end{equation}
All variables in \eqref{eqn:symmetric e} go over the same contour, which is specifically the difference of two circles, one surrounding $0$ and one surrounding $\infty$. As will be clear from Remark \ref{rem:move contours}, the definition of the integral in \eqref{eqn:symmetric e} is rather convoluted, and so it is useless for computations. However, it serves a very important purpose: because many rational functions in $z_1,...,z_k$ have the same symmetrization, one may use \eqref{eqn:symmetric e} to obtain linear relations between the various operators $e_{n_1}...e_{n_k}$. \\

\begin{remark}
\label{rem:move contours}

Let us explain how to change the contours from \eqref{eqn:non symmetric e} to those in \eqref{eqn:symmetric e}, or more specifically, we will show how to define the latter formula in order to match the former. Akin to \eqref{eqn:codim 2}, one can prove that:
$$
\wedge^\bullet(x \cdot \CO_\Delta) = 1 - \frac {[\CO_\Delta] \cdot x}{(1-xa)(1-xb)} 
$$
where $a$ and $b$ denote the Chern roots of $\Omega_S^1$. Therefore, let us think of $a$ and $b$ as formal variables, and note that the poles that involve $z_i$ and $z_j$ in \eqref{eqn:non symmetric e} are all of the form $z_i = a_{(i)} z_j$ or $z_i = b_{(i)} z_j$. Recall that the right-hand side of \eqref{eqn:non symmetric e} is an alternating sum of $2^k$ contour integrals. Let us focus on any one of these integrals: in it, some of the variables $z_i$ go around 0 and the other variables go around $\infty$. Call the former group of variables ``small" and the latter group ``big" for the given integral. Assume that the first line in the right-hand side of \eqref{eqn:non symmetric e} was replaced by:
\begin{equation}
\label{eqn:replacement}
z_1^{n_1} ... z_k^{n_k}  \prod^{i,j}_{\emph{small}} \Upsilon^{\esm}_{ij} \left( \frac {z_j}{z_i} \right) \prod^{i,j}_{\ebig} \Upsilon^{\ebig}_{ij} \left( \frac {z_j}{z_i} \right) \prod^{i \emph{ big}}_{j \emph{ small}} \Upsilon_{ij} \left( \frac {z_j}{z_i} \right) \prod^{j \emph{ big}}_{i \emph{ small}} \Upsilon_{ij} \left( \frac {z_j}{z_i} \right)
\end{equation}
where we define:
\begin{align*}
&\Upsilon_{ij}(x) = 1 - \frac {[\CO_{\Delta_{ij}}] \cdot x}{(1-xa_{(i)})(1-xb_{(i)})} \\
&\Upsilon^\esm_{ij}(x) = 1 - \frac {[\CO_{\Delta_{ij}}] \cdot x}{(1-xa^\esm_{(i)})(1-xb^\esm_{(i)})} \\
&\Upsilon^\ebig_{ij}(x) = 1 - \frac {[\CO_{\Delta_{ij}}] \cdot x}{(1-xa^\ebig_{(i)})(1-xb^\ebig_{(i)})}
\end{align*}
In the formula above, $a_{(i)}, b_{(i)}$ denote the Chern roots of $\Omega_S^1$ pulled back to the $i$--th factor of $S \times ... \times S$, while $a_{(i)}^\esm$, $a_{(i)}^\ebig$, $b_{(i)}^\esm$, $b_{(i)}^\ebig$ are complex parameters. We assume these parameters have absolute value $<1$ (those with superscript $\esm$) or $>1$ (those with superscript $\ebig$). Because of these assumptions, we may change the contours from \eqref{eqn:symmetric e} to \eqref{eqn:non symmetric e} without picking up any poles between the variables $z_i$ and $z_j$.

\tab  
Therefore, our prescription for defining \eqref{eqn:symmetric e} is the following: the right-hand side of relation \eqref{eqn:symmetric e} is an alternating sum of $2^k$ integrals, each of which corresponds to a partition of the set of variables $\{z_1,...,z_k\}$ into small and big variables. Replace the first line in the integrand of \eqref{eqn:symmetric e} by the expression \eqref{eqn:replacement}, and compute the integral by residues. \textbf{After} evaluating the integral, set the parameters $a_{(i)}^\ebig$, $a_{(i)}^\esm$ equal to $a_{(i)}$ and the parameters $b_{(i)}^\ebig$, $b_{(i)}^\esm$ equal to $b_{(i)}$. The integral thus defined is equal to \eqref{eqn:non symmetric e} because in the limit $z_1 \prec ... \prec z_k$, the value of the integral is a Laurent polynomial in $a_{(i)}$, $b_{(i)}$, and thus unaffected by the above procedure. \\

\end{remark}

\subsection{} Interpreting the composition $e_{n_1}...e_{n_k}$ as the integral of a symmetrization in \eqref{eqn:symmetric e} allows one to prove linear relations between these compositions, such as:
\begin{equation}
\label{eqn:funny stuff} 
[[e_{n+1}, e_{n-1}], e_n] \Big|_{\Delta} = 0 \qquad \qquad \forall \ n\in \BZ
\end{equation}
(see \cite{S} for the original context of this relation). Indeed, by \eqref{eqn:symmetric e}, relation \eqref{eqn:funny stuff} boils down to showing that:
\begin{equation}
\label{eqn:equ 1}
\sym \left[ z_1^n z_2^n z_3^n \left( \frac {z_1}{z_2} - \frac {z_2}{z_1} - \frac {z_2}{z_3} + \frac {z_3}{z_2} \right) \prod_{1\leq i < j \leq 3} \wedge^\bullet \left(\frac {z_j}{z_i} \cdot \CO_{\Delta_{ij}} \right) \right] \Big|_{\Delta} = 0
\end{equation}
The restriction of $[\CO_{\Delta_{ij}}]$ to the diagonal $\Delta \cong S \hookrightarrow S \times S \times S$ is given by \eqref{eqn:restriction}, and it is therefore independent of $i$ and $j$. Then \eqref{eqn:equ 1} is an immediate consequence of:
\begin{equation}
\label{eqn:equ 2}
\sym \left[ z_1^n z_2^n z_3^n \left( \frac {z_1}{z_2} - \frac {z_2}{z_1} - \frac {z_2}{z_3} + \frac {z_3}{z_2} \right) \prod_{1\leq i < j \leq 3} \frac {(z_i - z_j)(z_i - q z_j)}{(z_i - a z_j)(z_i - b z_j)} \right] = 0
\end{equation}
where $a$ and $b$ denote the Chern roots of $\Omega_S^1$. Equality \eqref{eqn:equ 2} is straightforward. \\

\begin{definition}
\label{def:full surface big}

Consider the abelian group:
$$
\CV_\ebig = \bigoplus_{k=0}^\infty K_{S \times ... \times S}(z_1,...,z_k)^\esym
$$
where the superscript $\esym$ means that we consider rational functions that are symmetric under the simultaneous permutation of the variables $z_i$ and the factors of the $k$--fold product $S \times ... \times S$. We endow $\CV_\ebig$ with the following associative product:
\begin{equation}
\label{eqn:gen shuffle product}
R(z_1,...,z_k) * R'(z_1,...,z_{k'}) = \frac 1{k! \cdot k'!} 
\end{equation}
$$
\esym \left[ (R \boxtimes 1^{\boxtimes k'}) (z_1,...,z_k) (1^{\boxtimes k} \boxtimes R') (z_{k+1},...,z_{k+k'}) \prod^{1 \leq i \leq k}_{1 \leq j \leq k'} \zeta^S_{ij} \left( \frac {z_i}{z_j} \right) \right]
$$
where:
$$
\zeta^S_{ij}(x) = \wedge^\bullet \left( - x \cdot \CO_{\Delta_{ij}} \right) \in K_{S \times ... \times S}(x)
$$
We call $\CV_\ebig$ the \textbf{big shuffle algebra} associated to $S$. \\

\end{definition}

\noindent In \eqref{eqn:gen shuffle product}, the rational functions $R$ and $R'$ have coefficients in the $K$--theory of $S^k$ and $S^{k'}$, respectively. We write $R \boxtimes 1^{\boxtimes k'}$ and $1^{\boxtimes k} \boxtimes R'$ for the pull-backs of these coefficients to the $K$--theory of $S^{k+k'}$ via the first and last projections, respectively. Then the second row of \eqref{eqn:gen shuffle product} is defined as the symmetrization with respect to all simultaneous permutations of the indices $z_1$,...,$z_{k+k'}$ and the $k+k'$ factors of $S^{k+k'}$. \\

\begin{definition}
\label{def:full surface small}

The subalgebra $\CV_\esm \subset \CV_\ebig$ is the $K_{S \times ... \times S}$--module generated by:
$$
z_1^{n_1} * ... * z_1^{n_k} \in \CV_\ebig 
$$
as $n_1,...,n_k$ go over $\BZ$. We call $\CV_\esm$ the \textbf{small shuffle algebra} associated to $S$. \\

\end{definition}

\begin{corollary}
\label{cor:act}

Under Assumption B, there is an action $\CV_\esm \curvearrowright \km'$ where:
\begin{equation}
\label{eqn:prescription}
\delta \left( \frac {z_1}z \right) \text{ acts as } e(z) \text{ of \eqref{eqn:def e}}
\end{equation}
The notion of ``action" refers to an abelian group homomorphism $\Phi$ given by:
$$
R(z_1,...,z_k) \in \CV_\esm \quad \stackrel{\Phi}\leadsto \quad \Phi(R) \in \emph{Hom} \left( K_{\CM}', K_{\CM \times S^k}' \right)
$$
satisfying: 
$$
\Phi(R * R') =  \left( \km' \xrightarrow{\Phi(R')} K_{\CM \times S^{k'}}' \xrightarrow{\Phi(R) \boxtimes \emph{Id}_{S^{k'}}} K_{\CM \times S^k \times S^{k'}}' \right)
$$
for all $R(z_1,...,z_k)$ and $R'(z_1,...,z_{k'})$ in $\CV_\esm$. \\

\end{corollary}

\begin{proof} Formula \eqref{eqn:prescription} completely determines the action $\CV_\sm \curvearrowright \km'$, since \eqref{eqn:symmetric e} then requires that an arbitrary $R(z_1,...,z_k) \in \CV_\sm$  acts by sending $\Psi \in \km'$ to:
$$
\frac 1{k!} \int^{z_1,...,z_k}_{\text{same contour}} \frac {R(z_1,...,z_k)}{\prod_{1 \leq i \neq j \leq k} \zeta_{ij}^S(z_i/z_j)} \Psi(..., \CT_s + \sum_{i=1}^k z_i l_{s(i)} ,...) \prod_s \prod_{i=1}^k \wedge^\bullet \left( \frac {z_i}{\CT_s \boxtimes l_{(i)}^s q_{(i)}^{-1}} \right)
$$
$\in K_{\CM \times S^k}'$, where in the formula above, the index $s$ goes over the set $I$ that appears in \eqref{eqn:diagonal decomposes}. This formula completes the proof, since it shows that any linear relations one may have between the rational functions $R(z_1,...,z_k) \in \CV_\sm$ give rise to linear relations between the corresponding homomorphisms $\Phi(R) \in \Hom(\km', K_{\CM \times S^k}')$. 

\end{proof}






\section{The universal shuffle algebra}
\label{sec:shuffle}

\medskip

\subsection{}

The purpose of this Section is to construct a universal model for the algebras that appear in Definitions \ref{def:full surface big} and \ref{def:full surface small}. For any $k \in \BN$, consider the ring:
\begin{equation}
\label{eqn:fk}
\BF_k = \BZ \left[\Delta_{ij},q^{\pm 1}_i \right]_{1 \leq i \neq j \leq k} \Big /_{\Delta_{ij} = \Delta_{ji} \text{ and relation \eqref{eqn:rel}}}
\end{equation}
subject to the relation:
\begin{equation}
\label{eqn:rel}
\Delta_{ij} \cdot (\text{any expression anti-symmetric in }i,j) = 0
\end{equation}
for all indices $i$ and $j$. Note that we have the action of the symmetric group $S(k) \curvearrowright \BF_k$ given by permuting the indices in $\Delta_{ij}$ and $q_i$, and also the homomorphisms: 
\begin{equation}
\label{eqn:first}
\iota_{\text{first}}:\BF_k \hookrightarrow \BF_{k+k'}, \qquad \qquad \qquad \Delta_{ij} \mapsto \Delta_{ij}, \ \ \quad q_i \mapsto q_i
\end{equation}
\begin{equation}
\label{eqn:last}
\iota_{\text{last}}:\BF_{k'} \hookrightarrow \BF_{k+k'}, \quad \qquad \Delta_{ij} \mapsto \Delta_{i+k,j+k}, \quad q_i \mapsto q_{i+k}
\end{equation}
For all $i \neq j$, define the following rational function with coefficients in $\BF_k$:
\begin{equation}
\label{eqn:zeta}
\zeta_{ij}(x) = 1 + \Delta_{ij} \cdot \frac x{(1-x)(1-xq_i)}
\end{equation}
Note that we may replace $q_i$ by $q_j$ in the right-hand side, in virtue of \eqref{eqn:rel}. \\

\begin{definition}
\label{def:shuffle}

The \textbf{big universal shuffle algebra} is the abelian group:
$$
\CS_\ebig = \bigoplus_{k=0}^\infty \BF_k(z_1,...,z_k)^\esym
$$
where the superscript $\esym$ means that we consider rational functions that are symmetric under the simultaneous actions $S(k) \curvearrowright \{z_1,...,z_k\}$ and $S(k) \curvearrowright \BF_k$. We endow $\CS_\ebig$ with the shuffle product:
\begin{equation}
\label{eqn:shuffle product}
R(z_1,...,z_k) * R'(z_1,...,z_{k'}) = \frac 1{k! \cdot k'!}
\end{equation}
$$
\esym \left[ \iota_{\emph{first}}(R) (z_1,...,z_k) \iota_{\emph{last}}(R') (z_{k+1},...,z_{k+k'}) \prod^{1 \leq i \leq k}_{k+1 \leq j \leq k+k'} \zeta_{ij} \left( \frac {z_i}{z_j} \right) \right]
$$
Define the \textbf{small universal shuffle algebra}:
\begin{equation}
\label{eqn:small}
\CS_\esm \subset \CS_\ebig
\end{equation}
as the direct sum over $k \geq 0$ of the $\BF_k$--modules generated by the shuffle elements:
\begin{equation}
\label{eqn:spanning set}
z_1^{n_1} * ... * z_1^{n_k}
\end{equation} 
as $n_1,...,n_k$ range over $\BZ$. \\
\end{definition}

\subsection{}
\label{sub:spec surface}

The universal shuffle algebras above may be \textbf{specialized} to an arbitrary smooth surface $S$, by which we mean that we specialize the coefficient ring to:
\begin{equation}
\label{eqn:specialization}
\BF_k \mapsto K_{S \times ... \times S}, \qquad \Delta_{ij} \mapsto [\CO_{\Delta_{ij}}], \qquad q_i \mapsto [\CK_i]
\end{equation}
where recall that $\CO_{\Delta_{ij}}$ denotes the structure sheaf of the $(i,j)$--th codimension 2 diagonal, and $\CK_i$ denotes the canonical line bundle on the $i$--th factor of the $k$--fold product $S \times ... \times S$. The most basic specialization is when $S = \BA^2$ and we consider $K$--theory equivariant with respect to parameters $a$ and $b$. Then we have:
$$
\BF_k \mapsto K^{\BC^* \times \BC^*}_{\BA^2 \times ... \times \BA^2} = \BZ[a^{\pm 1}, b^{\pm 1}]
$$
with the trivial $S(k)$ action, and specialization $\Delta_{ij} \mapsto (1-a)(1-b)$, $q_i \mapsto q = ab$. In this case, the shuffle algebra reduces to the well-known construction studied in \cite{FHHSY}, \cite{Shuf} and other papers, which are all defined with respect to the rational function:
$$
\zeta^{\BA^2}_{ij}(x) = \frac {(1-xa)(1-xb)}{(1-x)(1-xq)}
$$
The next interesting case is $S = \BP^2$, in which case we have $K_S = \BZ[t]/(1-t)^3$, where $t$ denotes the class of $\CO(-1)$. Then the specialization in question is:
$$
\BF_k \mapsto \BZ[t_1,...,t_k] \Big /_{(1-t_1)^3 = ... = (1-t_k)^3 = 0}
$$
given by $\Delta_{ij} \mapsto 1 + t_i^2 t_j + t_i t_j^2 - 3t_it_j$ and $q_i \mapsto t_i^3$. Explicitly, the multiplication \eqref{eqn:shuffle product} in the shuffle algebra is defined with respect to the rational function:
$$
\zeta^{\BP^2}_{ij}(x) = \frac {(1-xt_it_j)^3}{(1-x)(1-xt_i^2t_j)(1-xt_it_j^2)}
$$
As a final example, let us consider the minimal resolution of the $A_n$ singularity:
\begin{equation}
\label{eqn:blow}
S = \widetilde{\BA^2 / \mu_{n+1}} \stackrel{\text{blow-up}}\longrightarrow \BA^2 / \mu_{n+1}
\end{equation}
and $\mu_{n+1}$ is the group of order $n+1$ roots of unity inside $\BC^*$, acting anti-diagonally on $\BA^2$. It is more convenient to present $S$ as a hypertoric variety, specifically:
\begin{equation}
\label{eqn:minimal}
S = \Big \{(z_1,...,z_n,w_1,...,w_n) \in \BA^{2n} \text{ s.t. } z_1w_1 = ... = z_n w_n \Big\}^\circ \Big / (\BC^*)^{n-1}
\end{equation}
where the circle $\circ$ denotes the open subset of points such that $(z_i,w_j) \neq (0,0)$ for all $i < j$. The gauge torus $(\BC^*)^{n-1}$ acts by determinant 1 diagonal matrices on $z_1,...,z_n$ and by the inverse matrices on $w_1,...,w_n$. We consider the action:
$$
\BC^* \times \BC^* \curvearrowright S \qquad (a,b) \cdot (z_1,...,z_n,w_1,...,w_n) = \left( \frac {z_1}a,..., \frac {z_n}a, \frac {w_1}b,..., \frac {w_n}b \right)
$$
We will abuse notation and also write $a$ and $b$ for the elementary characters of $\BC^* \times \BC^*$, which are dual to the variables in the formula above. It is known that the $K$--theory group of $S$ is generated by the line bundles $t^{(i)} = a \CO( - z_i) = b^{-1} \CO (w_i)$ (where $\CO(-z_i)$ denotes the line bundle corresponding to the hypersurface $\{z_i = 0\}$):
$$
K^{\BC^* \times \BC^*}_S = \BZ[a^{\pm 1}, b^{\pm 1}, t^{(1)},...,t^{(n)}] \Big /_{t^{(1)}...t^{(n)} = 1 \text{ and } \left(t^{(i)} - a \right) \left( t^{(j)} - b^{-1} \right) = 0 \ \forall i<j}
$$
We leave the following Proposition to the interested reader, which one can prove for example by computing the intersection pairing \eqref{eqn:inner} and applying \eqref{eqn:diagonal}: \\

\begin{proposition}
\label{prop:diag resolution}

The $K$--theory class of the diagonal $\Delta \hookrightarrow S \times S$ is given by:
\begin{equation}
\label{eqn:diag resolution}
[\CO_\Delta] =  (1+q) \sum_{i=1}^n s^{(i)} \boxtimes \frac 1{s^{(i)}} - a \sum_{i=1}^n s^{(i-1)} \boxtimes \frac 1{s^{(i)}} - b \sum_{i=1}^n s^{(i)} \boxtimes \frac 1{s^{(i-1)}}
\end{equation}
where we write $s^{(i)} = t^{(1)}...t^{(i)}$. Note that $s^{(0)} := 1$ is equal to $s^{(n)} := t^{(1)}...t^{(n)}$ since $z_1...z_n$ is a regular function on $S$, and therefore the sums in \eqref{eqn:diag resolution} are cyclic.

\end{proposition}

\tab 
We conclude that the specialization of the universal shuffle algebra to the minimal resolution of the $A_n$ singularity involves setting:
$$
\BF_k \mapsto \BZ[a^{\pm 1}, b^{\pm 1},...,t^{(i)}_l,...]^{1\leq i \leq n}_{1\leq l \leq k} \Big /_{t^{(1)}_l...t^{(n)}_l = 1 \text{ and } \left(t^{(i)}_l - a \right) \left( t^{(j)}_l - b^{-1} \right) = 0 \ \forall i<j \ \forall l}
$$
as well as $q_i \mapsto q = ab$ and:
$$
\Delta_{ij} =  (1+q) \sum_{l=1}^n \frac {s^{(l)}_i}{s^{(l)}_j} - a \sum_{l=1}^n \frac {s^{(l-1)}_i}{s^{(l)}_j} - b \sum_{l=1}^n \frac {s^{(l)}_i}{s^{(l-1)}_j}
$$
with $s^{(l)}_i = t^{(1)}_i...t^{(l)}_i$. \\


\subsection{}

Going back to the universal shuffle algebras of Definition \ref{def:shuffle}, it is a very good problem to describe the subalgebra $\CS_\sm \subset \CS_\big$ explicitly. The only full description one has is in the specialization $S = \BA^2$, in which case we showed in \cite{Shuf} that the wheel conditions of \cite{FO} are necessary and sufficient to describe elements of the small shuffle algebra. In general, we now prove a necessary condition: \\

\begin{proposition}
\label{prop:polynomials surface}

Elements of the small shuffle algebra $\CS_\esm$ are of the form:
\begin{equation}
\label{eqn:wheel poly surface}
\frac {r(z_1,...,z_k)}{\prod_{1\leq i \neq j \leq k} (z_i - z_j q_j)}
\end{equation}
where $r$ is a Laurent polynomial with coefficients in $\BF_k$, which is symmetric with respect to the simultaneous actions $S(k) \curvearrowright \{z_1,...,z_k\}$ and $S(k) \curvearrowright \BF_k$.

\end{proposition}

\tab 
In other words, Proposition \ref{prop:polynomials surface} claims that despite the fact that the rational function $\zeta_{ij}(z_i/z_j)$ of \eqref{eqn:zeta} produces simple poles at $z_i - z_j$, such poles disappear for any element in $\CS_\sm$. This statement is not trivial. While it is true that any symmetric rational function in $z_1,...,z_k$ with constant coefficients and at most simple poles at $z_i-z_j$ is regular, this fails if the symmetric group also acts on the coefficients, e.g.:
$$
\sym \left( \frac {z_1q_1}{z_1 - z_2} \right) = \frac {z_1q_1}{z_1 - z_2} + \frac {z_2q_2}{z_2 - z_1} = \frac {z_1q_1 - z_2q_2}{z_1 - z_2}
$$

\begin{proof} By the very definition of the subalgebra $\CS_\sm$, it is enough to check the claim in the Proposition for the element $R(z_1,...,z_k) = z_1^{n_1} * ... * z_1^{n_k}$. By \eqref{eqn:shuffle product}, we have:
$$
R(z_1,...,z_k) = \sym \left[ z_1^{n_1}... z_k^{n_k} \prod_{1\leq i < j \leq k} \left(1 + \frac {\Delta_{ij} z_iz_j}{(z_j-z_i)(z_j - z_i q_i)} \right) \right] = 
$$
$$
= \sum_{\sigma \in S(k)} z_1^{n_{\sigma(1)}}...z_k^{n_{\sigma(k)}} \prod_{\sigma(i) < \sigma(j)} \left(1 + \frac {\Delta_{ij}z_iz_j}{(z_j-z_i)(z_j - z_iq_i)} \right)
$$
By clearing denominators, we see that $R(z_1,...,z_k)$ has the form \eqref{eqn:wheel poly surface}, with:
$$
r(z_1,...,z_k) = \frac {\rho(z_1,...,z_k)}{\prod_{1 \leq i < j \leq k} (z_j - z_i)}
$$
where:
$$
\rho(z_1,...,z_k) = \sum_{\sigma \in S(k)} \text{sgn}(\sigma) z_1^{n_{\sigma(1)}}...z_k^{n_{\sigma(k)}} \prod_{\sigma(i) < \sigma(j)} \Big[ (z_j-z_i)(z_j - z_iq_i) + \Delta_{ij}z_iz_j \Big] = 
$$
$$
= \sum_{\sigma \in S(k)} \text{sgn}(\sigma) z_1^{n_{\sigma(1)}}...z_k^{n_{\sigma(k)}} \prod_{i<j} \Big[ (z_j-z_i)(z_j q_j^{\delta_{\sigma(i) > \sigma(j)}} - z_iq_i^{\delta_{\sigma(i) < \sigma(j)}}) + \Delta_{ij}z_iz_j \Big]
$$
The Kronecker delta symbol $\delta_{i>j}$ takes value 1 if $i>j$ and 0 if $i<j$. To prove Proposition \ref{prop:polynomials surface}, it is enough to show that $z_2 - z_1$ divides the expression on the second line, or more specifically, that this expression vanishes when we set $z_1 = z_2 = s$:
\begin{equation}
\label{eqn:rho}
\rho(s,s,z_3,...,z_k) = \sum_{\sigma \in S(k)} \text{sgn}(\sigma) s^{n_{\sigma(1)}+ n_{\sigma(2)}} z_3^{n_{\sigma(3)}}...z_k^{n_{\sigma(k)}} \Delta_{12} z_1 z_2 \prod_{i=3}^k 
\end{equation}
$$
\Big[ (z_i-s)(z_i q_i^{\delta_{\sigma(1) > \sigma(i)}} - sq_1^{\delta_{\sigma(1) < \sigma(i)}}) + \Delta_{1i}sz_i \Big]  \Big[ (z_i-s)(z_i q_i^{\delta_{\sigma(2) > \sigma(i)}} - sq_2^{\delta_{\sigma(2) < \sigma(i)}}) + \Delta_{2i}sz_i \Big] ...
$$
where $...$ is a placeholder for terms that only involve $z_i$ with $i\geq 3$. Expression \eqref{eqn:rho} vanishes because the factor $\Delta_{12}$ on the first line ensures that:
$$
\Delta_{12} \Big[ (z_i-s)(z_i q_i^{\delta_{\sigma(1) > \sigma(i)}} - sq_1^{\delta_{\sigma(1) < \sigma(i)}}) + \Delta_{1i}sz_i \Big]  \Big[ (z_i-s)(z_i q_i^{\delta_{\sigma(2) > \sigma(i)}} - sq_2^{\delta_{\sigma(2) < \sigma(i)}}) + \Delta_{2i}sz_i \Big] = 
$$
$$
= \Delta_{12} \Big[ (z_i-s)(z_i q_i^{\delta_{\sigma(2) > \sigma(i)}} - sq_1^{\delta_{\sigma(2) < \sigma(i)}}) + \Delta_{1i}sz_i \Big]  \Big[ (z_i-s)(z_i q_i^{\delta_{\sigma(1) > \sigma(i)}} - sq_2^{\delta_{\sigma(1) < \sigma(i)}}) + \Delta_{2i}sz_i \Big]
$$
(the above equality is simply a particular application of \eqref{eqn:rel}) and so the summand of \eqref{eqn:rho} for any permutation $\rho$ cancels the corresponding summand for $\rho \circ (12)$.

\end{proof}

\subsection{}
\label{sub:full curve}

For the remainder of this Section, we will describe the analog of the universal shuffle algebras $\CS_\sm \subset \CS_\big$ when the surface is replaced by a curve (however, we make no claims about moduli spaces of stable sheaves on curves). In this case, the diagonal is a divisor and therefore:
$$
\Delta \hookrightarrow C \times C \quad \text{has }K\text{--theory class } \quad [\CO_\Delta] = 1 - (\text{line bundle}) \in K_{C \times C}
$$
Therefore, the universal coefficient ring will be defined as:
\begin{equation}
\label{eqn:fk curve}
\BF_k = \BZ[\Delta_{ij}]_{1\leq i \neq j \leq k} \Big /_{\Delta_{ij} = \Delta_{ji} \text{ and relation }\eqref{eqn:rel}} 
\end{equation}
together with the action $S(k) \curvearrowright \BF_k$ and the homomorphisms \eqref{eqn:first} and \eqref{eqn:last}. The $\zeta$ function of \eqref{eqn:zeta} will be replaced by its analogue for a curve:
\begin{equation}
\label{eqn:zeta curve}
\zeta_{ij}(x) = 1 + \Delta_{ij} \cdot \frac x{1-x}
\end{equation}
$$$$

\begin{definition}
\label{def:full curve}

Define the big and small universal shuffle algebras $\CS_\esm \subset \CS_\ebig$ by the formulas in Definition \ref{def:shuffle}, with \eqref{eqn:fk} and \eqref{eqn:zeta} replaced by \eqref{eqn:fk curve} and \eqref{eqn:zeta curve}.

\end{definition}

\tab 
As before, a \textbf{specialization} of the shuffle algebra $\CS_\big$ to a particular smooth curve $C$ will refer to the specialization of its underlying coefficient ring:
\begin{equation}
\label{eqn:spec curve}
\BF_k \mapsto K_{C \times ... \times C}, \qquad \qquad \Delta_{ij} \mapsto [\CO_{\Delta_{ij}}] \in K_{C \times ... \times C}
\end{equation}
The most basic specialization is $C = \BA^1$. It is not projective, so in order to place it in the above framework, we must replace its usual $K$--theory ring by the equivariant $K$--theory ring $K^{\BC^*}_{\BA^1} = \BZ[q^{\pm 1}]$. Explicitly, the specialization \eqref{eqn:spec curve}  is given by:
$$
\BF_k \mapsto K^{\BC^*}_{\BA^1 \times ... \times \BA^1} = \BZ[q^{\pm 1}], \qquad \Delta_{ij} \mapsto 1 - q
$$
By \eqref{eqn:small}--\eqref{eqn:spanning set}, the corresponding specialization of $\CS_\sm$ is spanned over $\BZ[q^{\pm 1}]$ by:
$$
z_1^{n_1} * ... * z_1^{n_k} = \sym \left[z_1^{n_1}...z_k^{n_k} \prod_{1\leq i < j \leq k } \frac {z_j - q z_i}{z_j - z_i} \right]
$$
When $n_1 \geq ... \geq n_k$, the right-hand side of the above expression yields (up to a constant) the well-known Hall-Littlewood polynomials. Thus $\CS_\sm$ is an integral form of the ring of symmetric polynomials in arbitrarily many variables over $\BQ(q)$. A similar phenomenon holds in the universal setting of Definition \ref{def:full curve}: \\

\begin{proposition}
\label{prop:polynomials}

Elements of the small shuffle algebra $\CS_\esm$ of Definition \ref{def:full curve} are symmetric Laurent polynomials in $z_1,...,z_k$ and $\{\Delta_{ij}\}_{1\leq i \neq j \leq k}$. \\

\end{proposition}

\noindent In Proposition \ref{prop:polynomials}, the word ``symmetric" means invariant under the action of $S(k)$ that simultaneously permutes indices in both the variables $z_i$ and the parameters $\Delta_{ij}$. That is precisely why Proposition \ref{prop:polynomials} is non-trivial. The proof follows that of Proposition \ref{prop:polynomials surface} very closely (we leave the details to the interested reader). \\

\subsection{}

Although Proposition \ref{prop:polynomials} shows that the small shuffle algebra is a subset of the abelian group of Laurent polynomials, describing this subset explicitly seems quite difficult. It is non-trivial even when we specialize to $C = \BP^1$:
\begin{equation}
\label{eqn:spec p1}
\BF_k \quad \mapsto \quad K_{\BP^1 \times ... \times \BP^1} = \BZ[t_1,...,t_k] \Big /_{(1-t_1)^2 = ... = (1-t_k)^2 = 0}
\end{equation}
where $t_i$ denotes $\CO(-1)$ on the $i$--th factor. The assignment \eqref{eqn:spec p1} is explicitly given by $\Delta_{ij} = 1 - t_i t_j$. In this specialization, elements of the shuffle algebra are Laurent polynomials in $z_1,...,z_k$ with coefficients in $t_1,...,t_k$ raised to the power 0 or 1, that are symmetric with respect to simultaneous permutation of the indices. To give a flavor of how these Laurent polynomials look like, let us work out the leading order term of the shuffle element \eqref{eqn:spanning set} in the generality of Definition \ref{def:full curve}: \\

\begin{proposition}
\label{prop:leading order}

In the algebra $\CS_\esm$ of Definition \ref{def:full curve}, we have for $n_1 \geq ... \geq n_k$:
\begin{equation}
\label{eqn:leading order}
z_1^{n_1} * ... * z_1^{n_k} = z_1^{n_1}...z_k^{n_k} \sum_{\sigma \text{ admissible}} \prod_{\sigma(i)<\sigma(j)} (1-\Delta_{ij}) + ...
\end{equation}
where ... stands for monomials in $z_1,...,z_k$ of lower lexicographic order, and a permutation $\sigma \in S(k)$ is called admissible when $i>j \text{ and }\sigma(i)<\sigma(j) \Rightarrow n_i = n_j$. \\

\end{proposition}

\begin{proof} Formula \eqref{eqn:leading order} and the proof below can be easily adapted outside the case when $n_1 \geq ... \geq n_k$, but we leave it out to avoid unnecessarily cumbersome notation. By definition, we have:
$$
z_1^{n_1} * ... * z_1^{n_k} = \sym \left[z_1^{n_1}...z_k^{n_k} \prod_{1\leq i < j \leq k } \left(1 + \frac {\Delta_{ij} z_i}{z_j - z_i} \right) \right] = 
$$
\begin{equation}
\label{eqn:jyn}
= \sum_{\sigma \in S(k)} z_1^{n_{\sigma(1)}}...z_k^{n_{\sigma}(k)} \prod_{\sigma(i) < \sigma(j)} \left(1 + \frac {\Delta_{ij} z_i}{z_j - z_i} \right)
\end{equation}
By Proposition \ref{prop:polynomials}, the right-hand side is a Laurent polynomial in $z_1,...,z_k$, and clearly, its biggest monomial in lexicographic order is precisely $z_1^{n_1}...z_k^{n_k}$. To work out the coefficient of this monomial, we must take the leading order term in the limit $|z_1| \gg ... \gg |z_k|$. Let us focus on the summand corresponding to a given permutation $\sigma$ in the right-hand side of \eqref{eqn:jyn}. The leading order monomial $z_1^{n_1}...z_k^{n_k}$ only appears when the permutation $\sigma$ is admissible, and the coefficient of this monomial is $1 - \Delta_{ij}$ if $\sigma(i) < \sigma(j)$ and $1$ otherwise.

\end{proof}

\section{Appendix}
\label{sec:appendix}

\medskip

\subsection{}
\label{sub:recall}

Let us present the definition of the moduli space $\CM$ of semistable sheaves on $S$, following Chapter 4 of \cite{HL}. Recall that we fix an ample divisor $H$ that corresponds to a line bundle henceforth denoted by $\CO_S(1)$. With respect to this line bundle, any coherent sheaf $\CF$ on $S$ has a Hilbert polynomial defined by:
\begin{equation}
\label{eqn:hilbert}
P_{\CF}(n) = \chi(S, \CF \otimes \CO_S(n)) 
\end{equation}
If $S$ is a surface and we write $r, c_1, c_2$ for the rank, first and second Chern classes of $\CF$, then the Hirzebruch-Riemann-Roch theorem gives us:
$$
P_{\CF}(n) = \frac {r H \cdot H}2 n^2 + \left(c_1 \cdot H - \frac {r H \cdot \CK_S}2 \right) n + \left(\frac {c_1\cdot c_1}2 - c_2 - \frac {c_1 \cdot \CK_S}2 + r \chi(S,\CO_S) \right)
$$
where $\CK_S$ denotes either the canonical bundle of $S$, or the corresponding divisor. One defines the reduced Hilbert polynomial of $\CF$ as:
\begin{equation}
\label{eqn:reduced}
p_{\CF}(n) = \frac {P_{\CF}(n)}r = \frac {c_1 \cdot H}r n + \frac 1r \left( \frac {c_1\cdot c_1}2 - c_2 - \frac {c_1 \cdot \CK_S}2 \right) + \text{polynomial}(n)
\end{equation}
where polynomial$(n)$ does not depend on $r, c_1, c_2$. Having defined the Hilbert polynomial, we turn to Grothendieck's $\quot$ scheme corresponding to a coherent sheaf $\CE$ on a projective scheme $S$ and a polynomial $P(n)$. Consider the functor:
$$
\sQuot(T) = \Big \{ \text{quotients } \CE_T \twoheadrightarrow \CF \text{ on }T \times S, \text{ flat over }T \text{ with } P_{\CF_t}(n) = P(n) \ \forall t \in T \Big \}
$$
where $\CE_T = \pi^*(\CE)$ under the projection map $\pi : T\times S \rightarrow S$. The property that $\CF$ is flat over $T$ implies that its fibers $\CF_t$ over all closed points $t \in T$ have the same Hilbert polynomial, which we assume is $P$. The following is due to Grothendieck: \\

\begin{theorem}
\label{thm:quot}

There exists a projective scheme $\emph{Quot}$ which \textbf{represents} the functor $T \mapsto \sQuot(T)$, which means that there exists a quotient:
$$
\CE_{\emph{Quot}} \twoheadrightarrow \tCU \qquad \text{on} \quad \emph{Quot} \times S
$$
flat over $\emph{Quot}$ with the Hilbert polynomials of its fibers equal to $P(n)$, with the following universal property. There is is a natural identification:
$$
\emph{Maps}(T, \emph{Quot}) = \sQuot(T)
$$
given by sending a map of schemes $\phi : T \rightarrow \emph{Quot}$ to the quotient $\phi^* \left( \CE_{\emph{Quot}} \twoheadrightarrow \tCU \right)$.

\end{theorem}

\tab 
We will not present the details of the construction of $\quot$, but the main idea is the following: since $S$ is projective, there exists an embedding $\iota : S \hookrightarrow \BP^N$ for some $N$. We may identify $\CE$ with $\iota_*\CE$, and this reduces the problem to constructing the Quot scheme for $S = \BP^N$. In this case, one shows that the assignment:
$$
\{\CE \twoheadrightarrow \CF\} \qquad \leadsto \qquad \{ H^0(\BP^N, \CE \otimes \CO_S(n)) \twoheadrightarrow H^0(\BP^N, \CF \otimes \CO_S(n)) \}
$$
is injective for large enough $n$. Moreover, this assignment realizes $\quot$ as a closed subscheme of the Grassmannian of $P_\CF(n)$--dimensional quotients of a $P_\CE(n)$ dimensional vector space. The ideal cutting out the closed subscheme is precisely the requirement that the $P_{\CF}(n)$--dimensional quotient is ``preserved" by multiplication with the generators of the coordinate ring of $\BP^N$, or in other words, gives rise to a sheaf on $\BP^N$. The universal quotient sheaf on the Grassmannian generates an $\CO_{\BP^N}$--module, which restricts to the universal sheaf on $\quot \times \BP^N$. \\

\subsection{}

Given two polynomials $p(n)$ and $q(n)$, we will write $p(n) \geq q(n)$ if this inequality holds for $n$ large enough. Note that this is equivalent to the fact that the coefficients of $p$ are greater than or equal to those of $q$ in lexicographic ordering. \\

\begin{definition}
\label{def:semistable}

A torsion-free sheaf $\CF$ on $S$ is called \textbf{semistable} if:
\begin{equation}
\label{eqn:semistable}
p_{\CF}(n) \geq p_{\CG}(n) \qquad \forall \ \CG \subset \CF
\end{equation}
If the inequality is strict for all proper $\CG$, then we call $\CF$ \textbf{stable}.

\end{definition}

\tab 
According to formula \eqref{eqn:reduced}, when $S$ is a surface the difference between the reduced Hilbert polynomials $p_{\CF}(n) - p_{\CG}(n)$ is linear in $n$, and therefore strict inequality in \eqref{eqn:semistable} boils down to:
\begin{equation}
\label{eqn:first condition}
\frac {c_1 \cdot H}r > \frac {c_1'\cdot H}{r'} \qquad \qquad \text{or}
\end{equation}
$$
\frac {c_1 \cdot H}r = \frac {c_1'\cdot H}{r'} \quad \text{and} \quad \frac 1r \left( \frac {c_1\cdot c_1}2 - c_2 - \frac {c_1 \cdot \CK_S}2 \right) > \frac 1{r'} \left( \frac {c_1'\cdot c_1'}2 - c_2' - \frac {c_1' \cdot \CK_S}2 \right)
$$
where $(r,c_1,c_2)$ denote the invariants of $\CF$ and $(r',c_1',c_2')$ denote the invariants of $\CG$. These properties explain the relevance of Assumption A of \eqref{eqn:assumption a}: if $\gcd(r, c_1 \cdot H) = 1$, then the second option above cannot happen for any proper subsheaf $\CG \subset \CF$. Therefore, a sheaf $\CF$ under Assumption A is stable if and only if it is semistable.

\tab 
Whenever $\CF' \subset \CF$ are sheaves on $S$ whose quotient is the skyscraper sheaf above some point $x \in S$, we will say that $\CF$ and $\CF'$ are ``Hecke modifications" of each other. The following observation will be very important for our purposes. \\

\begin{proposition}
\label{prop:hecke}

Under Assumption A of \eqref{eqn:assumption a}, for any Hecke modification $0 \rightarrow \CF' \rightarrow \CF \rightarrow \BC_x \rightarrow 0$, the sheaf $\CF$ is stable if and only if $\CF'$ is stable. \\

\end{proposition}

\begin{proof} The important observation is that $\CF$ and $\CF'$ have the same rank $r$ and first Chern class $c_1$. Suppose that $\CF'$ is not stable. Then there exists a sheaf $\CG \subset \CF'$ with invariants $r'$ and $c_1'$ such that the opposite inequality to \eqref{eqn:first condition} holds:
\begin{equation}
\label{eqn:second condition}
\frac {c_1 \cdot H}r < \frac {c_1'\cdot H}{r'}
\end{equation}
Note that equality cannot happen due to Assumption A. Since $\CG$ is also a subsheaf of $\CF$, this  implies that $\CF$ is not stable. Conversely, suppose that $\CF$ is not stable. Then there exists a sheaf $\CG \subset \CF$ with invariants $r'$ and $c_1'$ such that \eqref{eqn:second condition} holds. Since the sheaf $\CG \cap \CF' \subset \CF'$ has the same invariants $r'$ and $c_1'$, then $\CF'$ is not stable. 

\end{proof}

\subsection{}
\label{sub:git}

One cannot make an algebraic variety out of all coherent sheaves on $S$, even if one fixes the Hilbert polynomial $P$. But one can construct such a variety out of the semistable sheaves (the stable sheaves will form an open subvariety), and this will be our moduli space $\CM$. The main observation (see Theorem 3.3.7 of \cite{HL}) is that there exists a large enough $n$ such that for all semistable sheaves with Hilbert polynomial $P$, the sheaf $\CF \otimes \CO_S(n)$ has no higher cohomology, and moreover the natural evaluation map:
$$
H^0(S,\CF \otimes \CO_S(n)) \otimes \CO_S(-n) \twoheadrightarrow \CF
$$
is surjective. Letting $V$ be a vector space of dimension $P(n) = \dim H^0(S,\CF \otimes \CO_S(n))$, we consider the following special case of the Quot scheme of Theorem \ref{thm:quot}:
\begin{equation}
\label{eqn:quot}
\quot = \Big\{ V \otimes \CO_S(-n) \stackrel{\phi}\twoheadrightarrow \CF \Big\}
\end{equation}
Moreover, there exists an action $GL(V) \curvearrowright \quot$ given by the tautological action on the vector space $V$, and it is easy to see that the universal family is naturally linearized in such a way that the center $\BC^* = Z(GL(V))$ acts with weight 1. Moreover, the $GL(V)$ action clearly preserves the open subsets:
\begin{align*}
&R^{ss} = \Big \{ \phi \text{ as in }\eqref{eqn:quot} \text{ s.t. } \CF \text{ is semistable and } \phi \text{ induces } V \cong H^0(\CF \otimes \CO_S(n)) \Big\} \\
&R^{s} \ = \Big \{ \phi \text{ as in }\eqref{eqn:quot} \text{ s.t. } \CF \text{ is stable and } \phi \text{ induces } V \cong H^0(\CF \otimes \CO_S(n)) \Big\}
\end{align*}
of $\quot$. Note that Assumption A implies that $R^s = R^{ss}$, but this is certainly not necessary for the construction of these moduli spaces. For large enough $m \in \BN$ consider the $GL(V)$--linearized line bundle: 
\begin{equation}
\label{eqn:line bundle}
L_m =\det  p_{1*}(\tCU \otimes p_2^*\CO_S(m))
\end{equation} 
where the projections $p_1$ and $p_2$ are as in the following diagram:
\begin{equation}
\label{eqn:diagram quot}
\xymatrix{
& \tCU \ar@{.>}[d] & \\
& \quot \times S \ar[ld]_{p_1} \ar[rd]^{p_2} & \\
\quot & & S}
\end{equation}
Let us write $R = \overline{R^{ss}} \subset \quot$ and observe that it is preserved by the $GL(V)$ action. Therefore, the setup above is that of a reductive group $G$ on a projective scheme $X$, which is endowed with a $G$--linearized ample line bundle $L$. \\\ 

\begin{definition}
\label{def:git}

A point $x \in X$ is called \textbf{semistable} if:
\begin{equation}
\label{eqn:hm}
\forall \ \lambda : \BC^* \rightarrow G, \qquad \emph{weight}_{\BC^*} \left( L|_{\lim_{t\rightarrow 0} \lambda(t)\cdot x} \right) \leq 0
\end{equation}
The point $x$ is called \textbf{stable} if the inequality is strict for all non-trivial $\lambda$. \\

\end{definition}

\noindent Condition \eqref{eqn:hm} is called the Hilbert-Mumford criterion, and is equivalent to other definitions of semistable/stable points. One way of restating the condition is that $x$ is semistable if its $G$--orbit does not have one-parameter subgroups which converge to the zero section of the line bundle $L$. The following result is key to the construction of the moduli space of semistable sheaves (see Theorem 4.3.3 of \cite{HL}): \\

\begin{lemma}
\label{lem:fundamental}

The open subsets $R^{ss}$ and $R^s$ are the loci of semistable and stable points (respectively) of the action $GL(V) \curvearrowright R$, with respect to the line bundle $L_m$.

\end{lemma}

\tab
As a consequence of Lemma \ref{lem:fundamental}, the fundamental results of geometric invariant theory (see Section 4.2 of \cite{HL} for a review) imply that there exist quotients:
\begin{equation}
\label{eqn:construction}
\CM^{ss} = R^{ss} / GL(V) \qquad \text{and} \qquad \CM^s = R^s / GL(V)
\end{equation}
which are good and geometric, respectively. According to Lemma 4.3.1 of \cite{HL}, these quotients corepresent the functors of semistable and stable sheaves on $S$, respectively. \\

\subsection{}
\label{sub:universal}

To construct the universal family $\CU$ on $\CM^s \times S$, there is only one reasonable thing one can do: descend the universal family $\tCU$ on $R^s \times S$ to the $GL(V)$--quotient. According to Theorem 4.2.15 of \cite{HL}, this is possible if and only if the stabilizers of all points under the action $GL(V) \curvearrowright R^s$ act trivially on the fibers of $\tCU$. Note that a point $\{V \otimes \CO_S(-n) \twoheadrightarrow \CF\} \in R^s \times S$ is stabilized by $g \in GL(V)$ if and only if there exists an endomorphism $\phi \in \text{End}(\CF)$ such that the following diagram commutes:
$$
\xymatrix{
V \otimes \CO_S(-n) \ar@{->>}[r] \ar[d]^{g\otimes \text{Id}} & \CF \ar[d]^\phi \\
V \otimes \CO_S(-n) \ar@{->>}[r] & \CF}
$$
Since stable sheaves are simple, the endomorphism $\phi$ can only be a constant, and this forces $g \in \BC^* = Z(GL(V))$. We conclude that the stabilizer of any point in $R^s \times S$ is the center $\BC^*$, so descent is possible if and only if the universal family is invariant under the action of the center. However, this is not true since the center acts on the universal family with weight 1. 

\tab 
Fortunately, not all is lost. As shown in Proposition 4.6.2 of \cite{HL}, one could also get a universal sheaf on $\CM^s \times S$ by descending instead the sheaf:
\begin{equation}
\label{eqn:real universal}
\tCU \otimes p_1^* (A^{-1}) \in \text{Coh}(R^s \times S)
\end{equation}
for some line bundle $A$ on $R^s$. We abuse notation and write $p_1$ and $p_2$ for the maps \eqref{eqn:diagram quot} restricted from $\quot$ to its subscheme $R^s$. If the line bundle $A$ is $GL(V)$ linearized such that the center acts with weight 1, then the center will act on the sheaf \eqref{eqn:real universal} with weight 0, and therefore descends to a universal sheaf $\CU$ on $\CM^s \times S$. 

\tab 
To construct the line bundle $A$, Chapter 4.6 of \cite{HL} assumes the existence of a $K$--theory class $[B] \in \ks$ such that:
\begin{equation}
\label{eqn:chi is 1}
\chi(S,\CF \otimes [B]) = 1
\end{equation}
for all sheaves $\CF$ with Hilbert polynomial $P$. Then we may set:
\begin{equation}
\label{eqn:def a}
A = \det p_{1*}(\tCU \otimes p_2^*B)
\end{equation}
which will have weight 1 for the $\BC^* = Z(GL(V))$ action, as required. In our case, $S$ is a surface for which Assumption A guarantees that $\gcd(r, c_1 \cdot H) = 1$, so we have the following close variant of Corollary 4.6.7 of \cite{HL}: \\

\begin{proposition}
\label{prop:choose}

Let $a,b \in \BZ$ such that $a r + b (c_1 \cdot H) = 1$. If we choose: 
$$
[B] = \left( a + b \frac {H\cdot (H+\CK_S)}2 \right) [\CO_{\emph{pt}}] + b [\CO_H]
$$
formula \eqref{eqn:chi is 1} holds for all sheaves $\CF$ with Hilbert polynomial $P$.

\end{proposition}

\tab 
Indeed, the Proposition is a consequence of the fact that $\chi(S,\CF \otimes [\CO_{\pt}]) = r$ and:
$$
\chi(S,\CF \otimes [\CO_H]) = c_1 \cdot H - r \cdot \frac {H\cdot (H+\CK_S)}2
$$
both easy applications of the Hirzebruch-Riemann-Roch theorem. It will be very important for us that the class $B$ (and therefore also the line bundle $A$, and ultimately the universal sheaf $\CU$) only depends on $r$ and $c_1 \cdot H$, and \textbf{NOT} on $c_2$. \\

\subsection{}
\label{sub:flag}

For the remainder of this Section, we impose Assumption A and will define the moduli space $\bfZ$ of Subsection \ref{sub:locus}. Explicitly, let us fix a quadratic polynomial $P$ (which we will not explicitly mention from now on, but it will be implied that all sheaves denoted by $\CF$ have $P$ as Hilbert polynomial) and consider the functor $\sZ$ which associates to a scheme $T$ the set of quadruples of: \\

\begin{itemize}

\item a map $\phi: T \rightarrow S$ \\

\item an invertible sheaf $\CL$ on $T$ \\

\item a $T$--flat family of stable sheaves $\CF$ on $T \times S$ \\

\item a surjective homomorphism $\CF \stackrel{\psi}\twoheadrightarrow \Gamma_*(\CL)$, where $\Gamma = \text{graph}(\phi): T \hookrightarrow T \times S$ 

\end{itemize}

\tab 
The purpose of the current and next Subsections is to show that the functor $\sZ$ is representable by a scheme that will be denoted by $\bfZ$. Our starting point is the well-known fact (see Section 2.A of \cite{HL}) that there exists a scheme $\flag$ that represents the functor that associates to a scheme $T$ the set of quadruples consisting of: \\

\begin{itemize}

\item a map $\phi: T \rightarrow S$ \\

\item an invertible sheaf $\CL$ on $T$ \\

\item a $T$--flat family of quotients $V_T \otimes \CO_S(-n) \stackrel{\phi}\twoheadrightarrow \CF$ on $T \times S$\\

\item a surjective homomorphism $\CF \stackrel{\psi}\twoheadrightarrow \Gamma_*(\CL)$, where $\Gamma = \text{graph}(\phi): T \hookrightarrow T \times S$ 

\end{itemize}

\tab Explicitly, the functor above is represented by the projectivization of the universal sheaf on $\quot \times S$, and we will denote this scheme in terms of its closed points:
\begin{equation}
\label{eqn:flag scheme}
\flag = \Big\{ V \otimes \CO_S(-n) \stackrel{\phi}\twoheadrightarrow \CF \stackrel{\psi}\twoheadrightarrow \BC_x, \ x \in S \Big\}
\end{equation}
If we write $\CF' = \Ker \psi$ and $V' = \Ker \left(V \xrightarrow{\Gamma(\psi \circ \phi \otimes \CO_S(n))} \BC \right)$, then \eqref{eqn:flag scheme} reads:
\begin{equation}
\label{eqn:flag}
\flag = \left\{   \vcenter{\vbox{ \xymatrix{V \otimes \CO_S(-n) \ar@{->>}[r] & \CF \\
V' \otimes \CO_S(-n) \ar@{->>}[r] \ar@{^{(}->}[u] & \CF' \ar@{^{(}->}[u]}}} \right\}
\end{equation}
In this presentation, it is clear how to define the maps of schemes $\pi : \flag \rightarrow \quot$ 
and $\pi' : \flag \rightarrow \quot'$, where $\quot$ and $\quot'$ are the schemes \eqref{eqn:quot} defined with respect to the Hilbert polynomials $P$ and $P-1$, respectively. We therefore obtain:\begin{equation}
\label{eqn:flag to quot}
\flag \stackrel{\pi \times \pi'}\longrightarrow \quot \times \quot'
\end{equation}
with the following important property. The universal sheaves $\tCU$ and $\tCU'$ on $\quot \times S$ and $\quot' \times S$ are contained inside each other when pulled back to $\flag \times S$:
\begin{equation}
\label{eqn:inclusion}
{\pi'}^*(\tCU') \hookrightarrow \pi^*(\tCU)
\end{equation}
This follows from the construction of the moduli spaces $\flag$ and $\quot, \quot'$ as closed subschemes of a flag variety and Grassmannians, respectively. The universal sheaves on the moduli spaces are assembled from the universal bundles on flag varieties and Grassmannians, and therefore the inclusion \eqref{eqn:inclusion} boils down to the tautological inclusions between universal bundles on the flag variety. \\

\subsection{}

Since we are under Assumption A, all semistable sheaves are stable. Recall the open subscheme $R^s \subset \quot$ consisting of surjections \eqref{eqn:quot} which induce an isomorphism in cohomology, and where $\CF$ is stable. We write ${R'}^s \subset \quot'$ and $Q^s \subset \flag$ for the analogous open subschemes, and make the following observation: \\

\begin{proposition}
\label{prop:easy 1}

A point $z \in \emph{Flag}$ lies in $\pi^{-1}(R^s)$ iff it lies in ${\pi'}^{-1}({R'}^s)$.

\end{proposition}

\tab 
Indeed, this is a straightforward consequence of Proposition \ref{prop:hecke}, and it implies that the map \eqref{eqn:flag to quot} gives rise to a fiber square:
\begin{equation}
\label{eqn:pet}
\xymatrix{
 Q^s \ar@{^{(}->}[r] \ar[d]_{(\pi \times \pi')^s} & \flag  \ar[d]^{\pi \times \pi'} \\
R^s \times {R'}^s \ar@{^{(}->}[r] & \quot \times \quot'}
\end{equation}

\begin{proposition}
\label{prop:easy 2}

The arrow $(\pi \times \pi')^s$ is a trivial $\BC^*$--bundle onto its image. \\

\end{proposition}

\begin{proof} The claim boils down to the statement that if the horizontal arrows are given in diagram \eqref{eqn:flag}, then the vertical arrows are uniquely determined up to constant multiple. Because the vertical map on the left in \eqref{eqn:flag} is $H^0$ applied to the vertical map on the right $\otimes \ \CO_S(n)$, it is enough to prove that:
\begin{equation}
\label{eqn:0 or c}
\CF',\CF \text{ stable} \quad \Rightarrow \quad \Hom(\CF',\CF) = 0 \text{ or }\BC
\end{equation}
where $\CF'$ and $\CF$ have Hilbert polynomials $P-1$ and $P$, respectively. Assumption A implies that any non-zero homomorphism $\CF'\rightarrow \CF$ must be injective, since otherwise the image of such a homomorphism would have the impossible property that its reduced Hilbert polynomial is strictly contained between $P-1$ and $P$. So we assume that there exists an injection $i : \CF' \hookrightarrow \CF$, and we must prove that it is the only one up to constant multiple. Composing $\iota$ with the finite colength injection $j: \CF \hookrightarrow \CE := \CF^{\vee\vee}$, it is enough to show that $\Hom(\CF', \CE) = \BC$. But from the long exact sequence associated to $j \circ i$, we obtain:
$$
0 \longrightarrow \Hom(\CE, \CE) \longrightarrow \Hom(\CF', \CE) \longrightarrow \Ext^1(Q, \CE) \longrightarrow...
$$
where $Q = \CE/\CF'$ is a finite length sheaf. Since the double dual $\CE$ is stable (as follows from Proposition \ref{prop:hecke}), the space on the left is $\BC$, and since the double dual is locally free, the space on the right is 0. To elaborate the last claim, take a Jordan-Holder filtration of $Q$, so it is enough to prove that $\Ext^1(\BC_x, \CE) = 0$ for any closed point $x\in S$. Since $\CE$ is locally free, this is equivalent to the fact that there are no non-trivial extensions between the residue field and a free module over a regular local ring of dimension $\geq 2$.

\end{proof}

\noindent Consider the action of $GL(V) \times GL(V')$ on $\flag$ given by acting on the vector spaces $V$ and $V'$, and note that the vertical arrows in \eqref{eqn:pet} are equivariant with respect to the action (implicit in this is the fact that $Q^s$ and its closure $Q$ are preserved by the action). Consider the line bundle $L_m$ on $\quot$ defined in \eqref{eqn:line bundle} and the analogous $L_m'$ on $\quot'$. Then Lemma \ref{lem:fundamental} and Proposition \ref{prop:easy 1} imply that: \\

\begin{proposition}
\label{prop:easy 3}

The open set $Q^s$ is the locus of stable points of the action $GL(V) \times GL(V') \curvearrowright Q$, with respect to the line bundle $L_m \boxtimes L_m'$.

\end{proposition}

\tab 
Therefore, there exists a geometric quotient: 
\begin{equation}
\label{eqn:def z}
\bfZ := Q^s/GL(V) \times GL(V') \hookrightarrow R^s/GL(V) \times {R'}^s/GL(V') = \CM \times \CM'
\end{equation}
and to ensure that it has the desired properties, we must prove the following facts: \\

\begin{proposition}
\label{prop:easy 4}

The injective map of universal sheaves \eqref{eqn:inclusion} on $Q^s \times S$ descends to an injective map of sheaves \eqref{eqn:incl} on $\bfZ \times S$. \\

\end{proposition}

\begin{proof} Indeed, recall that the universal sheaf $\CU$ on $\CM \times S$ was obtained from $\tCU$ on $\quot \times S$ by descending the coherent sheaf \eqref{eqn:real universal}. Since the line bundle $A$ is defined by \eqref{eqn:def a} with $B$ being a fixed linear combination of $\CO_{\pt}$ and $\CO_H$, we have:
$$
\pi^*(A) \cong {\pi'}^*(A') := \alpha \in \text{Pic}(Q^s)
$$
Letting $p_1 : Q^s \times S \rightarrow Q^s$ denote the projection, the injection $\tCU \hookrightarrow \tCU'$ yields:
\begin{equation}
\label{eqn:inj}
\pi^*(\tCU) \otimes p_1^*(\alpha^{-1}) \hookrightarrow {\pi'}^*(\tCU') \otimes p_1^*(\alpha^{-1}) \in \text{Coh} ( Q^s \times S )
\end{equation}
Descending \eqref{eqn:inj} to the $GL(V) \times GL(V')$ quotient gives the map \eqref{eqn:incl} on $\bfZ \times S$.

\end{proof}

\begin{proposition}
\label{prop:easy 5}

The geometric quotient $\bfZ$ of \eqref{eqn:def z} represents the functor $\sZ$. \\

\end{proposition}

\begin{proof} The proof closely follows those of Lemma 4.3.1 and Proposition 4.6.2 of \cite{HL}, so we will just sketch the main ideas. An element in $\sZ(T)$ consists of the datum in the first 4 bullets in Subsection \ref{sub:flag}, which essentially boils down to an injective map $\CF'_T \hookrightarrow \CF_T$ of flat families of stable sheaves on $T \times S$, which has colength 1 above any closed point of $T$. For $n$ large enough (the ability to choose such an $n$ follows from the boundedness of the family of stable sheaves) we may use the standard projections:
$$
p_1 : T \times S \rightarrow T \qquad \text{and} \qquad p_2 : T \times S \rightarrow S
$$
to define the locally free sheaves $V'_T = p_{1*}(\CF' \otimes p_2^*\CO(n)) \hookrightarrow V_T = p_{1*}(\CF \otimes p_2^*\CO(n))$ on $T$. Consider the principal $GL(V) \times GL(V')$--bundle $\text{Frame} \rightarrow T$ which parametrizes all trivializations of the locally free sheaves $V'_T$ and $V_T$. The universal property of the scheme $\flag$ that represents the data in the last 4 bullets in Subsection \ref{sub:flag} implies the existence of a classifying homomorphism:
$$
\eta : \text{Frame} \longrightarrow \flag
$$
which actually takes values in $Q^s$, due to the fact that the families $\CF'_T$ and $\CF_T$ were stable to begin with. The homomorphism $\eta$ is $GL(V) \times GL(V')$--equivariant, and therefore descends to a map $T \longrightarrow \bfZ$ since the quotient \eqref{eqn:def z} is geometric. Finally, the fact that the pull-back of the universal sheaves under $\eta$ gives the inclusion $\CF' \hookrightarrow \CF$ can be descended to the level of the $GL(V) \times GL(V')$--quotient, thus proving the fact that $\bfZ$ represents the functor $\sZ$. 

\end{proof}

\subsection{}
\label{sub:push}

We now turn to certain computations in algebraic $K$--theory. Hereafter, $X$ will denote the dg scheme cut out by the Koszul complex of a section of a locally free sheaf on a Noetherian scheme. All coherent sheaves considered will be on $X$. \\

\begin{proposition}
\label{prop:general proj}

Let $U$ be a coherent sheaf of projective dimension 1, and recall:
$$
\pi:\BP_X (U) \rightarrow X \qquad \text{and} \qquad \pi' : \BP_X(U^\vee[1]) \rightarrow X
$$
the dg schemes considered in Subsection \ref{sub:derived bundles}. Then we have:
\begin{equation}
\label{eqn:push taut gen 1}
\pi_* \left[ \delta \left( \frac {\CO(1)}z \right) \right] = \wedge^\bullet\left( - \frac {U} z \right) \b
\end{equation}
\begin{equation}
\label{eqn:push taut gen 2}
\pi'_* \left[ \delta \left( \frac {\CO(-1)}z \right) \right] = \ \wedge^\bullet\left(\frac z{U} \right) \Big|_{0 - \infty}
\end{equation}
where the right-hand side is defined as in \eqref{eqn:expand}. Recall that $\frac 1{U}$ denotes $U^\vee$. \\ 

\end{proposition}

\begin{proof} When $U$ is locally free of rank $n$, Exericise III.8.4 of \cite{H} establishes:
$$
\pi_* [\CO(i)] = [R\pi_* \CO(i)] = \begin{cases} S^i U & \text{ if } i \geq 0 \\
0 & \text{ if } - n < i < 0  \\
(-1)^{n-1} S^{-i-n} U^\vee \otimes \left(\det U\right)^{-1} & \text{ if } i \leq -n
\end{cases} 
$$
Summing the above expression over all $i \in \BZ$ establishes \eqref{eqn:push taut gen 1}, since:
$$
\wedge^\bullet\left( - \frac Uz \right) = \sum_{i=0}^\infty \frac {[S^i U]}{z^i} \qquad \text{near } z = \infty
$$
$$
\wedge^\bullet\left( - \frac Uz \right) = (-1)^n \left( \det U \right)^{-1} \sum_{i=0}^\infty \frac {[S^i U^\vee]}{z^{-n-i}} \qquad \text{near } z = 0
$$
This proves the Proposition when $U$ is locally free. More generally, let us take $U \cong V/W$ with $V$ and $W$ locally free. As in \eqref{eqn:derived proj 1}, $\BP_X (U)$ is the dg subscheme of $\BP_X (V)$ defined as the virtual zero locus of the composed map:
$$
\rho^*(W) \hookrightarrow \rho^*(V) \stackrel{\text{taut}}\twoheadrightarrow \CO_{\BP_X(V)}(1)
$$
where the projection maps are as in the following diagram:
\begin{equation}
\label{eqn:diagram proj}
\xymatrix{
& \BP_X(U) \ar@{^{(}->}[r]^{\iota} \ar[rd]_\pi & \BP_X(V) \ar[d]^\rho \\
& & X \\}
\end{equation}
The corresponding push-forward map $\iota_*$ is defined in $K$--theory by replacing the structure sheaf of the dg subscheme $\BP_X(U) \hookrightarrow \BP_X(V)$ with the exterior algebra of the locally free sheaf which ``cuts it out", i.e. $\rho^*(W) \otimes \CO_{\BP_X(V)}(-1)$. Therefore:
$$
\pi_* \left[ \delta \left( \frac {\CO(1)}z \right) \right]  = \rho_* \left[ \delta \left( \frac {\CO_{\BP_X(V)}(1)}z \right) \wedge^\bullet \left( \frac {\rho^*(W)}{\CO_{\BP_X(V)}(1)} \right) \right] =
$$
$$
= \rho_* \left[ \delta \left( \frac {\CO(1)_{\BP_X(V)}}z \right) \cdot \wedge^\bullet \left(\frac {\rho^*(W)}z\right) \right] = \frac {\wedge^\bullet\left(\frac {W} z \right)}{\wedge^\bullet\left(\frac {V} z \right)} \Big|_{z = \infty} - \frac {\wedge^\bullet\left(\frac {W} z \right)}{\wedge^\bullet\left(\frac {V} z \right)} \Big|_{z = 0}
$$
where in the second equality we invoke the fundamental property \eqref{eqn:fundamental property} of $\delta$ functions, and in the third equality we invoke Proposition \ref{prop:general proj} for the locally free sheaf $V$, which we proved. Using the fact that $[U] = [V] - [W]$ and \eqref{eqn:multiplicative}, we conclude \eqref{eqn:push taut gen 1}. Formula \eqref{eqn:push taut gen 2} is proved similarly, so we leave it to the interested reader. 
 
\end{proof}

\subsection{}
\label{sub:blow}

For a locally free sheaf $U$ on a scheme $X$, recall that the projectivization $\BP_X(U) \rightarrow X$ represents the functor that associates to a scheme $T$ the set of triples: \\

\begin{itemize}

\item a morphism $\phi : T \rightarrow X$ \\

\item a line bundle on $T$ which we will denote by $\CO(1)$ \\

\item a surjective morphism $\phi^*(U) \stackrel{f}\twoheadrightarrow \CO(1)$

\end{itemize}

\tab 
The fact that the above functor is representable means that the line bundle $\CO(1)$ is the pull-back of the tautological line bundle on $\BP_X(U)$ via the morphism defined by the above 3 bullets. Therefore, we abuse notation, and henceforth write $\CO(1)$ for the tautological line bundle on $\BP_X(U)$ itself. To a short exact sequence of locally free sheaves $0 \rightarrow U' \rightarrow U \rightarrow Q \rightarrow 0$ on $X$, we may associate the following diagram:
\begin{equation}
\label{eqn:nice diagram}
\xymatrix{& Y \ar[ld]_{p'} \ar[rd]^p & & \\
\BP_X(U') \ar[rrd]_{\pi'} & & \BP_X(U) \ar[d]^\pi & \BP_X (Q) \ar@{_{(}->}[l]_{\iota} \ar[ld] \\
& & X &}
\end{equation}
where $Y$ is the closed subscheme of $\BP_X(U') \times_X \BP_X(U)$ whose $T$--points are morphisms $s:\CO(1)' \rightarrow \CO(1)$ which make the following diagram commute:
\begin{equation}
\label{eqn:nice square}
\xymatrix{\phi^*(U) \ar@{->>}[r]^-{f} & \CO(1) \\
\phi^*(U') \ar@{->>}[r]^{f'} \ar@{^{(}->}[u]  & \CO(1)' \ar[u]^s}
\end{equation}
We abuse notation by also referring to the tautological line bundle on $\BP_X(U')$ as $\CO(1)'$. Clearly, the maps $p$ and $p'$ in \eqref{eqn:nice diagram} are given by just remembering the top and bottom rows in \eqref{eqn:nice square}, respectively. Since $Q$ is assumed to be locally free, $\iota$ is a regular embedding. \\

\begin{proposition}
\label{prop:blow up}

The map $p'$ in the diagram \eqref{eqn:nice diagram} can be descibed as:
\begin{equation}
\label{eqn:proj proj}
Y \cong \BP_{\BP_X(U')} \left( E \right)
\end{equation}
where $E$ is the coherent sheaf on $\BP_X(U')$ obtained as the image of the tautological morphism ${\pi'}^*(U') \rightarrow \CO(1)'$ under the connecting homomorphism:
$$
\emph{Hom} \left( {\pi'}^*(U'), \CO(1)' \right) \longrightarrow \emph{Ext}^1\left( {\pi'}^*(Q), \CO(1)' \right)
$$
induced by the short exact sequence $0 \rightarrow U' \rightarrow U \rightarrow Q \rightarrow 0$. \\

\end{proposition}

\begin{proof} By definition, a map $T \rightarrow \BP_{\BP_X(U')} \left( E \right)$ amounts to a quadruple consisting of: \\

\begin{itemize} 

\item a morphism $\phi : T \rightarrow X$ \\

\item line bundles on $T$ which we will suggestively denote by $\CO(1)$ and $\CO(1)'$ \\

\item a surjective homomorphism $\phi^* (U') \stackrel{f'}\twoheadrightarrow \CO(1)'$ \\

\item a surjective homomorphism $\widetilde{\phi}^*(E) \twoheadrightarrow \CO(1)$

\end{itemize}

\tab 
where $\widetilde{\phi} : T \rightarrow \BP_X(U')$ is the morphism defined by the first three bullets. The extension $E$ is explicitly given by the middle space in the short exact sequence:
$$
0 \longrightarrow \CO(1)' \longrightarrow \frac {\phi^*(U) \oplus \CO(1)'}{\phi^*(U')} \longrightarrow \phi^*(Q) \longrightarrow 0
$$
The middle space is defined with the diagonal quotient by the inclusion $U' \hookrightarrow U$ and the map $f'$. Therefore, the datum of the fourth bullet above amounts to:
$$
\text{homomorphisms} \qquad \phi^*(U) \stackrel{f}\longrightarrow \CO(1) \qquad \text{and} \qquad \CO(1)' \stackrel{s}\longrightarrow \CO(1)
$$
which agree on $\phi^*(U')$. This is precisely the same as the top and right maps in the diagram \eqref{eqn:nice square}, which establishes the fact that $\BP_{\BP_X(U')} \left( E \right) = Y$.

\end{proof}

\subsection{} We will apply Proposition \ref{prop:blow up} in order to prove Proposition \ref{prop:comm rel 1}. However, we note that the setup therein involves a short exact sequence:
$$
0 \rightarrow U' \rightarrow U \rightarrow Q \rightarrow 0
$$
where $U'$ and $U$ are not locally free, but coherent sheaves of projective dimension 1. Therefore, let us write $U = V/W$ and $U' = V'/W'$ with $V,W,V',W'$ locally free on $X$, which are endowed with a commutative diagram of maps:
$$
\xymatrix{V' \ar[r]  & V \\
W' \ar@{^{(}->}[u] \ar[r] & W \ar@{^{(}->}[u]}
$$
that induces the injection $U' \rightarrow U$ (the commutativity of the square above follows from the explicit construction of $V,W,V',W'$ in Proposition \ref{prop:length 1}). Let us indicate the modifications necessary to make Proposition \ref{prop:blow up} apply to this more general setup. First of all, according to the principle laid out in Subsection \ref{sub:derived bundles}, $\BP_X(U)$ is a dg scheme over $X$. Therefore, it represents the functor which associates to a dg scheme $T$ the set of triples: \\

\begin{itemize}

\item a morphism $\phi : T \rightarrow X$ \\

\item a line bundle on $T$ which we will denote by $(\CO(1)_\bullet, d)$ (the $\bullet$ denotes the homological grading on coherent sheaves on $T$, which are graded $\CO_T$--modules) \\

\item a commutative diagram of morphisms:
$$
\xymatrix{\phi^*(V) \ar@{->>}[r]^-{f_0} & \CO(1)_0 \\
\phi^*(W) \ar@{^{(}->}[u] \ar[r]^-{f_{-1}} & \CO(1)_{-1} \ar[u]_-d}
$$

\end{itemize}

\noindent Similarly, diagram \eqref{eqn:nice diagram} should be replaced by:
\begin{equation}
\label{eqn:nice nice diagram}
\xymatrix{& Y \ar[ld]_{p'} \ar[rd]^p & \\
\BP_X(U') \ar[rd]_{\pi'} & & \BP_X(U) \ar[ld]^\pi \\
& X &}
\end{equation}
where $Y$ represents the functor which sends a dg scheme $T$ to the set of triples: \\

\begin{itemize}

\item a morphism $\phi : T \rightarrow X$ \\

\item line bundles $(\CO(1)_\bullet, d)$ and $(\CO(1)'_\bullet, d')$ on $T$ \\

\item a commutative diagram of morphisms:
$$
\xymatrix{& \phi^*(V) \ar@{->>}[rr]^-{f_0}  & & \CO(1)_0 \\
\phi^*(V') \ar[ru] \ar@{->>}[rr]^{\qquad \qquad f'_0}  & & \CO(1)'_0 \ar[ru]_-{s_0} & \\
& \phi^*(W) \ar@{^{(}->}[uu] \ar[rr]^-{\qquad f_{-1}} & & \CO(1)_{-1} \ar[uu] \\
\phi^*(W') \ar[ru] \ar@{^{(}->}[uu] \ar[rr]^-{f'_{-1}} & & \CO(1)'_{-1} \ar[uu] \ar[ru]_-{s_{-1}} &}
$$

\end{itemize}

\noindent The map $p':Y \rightarrow \BP_X(U')$ is given by remembering only the front square of the cube above. The analogue of Proposition \ref{prop:blow up} states that the map $p'$ can be written:
\begin{equation}
\label{eqn:proj proj dg}
Y \cong \BP_{\BP_X(U')} \left( E \right)
\end{equation}
where $E$ is the two-step complex in the middle of the diagram below:
$$
\xymatrix{0 \ar[r] & \CO(1)'_0 \ar[r] & \displaystyle \frac {{\pi'}^*(V) \oplus \CO(1)'_0}{{\pi'}^*(V')} \ar[r] & {\pi'}^*(V/V') \ar[r] & 0 \\
0 \ar[r] & \CO(1)'_{-1} \ar[u]^{d'} \ar[r] & \displaystyle \frac {{\pi'}^*(W) \oplus \CO(1)'_{-1}}{{\pi'}^*(W')} \ar[u] \ar[r] & {\pi'}^*(W/W') \ar[u] \ar[r] & 0}
$$
We leave the proof of \eqref{eqn:proj proj dg}, which closely follows that of Proposition \ref{prop:blow up}, as an exercise to the interested reader. \\

\subsection{}

We will prove a generalization of \eqref{eqn:diag codim 2}, concerning the symmetric powers of the structure sheaf of a regular subvariety. Note that the same proof works for arbitrary codimension, but the right-hand side of \eqref{eqn:codim 2} will be more complicated: \\

\begin{proposition}
\label{prop:codim 2}

If $\iota : X \hookrightarrow Y$ is a codimension 2 regular embedding, then:
\begin{equation}
\label{eqn:codim 2}
\wedge^\bullet( - z \cdot \CO_X) = 1 + \iota_* \left[ \frac {z}{(1-z)(1-z [\det \CN^\vee])} \right] \in K_Y(z)
\end{equation}
where $\CN$ denotes the normal bundle of $X$ in $Y$. \\

\end{proposition} 

\begin{proof} Let us first prove the Proposition when $X$ is the zero subscheme of a section $\CO_Y \rightarrow \CV^\vee$ for some rank 2 locally free sheaf $\CV$ on $Y$. If we write $v_1+v_2 = [\CV]$ for the Chern roots of this locally free sheaf, then we have:
$$
[\CO_X] = 1 - v_1 - v_2 + v_1 v_2  \quad \Rightarrow \quad \wedge^\bullet(-z \cdot \CO_X) = \frac {(1 - z v_1)(1 - z v_2)}{(1-z)(1-z v_1 v_2)} = 
$$
$$
= 1 + \frac {(1-v_1)(1-v_2) z}{(1-z)(1-zv_1 v_2)} = 1 + \iota_* \left[ \frac {z}{(1-z)(1-z [\det \CN^\vee])} \right]
$$
where in the last equality we used the fact that $\iota_*1 = [\CO_X] = (1-v_1)(1-v_2)$ and $\CV|_X = \CN^\vee$. Now let us assume that $X \hookrightarrow Y$ is a regular embedding, and let us use deformation to the normal bundle. This entails constructing the variety:
\begin{equation}
\label{eqn:def normal}
\widetilde{Y} = \text{Bl}_{X \times 0} (Y \times \BA^1)
\end{equation}
The projection map $\widetilde{Y} \rightarrow \BA^1$ is flat, and its fibers:
$$
\xymatrix{
Y_0 \ar[d] \ar@{^{(}->}[r] & {\widetilde{Y}} \ar[d] & Y_1 \ar[d] \ar@{_{(}->}[l] \\
0 \ar@{^{(}->}[r] & {\mathbb{A}}^1  & \ar@{_{(}->}[l] 1}
$$
are given by $Y_0 = \text{Tot}_X (\CN) \sqcup \text{Bl}_X Y$ and $Y_1 \cong Y$. Moreover, there exists a subvariety: 
\begin{equation}
\label{eqn:subvariety}
X \times \BA^1 \cong \widetilde{X} \stackrel{\widetilde{\iota}}\hookrightarrow \widetilde{Y}
\end{equation}
whose restrictions $X_0, X_1$ to $0, 1 \in \BA^1$ are $X \hookrightarrow \text{Tot}_X \CN$ and $X \hookrightarrow Y$, respectively. Let us denote the difference between the left and right-hand sides of \eqref{eqn:codim 2} by $\Gamma(X,Y)$, the quantity which we want to prove equals 0. Then we must show that:
\begin{equation}
\label{eqn:equivalences}
\Gamma(X_0, Y_0) = 0 \quad \Rightarrow \quad \Gamma(\widetilde{X}, \widetilde{Y}) = 0 \quad \Rightarrow \quad \Gamma(X_1, Y_1) = 0
\end{equation}
since the vanishing of $\Gamma(X_0,Y_0)$ is accounted for by the first part of the proof. Note that $\Gamma(\widetilde{X}, \widetilde{Y})|_{p} = \Gamma(X_p, Y_p)$ for each $p \in \{0,1\}$, which is a consequence of:
$$
\widetilde{Y} \rightarrow \BA^1 \text{ flat} \qquad \Rightarrow \qquad \widetilde{\iota}_*(\gamma)|_p = \iota_*(\gamma|_p) \qquad \forall p\in \{0,1\}, \forall \gamma \in K(\widetilde{X})
$$
This yields the second implication in \eqref{eqn:equivalences}. \\

\noindent As for the first implication, consider the action $\BC^* \curvearrowright \widetilde{Y}$ induced by the standard action $\BC^* \curvearrowright \BA^1$. The fixed point locus of this action is given by:
$$
\widetilde{Y}^{\BC^*} = X \sqcup \text{Bl}_X Y \hookrightarrow \text{Tot}_X (\CN) \sqcup \text{Bl}_X Y = Y_0
$$
Let us consider the following commutative diagram of ordinary and $\BC^*$--equivariant $K$--theory groups:
\begin{equation}
\label{eqn:diagram normal}
\xymatrix{
K_{\BC^*}(\widetilde{X}) \ar[d]_-{\text{for}} \ar@{^{(}->}[r]^-{\widetilde{\iota}'_*} &  K_{\BC^*}(\widetilde{Y}) \ar[d]_-{\text{for}} \ar@{^{(}->}[r]^-{\text{rest}'} & K_{\BC^*} (\widetilde{Y}^{\BC^*} ) \ar[d]^-{\text{for}'} \ar@{=}[r] & K(\widetilde{Y}^{\BC^*} ) \otimes_{\BZ} \BZ[t^{\pm 1}] \ar[d]^-{t \mapsto 1} \\
K(\widetilde{X}) \ar[r]^-{\widetilde{\iota}_*} & K(\widetilde{Y}) \ar[r]^-{\text{rest}} & K (\widetilde{Y}^{\BC^*} ) \ar@{=}[r] & K(\widetilde{Y}^{\BC^*} ) \otimes_{\BZ} \BZ}
\end{equation}
The maps labeled rest and rest$'$ are the natural restriction maps, and the one on top is injective due to Theorem 2 of \cite{VV}. The maps labeled for and for$'$ are the forgetful maps from $\BC^*$--equivariant to ordinary $K$--theory. We will use the notation in \eqref{eqn:diagram normal} from now on. The first implication of \eqref{eqn:equivalences} follows from: \\

\begin{claim}
 
If $c$ lies in the image of $\widetilde{\iota}_* : K(\widetilde{X}) \rightarrow K(\widetilde{Y})$, then:
$$
\emph{rest}(c) = 0 \quad \Rightarrow \quad c = 0
$$
Indeed, taking $c = \Gamma(\widetilde{X}, \widetilde{Y})$ yields the first implication of \eqref{eqn:equivalences}. \\
 
\end{claim}

\noindent To prove the claim, write $c = \widetilde{\iota}_*(\gamma)$ for some $\gamma \in K(\widetilde{X})$. Because $\widetilde{X} = X \times \BA^1$:
$$
K_{\BC^*}(\widetilde{X}) \cong K(X) \otimes_{\BZ} K_{\BC^*}(\BA^1) \cong K(X) \otimes_{\BZ} \BZ[t^{\pm 1}]  \cong K(\widetilde{X}) \otimes_{\BZ} \BZ[t^{\pm 1}]
$$
and therefore the class $\gamma$ can be lifted to $K_{\BC^*}(\widetilde{X})$. This implies that:
$$
c = \text{for}(\widetilde{c}) \qquad \text{for some} \qquad \widetilde{c} \in K_{\BC^*}(\widetilde{Y})
$$ 
and from the construction we may take $\widetilde{c}$ such that $\text{rest}'(\widetilde{c})$ is a $K$--theory class times $t^0$. Therefore, the injectivity of $\text{rest}'$ gives the first implication in the chain:
$$
\text{rest}(c) = \text{rest} \circ \text{for}(\widetilde{c}) = \text{for}' \circ \text{rest}'(\widetilde{c}) = 0 \quad \Rightarrow \quad \widetilde{c} = 0 \quad \Rightarrow \quad c = 0
$$

\end{proof}


\begin{thebibliography}{XXX}

\bibitem{Ba} Baranovsky V., {\em Moduli of sheaves on surfaces and action of the oscillator algebra}, J. Diff. Geom. 55 (2000), no. 2, 

\bibitem{BS} Burban I., Schiffmann O., {\em On the Hall algebra of an elliptic curve I}, Duke Math. J. 161 (2012), no. 7, 1171--1231

\bibitem{C} Caldararu A., {\em Derived categories of twisted sheaves on Calabi-Yau manifolds}, Ph.D. Thesis, Cornell University (2000)

\bibitem{CK} Ciocan-Fontanine I., Kapranov M., {\em Derived Quot schemes},  Ann. Sci. \'Ec. Norm. Sup\'er., S\'erie 4, Volume 34 (2001) no. 3, p. 403-440 

\bibitem{CG} Chriss N., Ginzburg V., {\em Representation theory and complex geometry}, Birkhauser 1997

\bibitem{DI} Ding J., Iohara K., {\em Generalization of Drinfeld quantum affine algebras}, Lett. Math. Phys. 41 (1997), no. 2, 181-–193

\bibitem{FO} Feigin B., Odesskii A., {\em Vector bundles on elliptic curve and Sklyanin algebras}, Topics in Quantum Groups and Finite-Type Invariants, Amer. Math. Soc. Transl. Ser. 2, 185 (1998), Amer. Math. Soc., 65-–84

\bibitem{FHHSY} Feigin B., Hashizume K., Hoshino A., Shiraishi J., Yanagida S., {\em A commutative algebra on degenerate $\BC \BP^1$ and MacDonald polynomials},  J. Math. Phys. 50 (2009), no. 9

\bibitem{FJMM} Feigin B., Jimbo M., Miwa T., Mukhin E., {\em Quantum toroidal} $\fgl_1$ {\em and the Bethe ansatz}, J. Phys. A 48 (2015)

\bibitem{FT} Feigin B., Tsymbaliuk A., {\em Heisenberg action in the equivariant $K-$theory of Hilbert schemes via Shuffle Algebra}, Kyoto J. Math. 51 (2011), no. 4

\bibitem{GT} Gholampour A., Thomas R., {\em Degeneracy loci, virtual cycles and nested Hilbert schemes}, at$\chi$iv:1709.06105

\bibitem{G} Grojnowski I., {\em Instantons and affine algebras I. The Hilbert scheme and vertex operators}, Math. Res. Lett. 3 (1996), no. 2

\bibitem{H} Hartshorne R., {\em Algebraic geometry}, Graduate texts in mathematics 52, ISBN 0-387-90244-9

\bibitem{HL} Huybrechts D., Lehn M., {\em The geometry of moduli spaces of sheaves}, 2nd edition, Cambridge University Press 2010, ISBN 978-0-521-13420-0

\bibitem{LePot} Le Potier J., {\em Lectures on vector bundles}, Cambridge studies in advanced mathematics 54 (1997), ISBN 0-521-48182-1

\bibitem{Markman} Markman E., {\em Integral generators for the cohomology ring of moduli spaces of sheaves over Poisson surfaces}, Adv. Math. 208 (2007), no 2, Pages 622--646

\bibitem{MO} Maulik D., Okounkov A., {\em Quantum groups and quantum cohomology}, ar$\chi$iv:1211.1287

\bibitem{M} Miki K., {\em A $(q, \gamma)$ analog of the $W_{1+\infty}$ algebra}, J. Math. Phys., 48 (2007), no. 12

\bibitem{Nak} Nakajima H., {\em Heisenberg algebra and Hilbert schemes of points on projective surfaces}, Ann. Math. 145, No. 2 (Mar 1997), 379--388

\bibitem{Nak2} Nakajima H., {\em Lectures on Hilbert schemes of points on surfaces}, University Lecture Series vol 18, 1999, 132 pp

\bibitem{W} Negu\cb t A., {\em The $q$-AGT-W relations via shuffle algebras}, ar$\chi$iv:1608.08613

\bibitem{Shuf} Negu\cb t A., {\em The shuffle algebra revisited},  Int. Math. Res. Not. 22 (2014), 6242--6275

\bibitem{Mod} Negu\cb t A., {\em Moduli of flags of sheaves and their $K$--theory}, Algebraic Geometry 2 (2015), 19--43

\bibitem{OSS} Okonek C., Schneider M., Spindler H., {\em Vector bundles on complex projective spaces}, Birkh\"auser Verlag, ISBN 978-3-0348-0150-8


\bibitem{S} Schiffmann O., {\em Drinfeld realization of the elliptic Hall algebra}, Journal of Algebraic Combinatorics vol 35 (2012), no 2, pp 237-–262 

\bibitem{SV} Schiffmann O., Vasserot E., {\em The elliptic Hall algebra and the equivariant $K-$theory of the Hilbert scheme of ${\mathbb{A}}^2$}, Duke Math. J. 162 (2013), no. 2, 279--366

\bibitem{SV2} Schiffmann O., Vasserot E., {\em Cherednik algebras, $W$--algebras and the equivariant cohomology of the moduli space of instantons on $\BA^2$}, Publ. Math. Inst. Hautes Etud. Sci., 118 (2013), Issue 1, 213-–342

\bibitem{VV} Vezzosi G., Vistoli A., {\em Higher algebraic K-theory for actions of diagonalizable groups}, Inv. Math. 153 (2003), Issue 1, 1-–44

\end{thebibliography}
\end{document}